\documentclass[a4paper, 12pt]{article} 

\usepackage[utf8]{inputenc}
\usepackage{lmodern}
\usepackage[top=25mm,bottom=30mm,left=25mm,right=25mm]{geometry}
\usepackage{amssymb, amsfonts, amsthm, amsmath, mathtools, mathrsfs}
\usepackage{stmaryrd,bbm}
\usepackage[dvipsnames]{xcolor}
\usepackage{tikz,caption,graphicx,subfigure}
\usepackage[english]{babel}
\usepackage[colorlinks]{hyperref}

\theoremstyle{plain}
\newtheorem{theorem}{Theorem}[section]
\newtheorem{corollary}[theorem]{Corollary}
\newtheorem{lemma}[theorem]{Lemma}
\newtheorem{proposition}[theorem]{Proposition}

\newtheorem{conjecture}[theorem]{Conjecture}
\newtheorem{question}[theorem]{Question}

\theoremstyle{definition}
\newtheorem{definition}[theorem]{Definition}

\theoremstyle{remark}
\newtheorem{remark}[theorem]{Remark}

\let\originalleft\left
\let\originalright\right
\renewcommand{\left}{\mathopen{}\mathclose\bgroup\originalleft}
\renewcommand{\right}{\aftergroup\egroup\originalright}

\newcommand{\N}{\mathbb{N}}
\newcommand{\Z}{\mathbb{Z}}
\newcommand{\R}{\mathbb{R}}

\newcommand{\T}{\mathbb{T}}
\newcommand{\1}{\mathbbm{1}}

\newcommand{\ind}[1]{\1_{\left\{#1\right\}}}

\newcommand{\floor}[1]{{\left\lfloor #1 \right\rfloor}}

\newcommand{\abs}[1]{\left\lvert #1 \right\rvert}

\numberwithin{equation}{section}

\DeclareMathOperator{\E}{\mathbb{E}}

\renewcommand{\P}{\mathbb{P}}

\newcommand{\calN}{\mathcal{N}}

\newcommand{\diff}{\mathrm{d}}
\newcommand{\e}{\mathrm{e}}

\renewcommand{\bar}[1]{\overline{#1}}
\newcommand{\egaldistr}{\overset{(d)}{=}}
\renewcommand{\tilde}[1]{\widetilde{#1}}

\renewcommand{\rho}{\varrho}
\renewcommand{\epsilon}{\varepsilon}

\newcommand\relphantom[1]{\mathrel{\phantom{#1}}}

\title{An exactly solvable continuous-time Derrida--Retaux~model}
\author{Yueyun Hu\footnote{LAGA, UMR 7539, CNRS Université Paris 13 - Sorbonne Paris Cité, Université Paris 8, 99 avenue Jean-Baptiste Cl\'ement, 93430 Villetaneuse, France.\newline {\tt yueyun@math.univ-paris13.fr} partially supported by ANR MALIN and ANR SWIWS} \and Bastien Mallein\footnote{LAGA, UMR 7539, Université Paris 13 - Sorbonne Paris Cité, Université Paris 8, 99 avenue Jean-Baptiste Cl\'ement, 93430 Villetaneuse, France. \newline {\tt mallein@math.univ-paris13.fr} partially supported by ANR MALIN} \and Michel Pain\footnote{DMA, \'Ecole Normale Sup\'erieure, PSL, CNRS, 45 rue d'Ulm, 75005 Paris, France \& LPSM, Sorbonne Universit\'e, Sorbonne Paris Cit\'e, CNRS, 4 place Jussieu, 75005 Paris, France.\newline {\tt michel.pain@ens.fr} partially supported by ANR MALIN}}

\begin{document}

\maketitle

\begin{abstract}
To study the depinning transition in the limit of strong disorder, Derrida and Retaux \cite{DR14} introduced a discrete-time max-type recursive model. It is believed that for a large class of recursive models, including Derrida and Retaux' model, there is a highly non-trivial phase transition. In this article, we present a continuous-time version of Derrida and Retaux model, built on a Yule tree, which yields an exactly solvable model belonging to this universality class. The integrability of this model allows us to study in details the phase transition near criticality and can be used to confirm the infinite order phase transition predicted by physicists. We also study the scaling limit of this model at criticality, which we believe to be universal.
\end{abstract}

\section{Introduction}
 
The problem of the depinning transition has attracted much attention among mathematicians and physicists over the last thirty years, see Giacomin \cite{G07, G11} and the references therein. To study the depinning transition in the limit of strong disorder, Derrida and Retaux introduced in \cite{DR14} a max-type recursive model, that can be defined, up to a simple change of variables, as follows:
\begin{equation}
  \label{eqn:derridaRetaux}
   X_{n+1} \egaldistr \left( X_n + \tilde{X}_n -1 \right)_+, \qquad \forall\,  n \geq 0, 
\end{equation}
where, for any $x \in \mathbb R$, $x_+ \coloneqq \max(x, 0)$, $\tilde{X}_n$ denotes an independent copy of $X_n$, and $\egaldistr$ stands for the identity in distribution. Note that this model was previously studied by Collet, Eckmann, Glaser and Martin \cite{CEGM84} for random variables taking integer values. Through studying the sequence of probability-generating functions, they identified the critical manifold of $(X_n)$, provided that $X_0 \in \N$ a.s. With renormalization group arguments, Derrida and Retaux studied this max-type recursive equation for real-valued random variables.

This recursive model can be viewed as a simplified version of the recursion equation obtained by taking the logarithm of the partition function $Z_n$ of the so-called hierarchical model, which satisfies
\begin{equation}
  \label{eqn:depinning}
  Z_{n+1} \egaldistr \frac{Z_n \tilde{Z}_n + b-1}{b}, \qquad \forall \, n \geq 0, 
\end{equation}
where $b>1$  is a fixed constant and $\tilde{Z}_n$ is an independent copy of $Z_n$. 
In other words, Equation \eqref{eqn:derridaRetaux} is a tropicalization of Equation \eqref{eqn:depinning}.
The hierarchical model is a pinning model on the diamond lattices, introduced by Derrida, Hakim and Vannimenus \cite{dhv92} and studied in depth by a series of recent works \cite{MoG08, DGLT09, glt2010, Lac10, BeT13}.
 
Despite its simplicity in the definition, the Derrida--Retaux model (which we often abbreviate in the rest of the article as DR model) turns out to exhibit complex behaviours at criticality, which are difficult to prove rigorously, and many fundamental questions remain open. It is believed that for a large class of recursive models, including \eqref{eqn:derridaRetaux} and \eqref{eqn:depinning}, there is an infinite order phase transition from the pinned to the unpinned regime. In this article, we present an exactly solvable version of a continuous-time generalization of the DR model. The integrability of the model allows us to study in details the phase transition of our model near criticality and can be used to clarify some of the conjectures made for the discrete DR models.

Before the introduction of the continuous-time version model, we (re)-present  the discrete-time  model \eqref{eqn:derridaRetaux} from three different though equivalent viewpoints. In particular, we give a process description of this model, which is best suited for comparison with the continuous-time version model that we are going to introduce.

\paragraph*{Notation}
By $f(t) \sim g(t)$ as $t \to a$, we mean that $\lim_{t \to a} f(t)/g(t) = 1$. The constant $C$ is a generic positive constant, which might change from line to line.
Moreover, we set $\N \coloneqq \{0,1,2,\dots \}$ and $\R_+ \coloneqq [0,\infty)$.

\subsection{The discrete-time Derrida--Retaux model}

At first,  one might observe that \eqref{eqn:derridaRetaux} does not describe a stochastic process $(X_n)_{n \geq 0}$, but rather a measure-valued sequence $(\mu_n)_{n \geq 0}$ where $\mu_n$ is the law of $X_n$. This leads us to the following definition.
\begin{definition}[(Discrete-time) Derrida--Retaux model]
A sequence $(\mu_n)_{n \geq 0}$ of probability distributions on $\R_+$ is called a \textit{DR model} if for all continuous bounded function $f$, we have
\[
\forall n \geq 0, \quad 
  \int_{\R_+} f(x) \mu_{n+1}(\diff x) = \int_{\R_+} f\left( (x+y-1)_+ \right) \mu_n(\diff x) \mu_n(\diff y).
\]
\end{definition}

The fundamental quantity in a DR model is its associated \textit{free energy}, defined as follows:
\begin{equation}
  \label{eqn:freeEnergy}
  F_\infty \coloneqq \lim_{n \to \infty} 2^{-n} \int_0^\infty x\mu_n(\diff x).
\end{equation}
It allows the distinction between the pinned and the unpinned regimes in this model. We refer to the model as \textit{unpinned} if $F_\infty=0$, and as \textit{pinned} if $F_\infty>0$. Note that as for $x, y \in \R_+$, $(x + y - 1)_+ \leq x + y$, the sequence $( 2^{-n}\int_0^\infty x \mu_n(\diff x))_{n \geq 0}$ is non-increasing, therefore the limit in \eqref{eqn:freeEnergy} is always well-defined.
One of the main questions for this model is to determine for which starting measure $\mu_0$ the model ends up being pinned or unpinned.

When the support of $\mu_0$ is included in $\N$, this question was previously answered by Collet et al.\@  \cite{CEGM84}. By a careful study of the recursion equation for the generating functions of $\mu_n$, they proved that
\begin{equation}
  \label{eqn:cegm}
  F_\infty = 0 \iff \int_0^\infty x 2^x \mu_0(\diff x) \leq \int_0^\infty 2^x \mu_0(\diff x) < \infty.
\end{equation}
However, this characterization of the unpinned phase, which is valid for integer-valued random variables does not extend to more general initial conditions, even to the half-integers valued random variables. We refer to Chen et al.\@ \cite{CDHLS17} for further discussions and a list of open questions on this topic.

To discuss the phase transition from the pinned to the unpinned regime, it will be more convenient to specify the mass at $0$ of $\mu_0$. Consider an initial distribution $\mu_0$ of form:
\[\mu_0= p \delta_0 + (1-p) \vartheta,\]
where $p\in [0, 1]$  and $\vartheta$ denotes a probability measure on $\R_+ \backslash \{0\}$. Denote by $F_\infty(p)$ the associated free energy.  Clearly, $p\to F_\infty(p)$ is non-increasing. Write $p_c$ for the critical parameter distinguishing between the pinned and the unpinned regimes
\[p_c\coloneqq \sup\{p\in [0, 1]: F_\infty(p) >0\},\]
with the convention that $\sup\varnothing =0$. 

Assuming that $p_c \in (0,1)$,  are of particular interest to characterize the rate of convergence of $F_\infty(p) \to 0$ as $p \uparrow p_c$, and to study the critical regime when $p=p_c$. For the rate of convergence, assuming $p_c >0$ and some integrability conditions on $\vartheta$, Derrida and Retaux \cite{DR14} conjectured that there exists some constant $C>0$ such that
\begin{equation}
  \label{derridaconjecture}
  F_\infty(p) = \exp \left( - \frac{C+o(1)}{\sqrt{p_c-p}}\right), \qquad p \uparrow p_c.
\end{equation}
When the support of $\vartheta$ is included in $\mathbb N\backslash \{0\}$, a weaker form of the above conjecture has been obtained  in a recent work by Chen et al.\@ \cite{CDDHLS}.

It is necessary to assume some integrability assumptions on $\vartheta$ in Derrida and Retaux' conjecture. Indeed, Collet et al.\@ \cite{CEGM84}'s result yields the following characterization:
\[ p_c < 1 \quad \Longleftrightarrow \quad \int_0^\infty x 2^x \vartheta(\diff x) < \infty,\]
see  Hu and Shi \cite{HuS17}. Furthermore, the latter paper shows  some  non-universal behaviors of $F_\infty(p)$ as $p \uparrow 1$ for a class of   distributions $\vartheta$ satisfying $p_c=1$, in contrast with \eqref{derridaconjecture}. 

Now we turn our attention  to a second viewpoint of the Derrida and Retaux model. 
If the support of $\mu_0$ is included in $\N$, the law $\mu_n$ can be constructed as a so-called \emph{parking scheme} on the binary tree, which can be described as follows.
For $n \in \N$, we denote by $\mathbb{T}_n$ the binary tree of height $n$, rooted at $\varnothing$.
To each leaf of the tree at depth $n$ is attached an i.i.d.\@ random variable with law $\mu_0$, representing the number of cars starting from that position.
Each internal node has a single parking spot, initially empty.
All cars drive from the leaves toward the root of the tree, parking as soon as possible.
For each $u \in \T_n$, we write $X^n(u)$ the number of cars passing by the parking spot at position $u$ without being allowed to park (because another car is already parked).
We call the random map $X^n \colon \T_n \to \R_+$ the Derrida--Retaux tree (or DR tree for short). See Figure \ref{fig:parking_scheme}.
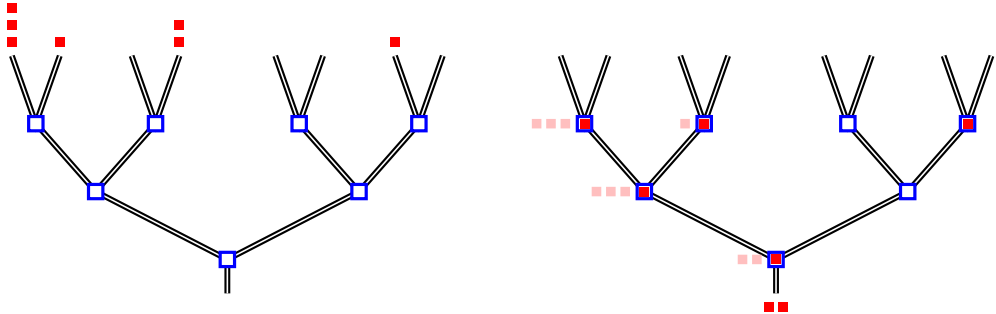
\begin{figure}[ht]
\centering
\subfigure{%
\begin{tikzpicture}[xscale = 0.315,yscale=.9]
  \draw [thick, double] (0,3) -- (1,2) -- (2,3);
  \draw [thick, double] (5,3) -- (6,2) -- (7,3);
  \draw [thick, double] (11,3) -- (12,2) -- (13,3);
  \draw [thick, double] (16,3) -- (17,2) -- (18,3);
  
  \draw [thick, double] (1,2) -- (3.5,1) -- (6,2);
  \draw [thick, double] (12,2) -- (14.5,1) -- (17,2);
  
  \draw [thick, double] (3.5,1) -- (9,0) -- (14.5,1);
  
  \draw [thick, double] (9,0) -- (9,-.5);

  \fill [color=white]  (1,2) ++ (+0.3,+0.105) -- ++ (-0.6,0) -- ++ (0,-.21) -- ++ (.6,0) -- ++ (0,.21) -- cycle;
  \draw [very thick, color=blue] (1,2)  ++ (+0.3,+0.105) -- ++ (-0.6,0) -- ++ (0,-.21) -- ++ (.6,0) -- ++ (0,.21) -- cycle;
  \fill [color=white]  (6,2) ++ (+0.3,+0.105) -- ++ (-0.6,0) -- ++ (0,-.21) -- ++ (.6,0) -- ++ (0,.21) -- cycle;
  \draw [very thick, color=blue] (6,2)  ++ (+0.3,+0.105) -- ++ (-0.6,0) -- ++ (0,-.21) -- ++ (.6,0) -- ++ (0,.21) -- cycle;
  \fill [color=white]  (12,2) ++ (+0.3,+0.105) -- ++ (-0.6,0) -- ++ (0,-.21) -- ++ (.6,0) -- ++ (0,.21) -- cycle;
  \draw [very thick, color=blue] (12,2)  ++ (+0.3,+0.105) -- ++ (-0.6,0) -- ++ (0,-.21) -- ++ (.6,0) -- ++ (0,.21) -- cycle;
  \fill [color=white]  (17,2) ++ (+0.3,+0.105) -- ++ (-0.6,0) -- ++ (0,-.21) -- ++ (.6,0) -- ++ (0,.21) -- cycle;
  \draw [very thick, color=blue] (17,2)  ++ (+0.3,+0.105) -- ++ (-0.6,0) -- ++ (0,-.21) -- ++ (.6,0) -- ++ (0,.21) -- cycle;
  \fill [color=white]  (3.5,1) ++ (+0.3,+0.105) -- ++ (-0.6,0) -- ++ (0,-.21) -- ++ (.6,0) -- ++ (0,.21) -- cycle;
  \draw [very thick, color=blue] (3.5,1)  ++ (+0.3,+0.105) -- ++ (-0.6,0) -- ++ (0,-.21) -- ++ (.6,0) -- ++ (0,.21) -- cycle;
  \fill [color=white]  (14.5,1) ++ (+0.3,+0.105) -- ++ (-0.6,0) -- ++ (0,-.21) -- ++ (.6,0) -- ++ (0,.21) -- cycle;
  \draw [very thick, color=blue] (14.5,1)  ++ (+0.3,+0.105) -- ++ (-0.6,0) -- ++ (0,-.21) -- ++ (.6,0) -- ++ (0,.21) -- cycle;
  \fill [color=white]  (9,0) ++ (+0.3,+0.105) -- ++ (-0.6,0) -- ++ (0,-.21) -- ++ (.6,0) -- ++ (0,.21) -- cycle;
  \draw [very thick, color=blue] (9,0)  ++ (+0.3,+0.105) -- ++ (-0.6,0) -- ++ (0,-.21) -- ++ (.6,0) -- ++ (0,.21) -- cycle;
  
  \fill [color=red] (0,3) ++ (0,0.2) ++ (+0.2,+0.07) -- ++ (-0.4,0) -- ++ (0,-.14) -- ++ (.4,0) -- ++ (0,.14) -- cycle;
  \fill [color=red] (0,3) ++ (0,0.45) ++ (+0.2,+0.07) -- ++ (-0.4,0) -- ++ (0,-.14) -- ++ (.4,0) -- ++ (0,.14) -- cycle;
  \fill [color=red] (0,3) ++ (0,0.7) ++ (+0.2,+0.07) -- ++ (-0.4,0) -- ++ (0,-.14) -- ++ (.4,0) -- ++ (0,.14) -- cycle;
  \fill [color=red] (2,3) ++ (0,0.2) ++ (+0.2,+0.07) -- ++ (-0.4,0) -- ++ (0,-.14) -- ++ (.4,0) -- ++ (0,.14) -- cycle;
  \fill [color=red] (7,3) ++ (0,0.2) ++ (+0.2,+0.07) -- ++ (-0.4,0) -- ++ (0,-.14) -- ++ (.4,0) -- ++ (0,.14) -- cycle;
  \fill [color=red] (7,3) ++ (0,0.45) ++ (+0.2,+0.07) -- ++ (-0.4,0) -- ++ (0,-.14) -- ++ (.4,0) -- ++ (0,.14) -- cycle;
  \fill [color=red] (16,3) ++ (0,0.2) ++ (+0.2,+0.07) -- ++ (-0.4,0) -- ++ (0,-.14) -- ++ (.4,0) -- ++ (0,.14) -- cycle;
    
  \draw [color=white] (20,3) -- (20,2);
  \fill [color=white] (9,-.7) ++ (-0.3,0) ++ (+0.2,+0.07) -- ++ (-0.4,0) -- ++ (0,-.14) -- ++ (.4,0) -- ++ (0,.14) -- cycle;
\end{tikzpicture}
}%
\subfigure{
\begin{tikzpicture}[xscale = 0.315,yscale=0.9]
  \draw [thick, double] (0,3) -- (1,2) -- (2,3);
  \draw [thick, double] (5,3) -- (6,2) -- (7,3);
  \draw [thick, double] (11,3) -- (12,2) -- (13,3);
  \draw [thick, double] (16,3) -- (17,2) -- (18,3);
  
  \draw [thick, double] (1,2) -- (3.5,1) -- (6,2);
  \draw [thick, double] (12,2) -- (14.5,1) -- (17,2);
  
  \draw [thick, double] (3.5,1) -- (9,0) -- (14.5,1);
  
  \draw [thick, double] (9,0) -- (9,-.5);

  \fill [color=white]  (1,2) ++ (+0.3,+0.105) -- ++ (-0.6,0) -- ++ (0,-.21) -- ++ (.6,0) -- ++ (0,.21) -- cycle;
  \draw [very thick, color=blue] (1,2)  ++ (+0.3,+0.105) -- ++ (-0.6,0) -- ++ (0,-.21) -- ++ (.6,0) -- ++ (0,.21) -- cycle;
  \fill [color=red] (1,2) ++ (+0.2,+0.07) -- ++ (-0.4,0) -- ++ (0,-.14) -- ++ (.4,0) -- ++ (0,.14) -- cycle;
  \fill [color=pink] (1,2) ++ (-.8,0) --++ (+0.2,+0.07) -- ++ (-0.4,0) -- ++ (0,-.14) -- ++ (.4,0) -- ++ (0,.14) -- cycle;
  \fill [color=pink] (1,2) ++ (-1.4,0) --++ (+0.2,+0.07) -- ++ (-0.4,0) -- ++ (0,-.14) -- ++ (.4,0) -- ++ (0,.14) -- cycle;
  \fill [color=pink] (1,2) ++ (-2.,0) --++ (+0.2,+0.07) -- ++ (-0.4,0) -- ++ (0,-.14) -- ++ (.4,0) -- ++ (0,.14) -- cycle;

  \fill [color=white]  (6,2) ++ (+0.3,+0.105) -- ++ (-0.6,0) -- ++ (0,-.21) -- ++ (.6,0) -- ++ (0,.21) -- cycle;
  \draw [very thick, color=blue] (6,2)  ++ (+0.3,+0.105) -- ++ (-0.6,0) -- ++ (0,-.21) -- ++ (.6,0) -- ++ (0,.21) -- cycle;
  \fill [color=red] (6,2) ++ (+0.2,+0.07) -- ++ (-0.4,0) -- ++ (0,-.14) -- ++ (.4,0) -- ++ (0,.14) -- cycle;
  \fill [color=pink] (6,2) ++ (-.8,0) --++ (+0.2,+0.07) -- ++ (-0.4,0) -- ++ (0,-.14) -- ++ (.4,0) -- ++ (0,.14) -- cycle;

  \fill [color=white]  (12,2) ++ (+0.3,+0.105) -- ++ (-0.6,0) -- ++ (0,-.21) -- ++ (.6,0) -- ++ (0,.21) -- cycle;
  \draw [very thick, color=blue] (12,2)  ++ (+0.3,+0.105) -- ++ (-0.6,0) -- ++ (0,-.21) -- ++ (.6,0) -- ++ (0,.21) -- cycle;

  \fill [color=white]  (17,2) ++ (+0.3,+0.105) -- ++ (-0.6,0) -- ++ (0,-.21) -- ++ (.6,0) -- ++ (0,.21) -- cycle;
  \draw [very thick, color=blue] (17,2)  ++ (+0.3,+0.105) -- ++ (-0.6,0) -- ++ (0,-.21) -- ++ (.6,0) -- ++ (0,.21) -- cycle;
  \fill [color=red] (17,2) ++ (+0.2,+0.07) -- ++ (-0.4,0) -- ++ (0,-.14) -- ++ (.4,0) -- ++ (0,.14) -- cycle;

  \fill [color=white]  (3.5,1) ++ (+0.3,+0.105) -- ++ (-0.6,0) -- ++ (0,-.21) -- ++ (.6,0) -- ++ (0,.21) -- cycle;
  \draw [very thick, color=blue] (3.5,1)  ++ (+0.3,+0.105) -- ++ (-0.6,0) -- ++ (0,-.21) -- ++ (.6,0) -- ++ (0,.21) -- cycle;
  \fill [color=red] (3.5,1) ++ (+0.2,+0.07) -- ++ (-0.4,0) -- ++ (0,-.14) -- ++ (.4,0) -- ++ (0,.14) -- cycle;
  \fill [color=pink] (3.5,1) ++ (-.8,0) --++ (+0.2,+0.07) -- ++ (-0.4,0) -- ++ (0,-.14) -- ++ (.4,0) -- ++ (0,.14) -- cycle;
  \fill [color=pink] (3.5,1) ++ (-1.4,0) --++ (+0.2,+0.07) -- ++ (-0.4,0) -- ++ (0,-.14) -- ++ (.4,0) -- ++ (0,.14) -- cycle;
  \fill [color=pink] (3.5,1) ++ (-2.,0) --++ (+0.2,+0.07) -- ++ (-0.4,0) -- ++ (0,-.14) -- ++ (.4,0) -- ++ (0,.14) -- cycle;

  \fill [color=white]  (14.5,1) ++ (+0.3,+0.105) -- ++ (-0.6,0) -- ++ (0,-.21) -- ++ (.6,0) -- ++ (0,.21) -- cycle;
  \draw [very thick, color=blue] (14.5,1)  ++ (+0.3,+0.105) -- ++ (-0.6,0) -- ++ (0,-.21) -- ++ (.6,0) -- ++ (0,.21) -- cycle;

  \fill [color=white]  (9,0) ++ (+0.3,+0.105) -- ++ (-0.6,0) -- ++ (0,-.21) -- ++ (.6,0) -- ++ (0,.21) -- cycle;
  \draw [very thick, color=blue] (9,0)  ++ (+0.3,+0.105) -- ++ (-0.6,0) -- ++ (0,-.21) -- ++ (.6,0) -- ++ (0,.21) -- cycle;
  \fill [color=red] (9,0) ++ (+0.2,+0.07) -- ++ (-0.4,0) -- ++ (0,-.14) -- ++ (.4,0) -- ++ (0,.14) -- cycle;
  \fill [color=pink] (9,0) ++ (-.8,0) --++ (+0.2,+0.07) -- ++ (-0.4,0) -- ++ (0,-.14) -- ++ (.4,0) -- ++ (0,.14) -- cycle;
  \fill [color=pink] (9,0) ++ (-1.4,0) --++ (+0.2,+0.07) -- ++ (-0.4,0) -- ++ (0,-.14) -- ++ (.4,0) -- ++ (0,.14) -- cycle;
  
  \fill [color=red] (9,-.7) ++ (0.3,0) ++ (+0.2,+0.07) -- ++ (-0.4,0) -- ++ (0,-.14) -- ++ (.4,0) -- ++ (0,.14) -- cycle;
  \fill [color=red] (9,-.7) ++ (-0.3,0) ++ (+0.2,+0.07) -- ++ (-0.4,0) -- ++ (0,-.14) -- ++ (.4,0) -- ++ (0,.14) -- cycle;
  \fill [color=white] (0,3) ++ (0,0.7) ++ (+0.2,+0.07) -- ++ (-0.4,0) -- ++ (0,-.14) -- ++ (.4,0) -- ++ (0,.14) -- cycle;
  \draw [color=white] (-2,3) -- (-2,2);
\end{tikzpicture}
}
\caption{Illustration of the parking scheme on a binary tree. On the left, cars (red squares) are starting from the leaves and each internal node has an empty parking spot. On the right, cars drove down to the root, parking as soon as they could. At each node $u$, the number $X^n(u)$ of cars that passed by $u$ without being able to park is indicated by pink squares.} \label{fig:parking_scheme}
\end{figure}

Observe that the following equation holds for each internal node $u \in \T_n$:
\begin{equation}
  \label{eqn:derridaRetauxDiscretOnTree}
  X^n(u) = \left( X^n(u1) + X^n(u2) - 1\right)_+,
\end{equation}
where $u1$ and $u2$ are the left and right children of $u$ respectively.
As a result, one can observe that for any $0\le k \le n$,  the family $X^n(v)$ for all nodes $v$ at the $k$th generation of $\T_n$ forms a collection of  i.i.d.\@ random variables with common  law $\mu_{n-k}$. In particular $X^n(\varnothing)$ has law $\mu_n$. 

We generalize this definition to any initial law $\mu_0$ on $\R_+$. For any $u \in \mathbb{T}_n$, denote by $\lvert u \rvert$ its generation. As such, $\{u : \lvert u \rvert = n\}$ is exactly the set of leaves of $\mathbb{T}_n$.

\begin{definition}[(Discrete time) Derrida--Retaux tree]\label{def:ddrt}
Let $\mu_0$ be a probability distribution on $\R_+$.
A \textit{DR tree} of height $n$ and initial law $\mu_0$ is a random map $X^n \colon \mathbb{T}_n \to \R_+$, with $\mathbb{T}_n$ being the binary tree of height $n$, which can be constructed as follows:
\begin{itemize}
  \item The values of the leaves $(X^n(u), |u|=n)$ are i.i.d.\@ random variables with law $\mu_0$.
  \item Each internal node value is constructed as $X^n(u) = (X^n(u1) + X^n(u2) - 1)_+$, for any $u \in \mathbb{T}_n$ with $|u|< n$.
\end{itemize}
\end{definition}

Equation \eqref{eqn:derridaRetauxDiscretOnTree} is a max-type recursive equation on trees, that belongs to the class of recursive models analysed in the survey of Aldous and Bandyopadhyay \cite{AB05}.
The parking scheme we describe here has been studied by Goldschmidt and Przykucki \cite{GP16} on critical Poisson Galton--Watson trees conditioned to be large.
Curien and Hénard \cite{CurienHenard} are also studying parking scheme on uniform trees with $n$ nodes.
In these different cases, the key observation is that the models can be analysed through the study of some generating function.
However, parking on critical and supercritical Galton--Watson trees appear to belong to different \emph{universality class}, as their behavior, in particular near criticality is very different.
Namely, the phase transition of the parking scheme on a critical Galton-Watson tree is of the second order in all known examples (\cite{GP16,CurienHenard}), while it is of infinite order on supercritical Galton--Watson trees (see \eqref{derridaconjecture}). 

The DR tree gives a natural method to generate a random variable $X_n$ with law $\mu_n$.
However, this construction does not make easy the consideration of the family $(X_n, n \geq 0)$ as a process.
Indeed, making the cars start from generation $n$ and flow back to the root only allows to construct the first $n$ generations of that process.
We introduce here a self-contained construction of $(X_n, n \geq 0)$ as a time-inhomogeneous discrete-time Markov chain.

\begin{definition}[Derrida--Retaux process]
Let $\mu_0$ be a probability distribution on $\R_+$.
A \textit{DR process} with initial law $\mu_0$ is the process constructed from the following recursion equation
\begin{equation}
  \label{eqn:mckeanVlasovDiscret}
  \forall n \geq 0, \quad
  \addtolength\arraycolsep{-0.125cm}
  \left\{
  \begin{array}{rl}
	X_{n+1} &= \left(X_n + F^{-1}_n(U_{n+1}) - 1\right)_+, \\
    F_n(x) &= \P(X_n \leq x) \text{ for all } x \geq 0.
  \end{array}
  \right.  
\end{equation}
where $(U_n, n \geq 1)$ is an i.i.d.\@ family of uniform random variables, and $X_0$ is an independent random variable with law $\mu_0$.
\end{definition}

Note that the DR process can be seen as a discrete-time version of a solution of a McKean--Vlasov type SDE, see McKean \cite{McK66},  i.e. a Markov process interacting with its distribution.
It follows from an easy recursion that for all $n \in \N$, $X_n$ has law $\mu_n$.

The aim of this article is to define continuous-time versions of the DR model, tree and process.
In particular, we are looking for a process with a density $r_t(x)$ with respect to the Lebesgue measure which could solve \cite[Equation (33)]{DR14}, that we recall here
\begin{equation}
  \label{eqn:equation33}
  \partial_t r = \partial_x r + r \ast r - r.
\end{equation}
This differential equation \eqref{eqn:equation33} plays a key role in the prediction \eqref{derridaconjecture},  and was obtained by Derrida and Retaux as an informal scaling limit of the sequence of measures $(\mu_{\floor{nt}}(n \diff x))$.

\subsection{A continuous-time version of the model, tree and process}

By analogy with the DR process, we introduce a continuous-time version as the following Markov process. Starting from an initial distribution $X_0$, the process drifts downward at speed $1$ until reaching $0$, where it stays put. Additionally, at each atom $t$ of a Poisson point process with intensity $1$, an independent copy of $X_t$ is added to $X_t$. More formally, the process is defined in the following way.
 
\begin{definition}[(Continuous-time) Derrida--Retaux process]
\label{def:drp}
A \textit{DR process} $(X_t)_{t\geq 0}$ is a solution of the McKean--Vlasov type differential equation
\begin{equation}
  \label{eqn:mcKeanRepresentation}
  \addtolength\arraycolsep{-0.125cm}
  \left\{
  \begin{array}{rl}
	X_t & = X_0 - \int_0^t \ind{X_s > 0} \diff s + \int_0^t \int_{[0,1]} F_s^{-1}(u) N(\diff s, \diff u) \\
    F_t(x) & = \P(X_t \leq x), \text{ for all } x \geq 0, \, t\ge 0,
  \end{array}  
  \right.  
\end{equation}
where $N$ is a Poisson point process on $\R_+ \times [0,1]$ with intensity $\diff s \diff u$ and $F_s^{-1}$ is the right-continuous inverse of $F_s$ defined by $F_s^{-1}(u) \coloneqq \inf\{x\ge0: F_s(x) > u\}$.
\end{definition}

Observe that this continuous-time process has the same features as the discrete-time DR process.
It exhibits the similar pattern of loosing mass up to hitting zero and being added independent copies of itself.
Note that, in order to study a continuous-time DR process $Y$ loosing mass at speed $a$ and being added independent copies of $x$ at rate $b$, one only needs to make the space-time change $Y_t =\frac{a}{b}X_{bt}$.

We prove that for any initial random variable $X_0$ on ${\mathbb R}_+$, there is a unique strong solution of \eqref{eqn:mcKeanRepresentation}. Moreover, setting $\mu_t$ for the law of $X_t$, the family $(\mu_t)_{t \geq 0}$ is continuous with respect to the usual topology of the weak convergence. It turns out that the family $(\mu_t)_{t\ge0}$ will be the solution of a partial differential equation similar to \eqref{eqn:equation33}. In that sense, it will be called a DR model, enticing the following definition.
\begin{definition}
\label{def:drm}[(Continuous-time) Derrida--Retaux model]
A \textit{DR model} is a family $(\mu_t)_{t\ge0}$ of weakly continuous probability distributions that are weak solution to the partial differential equation
\begin{equation}
  \label{eqn:pdeDef}
  \partial_t \mu_t 
  = \partial_x (\ind{x > 0} \mu_t) + \mu_t \ast \mu_t - \mu_t, 
\end{equation}
where by ``weak solution'' we mean that  
for any $\mathcal{C}^1$ bounded functions $f$ with bounded derivative and any $t>0$, we have
\begin{align}
  \label{eqn:pdeWeak1}
  & \int_\R f(x) \mu_t(\diff x) - \int_\R f(x) \mu_0(\diff x) \nonumber \\
  &  = \int_0^t\left( - \int_\R (f'(x) \ind{x>0} + f(x)) \mu_s(\diff x) + \int_{\R^2} f(x+y)\mu_s(\diff x)\mu_s(\diff y)\right) \diff s. 
\end{align}
\end{definition}

\begin{remark}
Note that \eqref{eqn:pdeDef} is very similar to  Equation \eqref{eqn:equation33} predicted by physicists. In particular, the behaviors of the PDEs are the same for all $x > 0$. The effect of adding the $\ind{x>0}$ term in \eqref{eqn:pdeDef} is to specify the evolution of the Dirac mass at $0$ separately. Contrarily to what happens in \eqref{eqn:equation33}, this ensures that the total mass of the measures is preserved.
\end{remark}

The partial differential equation \eqref{eqn:pdeDef} might be related to a Smoluchowski-type coagulation equation \cite{Smoluchowski1916} with constant kernel. However, we were not able to find references in the PDE literature to partial differential equations exhibiting a behavior similar to \eqref{eqn:pdeDef}. Nevertheless, it has recently gained focus in the probability literature, for example in the recent article of Lambert and Schertzer \cite{Lam18}, where a PDE of the form
\[
  \partial_t \rho = \partial_x (\psi \rho) + a(t) (\rho \ast \rho - \rho),
\]
with $\psi$ the branching mechanism of a continuous-state branching process, was introduced and studied in detail. 

The family of law $(\mu_t)_{t \geq 0}$ associated to the DR process can also be constructed on a Yule tree with a method similar to the parking scheme. We recall that a Yule tree is a continuous tree in which each branch has i.i.d.\@ exponential length of parameter 1, and the branching is always binary. For $s \geq 0$, we denote by $\calN_s$ the set of individuals alive in the tree at time $s$.

For some fixed $t > 0$, we define a process $X^t = (X_s^t(u), s \in [0,t], u \in \calN_s)$ on a Yule tree truncated at height $t$ by the following \textit{painting scheme}. 
On each leaf $u \in \calN_t$ of the truncated tree, there is a painter with an i.i.d.\@ random amount of paint $X_t^t(u)$, distributed according to the law $\mu_0$. Then, the procedure is completely deterministic. Each painter climbs down the tree, painting the branches of the tree, using a quantity $1$ of paint per unit of branch length. When two painters meet, they put their remaining paint in common. For each $s \in [0,t]$ and $u \in \calN_s$, the quantity $X^t_s(u)$ is the amount of paint left when the painters reach level $s$ in the branch $u$. See Figure \ref{fig:painting_scheme} for an illustration. This leads to the following definition.
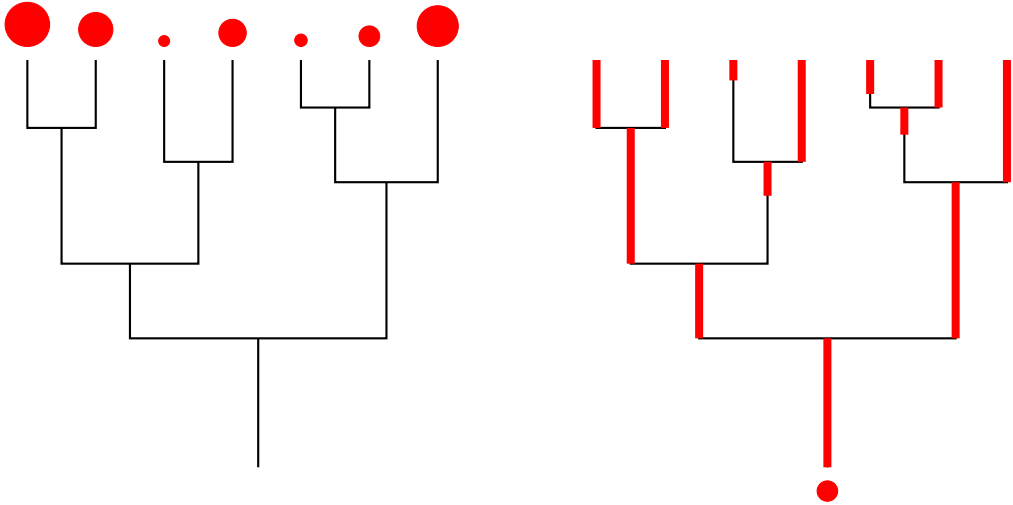
\begin{figure}[ht]
\centering
\subfigure{%
\begin{tikzpicture}[scale=0.9]
  \draw [thick] (0,5) -- (0,4) -- (1,4) -- (1,5);
  \draw [thick] (2,5) -- (2,3.5) -- (3,3.5) -- (3,5);
  \draw [thick] (4,5) -- (4,4.3) -- (5,4.3) -- (5,5);
  \draw [thick] (6,5) -- (6,3.2) -- (4.5,3.2) -- (4.5,4.3);
  \draw [thick] (2.5,3.5) -- (2.5,2) -- (0.5,2) -- (0.5,4);
  \draw [thick] (5.25, 3.2) -- (5.25,0.9) -- (1.5,0.9) -- (1.5,2);
  \draw [thick] (3.375,.9) -- (3.375,-1);
  
  \draw [red,fill=red] (0,5) ++ (0,.525) circle (.325);
  \draw [red,fill=red] (1,5) ++ (0,.45) circle (.25);
  \draw [red,fill=red] (2,5) ++ (0,.28) circle (.08);
  \draw [red,fill=red] (3,5) ++ (0,.4) circle (.2);
  \draw [red, fill=red] (4,5) ++ (0,.291) circle (.091);
  \draw [red, fill=red] (5,5) ++ (0,.351) circle (.151);
  \draw [red, fill=red] (6,5) ++ (0,.5) circle (.3);

  \draw [white,fill=white] (3.375,-1) ++ (0,-.35) circle (.15);
  \draw [white] (7,5) -- (7,4);
\end{tikzpicture}
}%
\subfigure{
\begin{tikzpicture}[scale=0.9]
  \draw [thick] (0,5) -- (0,4) -- (1,4) -- (1,5);
  \draw [thick] (2,5) -- (2,3.5) -- (3,3.5) -- (3,5);
  \draw [thick] (4,5) -- (4,4.3) -- (5,4.3) -- (5,5);
  \draw [thick] (6,5) -- (6,3.2) -- (4.5,3.2) -- (4.5,4.3);
  \draw [thick] (2.5,3.5) -- (2.5,2) -- (0.5,2) -- (0.5,4);
  \draw [thick] (5.25, 3.2) -- (5.25,0.9) -- (1.5,0.9) -- (1.5,2);
  \draw [thick] (3.375,.9) -- (3.375,-1);
  
  \draw [white,fill=white] (0,5) ++ (0,.525) circle (.325);
  \draw [white] (-1,5) -- (-1,4);

  \draw [line width=3,color=red] (0,5) -- (0,4) ++ (1,0) -- (1,5);
  \draw [line width=3,color=red] (2,5) -- ++ (0,-0.3);
  \draw [line width=3,color=red] (3,5) -- ++ (0,-1.5);
  \draw [line width=3,color=red] (2.5,3.5) -- (2.5,3);
  \draw [line width=3,color=red] (4,5) -- ++ (0,-0.5);
  \draw [line width=3,color=red] (5,4.3) -- (5,5);
  \draw [line width=3,color=red] (6,5) -- (6,3.2);
  \draw [line width=3,color=red] (4.5,3.9) -- (4.5,4.3);
  \draw [line width=3,color=red] (0.5,2) -- (0.5,4);
  \draw [line width=3,color=red] (5.25, 3.2) -- (5.25,0.9) ++ (-3.75,0) -- (1.5,2);
  \draw [line width=3,color=red] (3.375,.9) -- (3.375,-1);
  
  \draw [red,fill=red] (3.375,-1) ++ (0,-.35) circle (.15);
\end{tikzpicture}
}
\caption{Illustration of the painting scheme on a Yule tree. On the left, painters are starting on the leaves with a random amount of red paint. On the right, painters climbed down to the root, painting the tree as long as they could. The remaining paint at the root is $X^t_0(\emptyset)$.} \label{fig:painting_scheme}
\end{figure}

\begin{definition}[(Continuous-time) Derrida--Retaux tree]
\label{def:drt}  Let $\mu_0$ be a probability distribution on $\R_+$.
For $t > 0$, a \textit{DR tree} of height $t$ and initial law $\mu_0$ is a random process $X^t = (X_s^t(u), s \in [0,t], u \in \calN_s)$, where $X^t_s(v)$ represents the amount of paint remaining at level $s$ on the branch $v$ in the painting scheme starting with $(X_t^t(u), u \in \calN_t)$ that are i.i.d.\@ with law $\mu_0$ conditionally on the Yule tree.
\end{definition}

Denoting by $\bar{\mu}_t$ the law of $X^t_0(\varnothing)$, we call $(\bar{\mu}_t)_{t \geq 0}$ the \textit{family of laws obtained from the tree-painting scheme} (in particular, $\bar{\mu}_0= \mu_0$).
We observe that by the branching property of the Yule tree, the subtree starting from individual $u$ at time $s$ is a Yule tree of height $t-s$. Therefore for all $s \leq t$, the family $(X^t_s(u), u \in \mathcal{N}_s)$ are i.i.d.\@ random variables with law $\bar{\mu}_{t-s}$.
Contrarily to the two previous definitions, the question of existence and uniqueness of $(\bar{\mu}_t)_{t \geq 0}$ is straightforward in Definition~\ref{def:drt}. We prove later that $(\bar{\mu}_t)_{t \geq 0}$ is a DR model, see Theorem \ref{t:main1} below.

We refer to Section~\ref{subsec:Yule} for a detailed definition of the continuous-time  DR tree. However, it is worth noticing that the  painting scheme can be obtained straightforwardly as a scaling limit, when $K\to\infty$,  of the parking scheme on a Galton--Watson tree of height $Kt$, in which particles give birth to two children with probability $1/K$ and to one child with probability $1-1/K$, and the starting number of cars in a given leaf has law $\floor{KX_0}$, with $X_0$ of law $\mu_0$.

The first main result of the article justifies and unifies the three definitions of the continuous-time model given above.
\begin{theorem}\label{t:main1}
Let $\mu_0$ be a probability distribution on $\R_+$, and $(\bar{\mu}_t)_{t \geq 0}$ the family of laws obtained from the tree-painting scheme.
\begin{enumerate}
  \item[(i)] There exists a unique strong solution $(X_t)_{t\geq 0}$ to  Equation \eqref{eqn:mcKeanRepresentation} with $X_0$ of law~$\mu_0$. Moreover, for all $t \geq 0$, $X_t$ has law $\bar{\mu}_t$.
  \item[(ii)] The family $(\bar{\mu}_t)_{t \geq 0}$ is the unique weak solution of the PDE \eqref{eqn:pdeDef} with initial condition $\mu_0$, therefore a DR model with initial law $\mu_0$. 
\end{enumerate}
\end{theorem}

Additionally, we obtain preliminary results on the asymptotic behavior of general continuous-time DR models, see Section \ref{s:general}. However, many fundamental  questions remain open for this complex dynamics in the general case.  In the next subsection, we shall present a simple two-parameter family of measures which turns out to be stable by the dynamics \eqref{eqn:pdeDef}. Therefore, by restricting ourselves to this family, we  are able to introduce an  exactly solvable model, which exhibits the researched phase transition at criticality.

\subsection{An exactly solvable model}

Let $(\mu_t)_{t\ge0}$ be a continuous-time DR model. We are particularly interested in an exactly solvable case where  the initial law $\mu_0$ is a mixture of an exponential distribution and a Dirac mass at 0:
\begin{equation}
  \label{eqn:initialDistribution1}
  \mu_0= p \delta_{0}+ (1-p) \lambda \e^{-\lambda x} \1_{\{x>0\}} \diff x, \qquad p\in [0, 1], \lambda>0.
\end{equation}
Denote by $ F_\infty(p, \lambda) $ the associated \textit{free energy}: 
\begin{equation}
  \label{eqn:freeEnergy?}
  F_\infty(p, \lambda) \coloneqq \lim_{t \to \infty} \e^{-t}  \int_0^\infty x \mu_t(\diff x)  .
\end{equation}
To describe the asymptotic behavior of $\mu_t$ with respect to the parameters $(p, \lambda)$ of the initial distribution $\mu_0$, we introduce three sets:
\begin{align*}
 {\mathscr P} & \coloneqq \left\{ (p,\lambda) \in [0,1) \times [0,\infty) : \lambda < 1 \text{ or } (\lambda  > 1 \text{ and } p < \lambda - \lambda \log \lambda) \right\},\\
 {\mathscr C} & \coloneqq \left\{ (p,\lambda) \in [0,1) \times [1,\infty) : p = \lambda - \lambda \log \lambda \right\},\\
  {\mathscr U} & \coloneqq \left\{ (p,\lambda) \in [0,1) \times [1,\infty) : p > \lambda - \lambda \log \lambda \right\},
\end{align*}
which are represented in Figure \ref{fig:phaseDiagramIntro}.  The case $p=1$ is  degenerate  therefore excluded from the set of parameters. 

The following result gives a classification of this DR model in pinned, unpinned and critical behaviors for any starting distribution.
\begin{theorem}
\label{thm:asympBehaviour} Let $(\mu_t)_{t\ge0}$ be the continuous-time DR model with initial distribution $\mu_0$ given by \eqref{eqn:initialDistribution1}. 
\begin{enumerate}
\item[(i)] If $(p ,\lambda) \in {\mathscr P}$, then $F_\infty(p, \lambda)>0$.
\item[(ii)] If $(p ,\lambda) \in {\mathscr U}$, then there exists some positive constant $C$ such that 
\[
\int_0^\infty y \mu_t(\diff y)  \sim C \e^{-(x-1)t}, \qquad \text{as } t \to \infty,
\]
where $x$ is the unique solution larger than 1 of the equation $Hx - x\log x = 1$ and $H := \frac{p}{\lambda} - \log \lambda$.
\item[(iii)] If $(p ,\lambda) \in {\mathscr C}$, then 
\[
\int_0^\infty y \mu_t(\diff y) \sim \frac{2}{t^2}, \qquad \text{as } t\to \infty.
\]
Furthermore, conditionally on $(0, \infty)$, $\mu_t$ converges  weakly  to a standard exponential distribution. 
\end{enumerate}
\end{theorem}

For the discrete-time DR model, the counterpart of (iii) was conjectured in Chen et al.\@ \cite{CDHLS17}. 

\begin{figure}[ht]
\centering
\begin{tikzpicture}[scale=2.5]
  \fill [color=green!50] (0,0) -- (0,1) -- (1,1) -- plot[domain = 1:2.71] (\x,{\x - \x*ln(\x)}) -- (2.71,0) -- cycle;
  \fill [color =blue!50] (1,1) -- plot[domain = 1:2.71] (\x,{\x - \x*ln(\x)}) -- (2.71,0) -- (4,0) -- (4,1) -- cycle;

  \draw [color=red, very thick, domain = 1: 2.71828, samples=200] plot (\x,{\x-\x*ln(\x)}); 

  \draw [->,>=latex] (-0.1,0) -- (4.2,0) node[below] {$\lambda$};
  \draw (0,0) node[below left] {0};
  \draw [->,>=latex] (0,-0.1) -- (0,1.2) node[left] {$p$};
  \draw [very thick ,color=orange] (1,1) -- (4,1);
  \draw [color=orange] (0,0) node {$\bullet$};
  
  \draw (0.03,1) -- (-0.03,1) node[left]{$1$};
  \draw (1,0.03) -- (1,-0.03) node[below]{$1$};
  \draw (2.71828,0.03) -- (2.71828,-0.03) node[below]{$\e$};
  
  \draw (1, 0.5) node {$\mathscr{P}$};
  
  \draw[color=red] (1.96,0.5) node {$\mathscr{C}$};
  
  \draw (3.3,0.5) node{$\mathscr{U}$};
\end{tikzpicture}
\caption{The pinned and unpinned regions of the solvable DR model}
\label{fig:phaseDiagramIntro}
\end{figure}
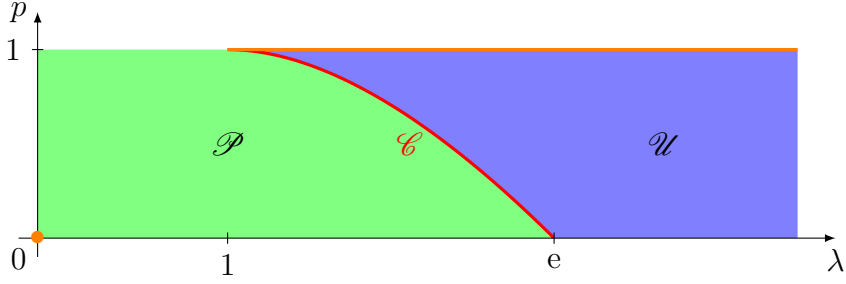

Let us call ${\mathscr P}$, ${\mathscr U}$ and ${\mathscr C}$ respectively the \textit{supercritical}, \textit{subcritical} and \textit{critical} zones of the dynamics. Note that the critical zone ${\mathscr C}$ is the graph of the function $\lambda \mapsto \lambda - \lambda \log \lambda$ on $(1,\infty)$, which splits the domain $[0,1) \times \R_+$ into the supercritical and subcritical zones. In the supercritical zone, the dynamics is such that the measure $\mu_t$ is attracted by the stable equilibrium $(0,0)$ (in orange in Figure \ref{fig:phaseDiagramIntro}). In the critical zone, $\mu_t$ is attracted by the critical point $(1,1)$ (in red in Figure \ref{fig:phaseDiagramIntro}). Finally, in the subcritical zone, $\mu_t$ converges toward a point on the orange lane, with $p=1$ and $\lambda > 1$. In other words, $\mu_t$ converges toward $\delta_0$ in probability if $(p,\lambda) \in \mathscr{U} \cup \mathscr{C}$.

It is of interest to  give  the precise behavior of $p \mapsto F_\infty(p,\lambda)$ as the function hits $0$, whenever that point exists. Note that for $\lambda > e$, one has $F_\infty(p,\lambda)= 0$ for all $p \in (0,1)$, therefore this behavior can be observed only for $\lambda < e$.
\begin{theorem} \label{thm:DRConjectureSimple} 
Fix some $\lambda \in (0,e)$ and let $p \in (0,1)$ vary.
\begin{enumerate}
\item[(i)]  If $\lambda \in (1,e)$, setting $p_c = \lambda - \lambda \log \lambda$, there exists $C>0$ such that
\[
F_\infty(p,\lambda) 
\sim \ind{p < p_c} C \exp\left( - \pi \sqrt{2\lambda} (p_c - p)^{-1/2} \right) 
\quad \text{as } p \uparrow p_c.
\]
\item[(ii)]  If $\lambda =1$, we have $F_\infty(1,1) =0$ and there exists $C>0$ such that
\[
F_\infty(p,\lambda) 
\sim C (1-p)^{2/3} 
\exp\left( -\frac{\pi}{\sqrt{2}} (1-p)^{-1/2} \right) 
\quad \text{as } p \uparrow  1.
\]
\item[(iii)]  If $\lambda \in (0,1)$, we have $F_\infty(1,\lambda) = 0$ and there exists $C>0$ such that
\[
F_\infty(p,\lambda) 
\sim C (1-p)^{1/(1-\lambda)} 
\quad \text{as } p \uparrow  1.
\]
\end{enumerate}
\end{theorem}

We may observe the difference of a factor $2$ in the exponents of the right-hand sides in cases (i) and (ii).
Moreover, (i) and (ii) prove the DR conjecture for our model, with enough precision to see the polynomial factor that appears in case (ii).
On the other hand, part (iii) gives a continuous version of the polynomial decay obtained in by Hu and Shi \cite{HuS17}.

We shall take a particular interest in  the  DR tree in the critical regime, conditioned on survival. It is believed that the conditioned tree  limit  is an universal characteristic, and would appear in numerous models related to DR at criticality.  
For $t>0$, let $X^t$ denote the continuous-time DR tree with initial distribution $\mu_0$ given by \eqref{eqn:initialDistribution1} (see Definition \ref{def:drt}). 
Conditionally on the event $\{X_0^t(\varnothing)>0\}$, 
let $\mathfrak{X}^t$ be the restriction of the process $X^t$ to the painted connected component of the root $\varnothing$. 
Assuming that the paint is red, the tree $ \mathfrak{X}^t $ is called the \textit{red tree}, see Section \ref{s:criticalDR} for a more precise definition and properties.

\begin{theorem}
\label{t:critical} Assume $(p, \lambda) \in {\mathscr C}$.
\begin{enumerate}
\item[(i)] The rescaled process $(\frac1{t} \mathfrak{X}^t_{s t})_{0\le s < 1}$ converges locally in law as $t \to \infty$ toward a time-inhomogeneous branching Markov process $(\mathfrak{X}_s)_{0\le s < 1} $.
\item[(ii)] For each $t \geq 0$, let $N_t$ denote the number of leaves in $\mathfrak{X}^t$. There exists a positive constant $C$ such that $N_t/t^2$ converges in law to $C \int_0^1 r^2(s) \diff s$ as $t \to \infty$, where $(r(s))_{0\le s\le 1}$ denotes a $4$-dimensional Bessel bridge. 
\end{enumerate}
\end{theorem}

The limiting process $\mathfrak{X}$ is a continuous-time system of particles with masses, that grow continuously and split at random times. As such, it can be called a time-inhomogeneous Markovian growth-fragmentation process. It can be described as satisfying the following properties:
\begin{enumerate}
  \item It starts at time $0$ with a unique particle of mass $0$.
  \item The mass associated to each particle grows linearly at speed $1$.
  \item A particle of mass $m$ at time $s$ splits at rate $2m/(1-s)^2$ into two children, the mass $m$ being split uniformly between the two children.
  \item Particles behave independently after their splitting time.
\end{enumerate}
The local convergence in the statement (i) can be seen equivalently as the convergence of the particle system on the interval $[0,1-\epsilon]$ for all $\epsilon > 0$, or the joint convergence in law of the splitting time of the individuals in the $n$ first generations of the process for all $n \geq 1$. 
Indeed, for fixed $\epsilon> 0$, the total number of particles in $\mathfrak{X}^t$ at height $(1-\epsilon)t$ remains tight as $t \to \infty$. Conversely, the number of particles at height $t$ explodes as $t \to \infty$ as shown in statement (ii); see Figure \ref{fig:red_tree}.
We refer to Theorems \ref{th:scaling_limit_of_the_red_tree} and \ref{th:total_number_mass_of_red_leaves} respectively for a more refined statement of (i) and (ii), respectively. 
\begin{figure}[ht]
\centering
\includegraphics{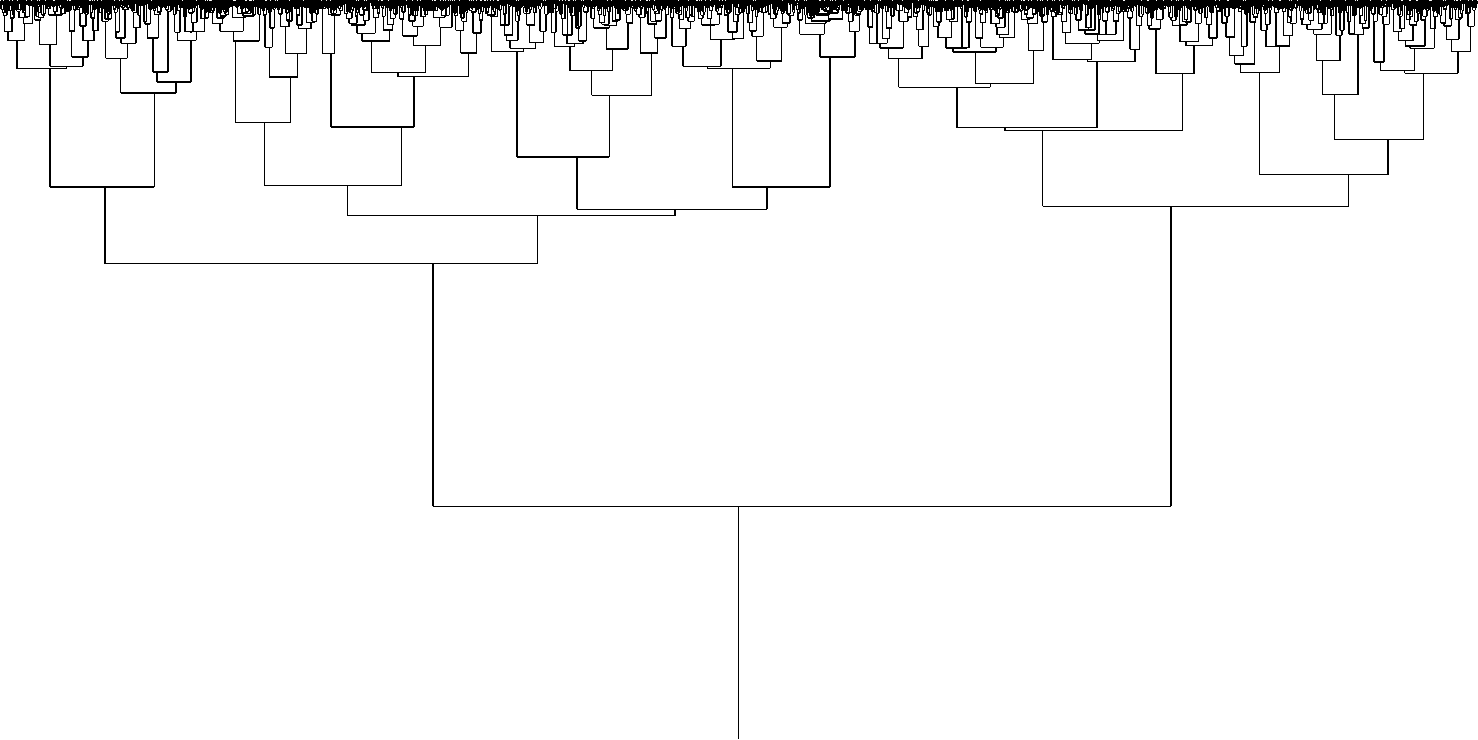}
\caption{Realization of the red tree $\mathfrak{X}^t$ with $t = 200$ and initial condition $(p,\lambda) = (0,e)$.}
\label{fig:red_tree}
\end{figure}

\begin{remark}
The counterpart of Theorem \ref{t:critical}(ii) for the discrete DR model is predicted to hold by B. Derrida (personal communication).
\end{remark}

\subsection{Organization of the paper}

The rest of the article is organised as follows. In Section \ref{sec:defmodel}, we define properly the three constructions of the continuous version of the DR model. These constructions turn out to be equivalent, as shown by Propositions \ref{prop:defMcKean} and \ref{prop:defedp}. In particular, Theorem~\ref{t:main1} follows.

In Section \ref{s:casexponentiel} we study the exactly solvable model in which the evolution is restricted to the class of mixtures of exponential distributions and Dirac masses at $0$.  We give at first the phase diagram of the model (Proposition \ref{thm:asympBehaviourIntro}), then analyze the asymptotic of  a system of differential equations (Propositon \ref{prop:asymptotics_of_p_and_lambda}), and prove Theorems \ref{thm:asympBehaviour} and \ref{thm:DRConjectureSimple} in Sections \ref{subsec:asymptoticBehaviourSystDiff} and \ref{subsec:proofDR} respectively. 

Sections \ref{s:criticalDR} and \ref{s:nombre} are devoted to the study of the critical regime of our model. Conditioned on survival, the DR tree at any fixed height is a time-inhomogeneous branching Markov process (Proposition \ref{prop:law_of_the_red_tree}). By letting the height go to infinity, the scaling limit of the conditioned tree exists (Theorem \ref{th:scaling_limit_of_the_red_tree})  which gives the part~(i) in Theorem \ref{t:critical}. In Section \ref{s:nombre}, we are interested in some quantitative characteristics of the  conditioned  (red) tree, the number $N_t$ and mass  $M_t$ of  leaves of the red tree of height~$t$.  We prove in Lemma \ref{lem:function_phi} that the joint Laplace transform of  $(N_t, M_t)$ is given by the unique solution of a differential equation whose asymptotics are given in Section~\ref{ssub:laplace}. This allows us to obtain  the convergence in law of $(N_t/t^2, M_t/t^2)$ (Theorem \ref{th:total_number_mass_of_red_leaves}) as $t\to\infty$. The part~(ii) in Theorem \ref{t:critical} is a particular case  of Theorem \ref{th:total_number_mass_of_red_leaves}.

In Section \ref{s:general}, we give some  results on the continuous--time DR model with a general initial distribution on $\R_+$.  In particular, we prove that $\e^{-t} X_t$ converges in law to an exponential distribution in the pinned case (Proposition \ref{p:convergenceexponentielle}) and give some necessary conditions for the model to being unpinned (Proposition \ref{p:conditionnessaire}).  We end Section~\ref{s:general} by presenting some open questions.

\section{Constructions of the continuous-time Derrida--Re-taux model}
\label{sec:defmodel}

In this section, we first give a precise definition of the DR tree and a characterization of the family of laws obtained from the tree-painting scheme.
Then, we show that the DR process is well-defined as a solution of the ``McKean-Vlasov type'' stochastic differential equation \eqref{eqn:mcKeanRepresentation} and that its one-dimensional marginal distributions can be obtained from the tree-painting scheme.
Finally, we prove that the DR model is well-defined as the solution of the partial differential equation \eqref{eqn:pdeDef}: existence follows from It\^o formula applied to the DR process and uniqueness from the characterization of the family of laws obtained from the tree-painting scheme.
This shows that these constructions are equivalent, as stated in Theorem~\ref{t:main1}.

\subsection{Construction from a Yule tree: the Derrida--Retaux tree}
\label{subsec:Yule}

We recall that a \textit{Yule tree} is the genealogical tree of a simple population, in which every individual splits after an independent exponential time of parameter one into two children. This tree can be constructed straightforwardly using the standard Ulam-Harris-Neveu notation that we now recall. We denote by $\T = \bigcup_{n \in \Z_+} \{1,2\}^n$ the space of all finite words on $\{1,2\}$, where by convention $\{\varnothing\} = \{1,2\}^0$. The element $u = (u(1),\dots,u(n)) \in \T$ represents the $u(n)$th child of the $u(n-1)$th child of the\dots of the $u(1)$th child of the root~$\varnothing$. We denote by $\lvert u \rvert = n$ the generation to which $u$ belongs, and for $k \leq n$ by $u_k = (u(1),\ldots, u(k))$ its ancestor alive at generation~$k$. We set $u1$ and $u2$ the labels of the two children of $u$, obtained by adding $1$ or $2$ to the end of the word $u$.

The Yule tree is constructed as follows. We denote by $\{e_u, u \in \T\}$ a family of i.i.d.\@ exponential random variables, that represent the time that individual $u$ waits until splitting into its two children $u1$ and $u2$. We then define recursively the birth time $b_u$ of each individual $u \in \T$ by
\[
  b_u = \sum_{k < \lvert u \rvert} e_{u_k}, \quad \text{with } b_\varnothing = 0.
\]
We also set $d_u = b_u + e_u$ the time at which the individual $u$ splits into its children, and $\mathcal{N}_t = \{ u \in \T : b_u \leq t < d_u\}$ the set of particles alive at time $t$.

We now give the detailed definition of the continuous-time DR tree, by analogy with its definition for Galton--Watson trees. Let $\mu_0$ be a probability distribution on $\R_+$ and $t>0$. 
Given the Yule tree, we denote by $(X^t_t(u), u \in \mathcal{N}_t)$ i.i.d.\@ random variables with law $\mu_0$. 
For $u \in \calN_t$, we define $X^t_s(u)=(X^t_t(u)-(t-s))_+$ for each $s \in [b_u,t]$. 
Then, we proceed recursively from the leaves to the root by setting, for any $v \in \T$ such that $d_v \leq t$ and any $s \in [b_v,d_v)$,
\begin{equation}\label{Xtsv}
  X^t_s(v) = (X^t_{b_{v1}}(v1)+X^t_{b_{v2}}(v2) - (d_v - s))_+,
\end{equation}
noting that $b_{v1} = b_{v2} = d_v$ (see Figure \ref{fig:painting_procedure}). The process $X^t=(X^t_s(u), s \in [0,t], u \in \mathcal{N}_s)$ is \textit{a continuous-time DR tree} of height $t$ and initial law $\mu_0$, introduced in Definition~\ref{def:drt}.
Denoting by $\mu_t$ the law of $X^t_0(\varnothing)$, the family $(\mu_t)_{t \geq 0}$ is called \textit{the family of laws obtained from the tree-painting scheme}.

\begin{figure}[!t]
\centering
\begin{tikzpicture}[scale=0.5]
\draw[->,>=latex] (0,-1) -- (0,13);
\draw (0.1,0) -- (-0.1,0) node[left]{$0$};
\draw (0.1,12) -- (-0.1,12) node[left]{$t$};
\draw[dashed] (-0.1,12) -- (17.1,12);
\draw[->,>=latex] (17,13) -- (17,-1);

\node[text width=2.5cm,text centered] at (-4,11.5) 
	{\small orientation of time for the Yule tree};
\node[text width=2.5cm,text centered] at (19.5,1) 
	{\small orientation of time for the painting procedure};

\draw[thick] (3,12) -- (3,7) -- (6,7) -- (6,12);
\draw (3,9.5) node[left]{$v1$};
\draw (6,9.5) node[left]{$v2$};
\draw[thick] (9,12) -- (9,8) -- (13.5,8) -- (13.5,10);
\draw[thick] (12,12) -- (12,10) -- (15,10) -- (15,12);
\draw[thick] (4.5,7) -- (4.5,2) -- (11.25,2) -- (11.25,8);
\draw (4.5,4) node[left]{$v$};
\draw[thick] (7.875,2) -- (7.875,0);

\draw[red] (12,12) node{$\bullet$};
\draw[red] (12,12) node[above]{$X^t_t(u)$};
\draw[red] (6,7) node{$\bullet$};
\draw[red] (6,7) node[right]{$X^t_{b_{v2}}(v2)$};
\draw[red] (3,7) node{$\bullet$};
\draw[red] (3-0.3,7) node[below]{$X^t_{b_{v1}}(v1)$};
\draw[red] (4.5,5) node{$\bullet$};
\draw[red] (4.5,5) node[right]{$X^t_s(v)$};

\draw[dashed] (-0.1,7) -- (4.5,7);
\draw[dashed] (-0.1,7) node[left]{$d_v = b_{v1} = b_{v2}$};
\draw[dashed] (-0.1,2) -- (4.5,2);
\draw[dashed] (-0.1,2) node[left]{$b_v$};
\draw[dashed] (-0.1,5) -- (4.5,5);
\draw[dashed] (-0.1,5) node[left]{$s$};
\end{tikzpicture}
\caption{Illustration of the definition of the $X_s^t(u)$'s on the Yule tree.}
\label{fig:painting_procedure}
\end{figure}
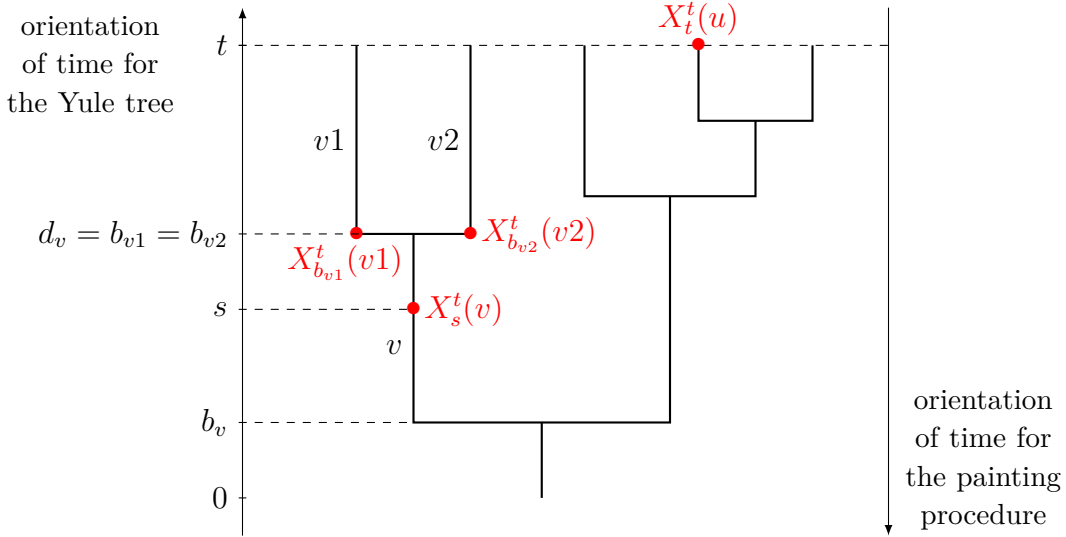

We now give  a characterization of family of laws obtained from the tree-painting scheme. For any probability distribution $\nu$ on $\R_+$ and $t > 0$, we introduce the shifted distribution $\tau_t \nu$ satisfying
\begin{equation}
  \label{eqn:defShift}
  \forall f \in \mathcal{C}_b(\R_+), \quad 
  \int_{\R_+} f(x) (\tau_t \nu)(\diff x) = f(0) \nu([0,t]) + \int_{\R_+} f(x-t) \nu(\diff x),
\end{equation}
that represents the law of $(Z-t)_+$ if $Z$ has law $\nu$. We can now describe the law of $\mu_t$ by decomposing it with respect to the first branching time in the Yule tree.
\begin{lemma}
\label{lem:recursiveEquation}
The family $(\mu_t)_{t\geq 0}$  of laws on $\R_+$   is a  family   obtained from the tree-painting scheme if and only if
\begin{equation}
  \label{eqn:recursiveEquation}
  \forall t \geq 0, \quad
  \mu_t = \int_0^t \e^{-s} \tau_s(\mu_{t-s} \ast \mu_{t-s}) \diff s + \e^{-t} \tau_t(\mu_0).
\end{equation}
\end{lemma}

\begin{proof}
Let $\T$ be a Yule tree, we recall that the first branching time $d_\varnothing$ is distributed according to an exponential random variable with parameter 1. Moreover, the subtrees starting from each child are i.i.d.\@  Yule trees. As a result, conditionally on $\{d_\varnothing=s\}$ with $s \leq t$, $X_s^t(1)$ and $X_s^t(2)$ are two independent random variables with law $\mu_{t-s}$, where as  before, $1$ and $2$ are the two children of the root $\varnothing$.  Note that by construction,  for every measurable bounded function $f$, we have
\[
  \E\left[f(X^t_0(\varnothing))\right] 
  = \E\left[f ( (X_{d_\varnothing}^t(1) 
  		+ X_{d_\varnothing}^t(2)- d_\varnothing)_+) 
  		\ind{d_\varnothing \le  t}\right]
  + \E\left[f((X_0-t)_+) \ind{d_\varnothing > t}\right],
\]
where $X_0$ is a random variable independent of $d_\varnothing$ with law $\mu_0$. Integrating w.r.t.\@ $d_\varnothing$, this proves that $(\mu_t)$ satisfies \eqref{eqn:recursiveEquation}.

Reciprocally, let $(\mu_t)_{t \geq 0}$  be a family of probability measures satisfying \eqref{eqn:recursiveEquation}. Fix $t>0$, we shall construct a continuous-time DR tree $X^t $ of height $t$ and initial law $\mu_0$ such that $X_0^t(\varnothing)$ has law $\mu_t$. To this end,  we consider on the same probability space three independent random variables $T, U, \widetilde U$ such that  $T$  has the standard exponential distribution, $U$ and $\widetilde U$ are uniformly distributed on $[0, 1]$. For any $s\ge0$, let $F_s(x) \coloneqq \mu_s([0, x]), x\ge0$ and $F_s^{-1}$ be the right-continuous inverse of $F_s$. Let $\varnothing$ be the root of the DR tree that we are constructing. Define $d_\varnothing\coloneqq T$ and denote by $1$ and $2$ the two children of $\varnothing$. Let
\[
  X_0^t(\varnothing)\coloneqq
  (X_0 - t)_+  \ind{T >t } + (X_T^t(1) + X_T^t(2)- T)_+\ind{T \le t},
\]
where $X_T^t(1)\coloneqq F^{-1}_{t-T}(U)$, $X_T^t(2)\coloneqq F^{-1}_{t-T}(\widetilde U)$ and $X_0 = F^{-1}_0(U)$. By \eqref{eqn:recursiveEquation}, we observe immediately that $X_0^t(\varnothing)$ has law $\mu_t$.  

Now, on the event $\{ T \leq t\}$, we consider $1$ and $2$ as two new roots, and we replace $X_T^t(1)$ and $X_T^t(2)$ by new random variables constructed using new independent exponential and uniforms variables on an eventually enlarged probability space, constructed in the same way as $X_0^t(\varnothing)$. The procedure can be iterated until one obtains a Yule tree of height $t$, with i.i.d.\@ random variables of law $\mu_0$ attached to each leaf of $\T_t$, in such a way that $X_0^t(\varnothing)$ is the result of the painting scheme applied to that tree, which concludes the proof. \end{proof}

\subsection{Associated McKean-Vlasov type SDE}

In the previous section, we introduced a family  $(\mu_t)_{t\ge 0}$ of probability distributions on $\R_+$,  obtained from the tree-painting scheme on a Yule tree.  This procedure allows also to introduce a Markov process $(X_s)_{s \in [0,t]}$ such that for each $s \in [0,t]$, $X_s$ has law $\mu_s$: for this, consider a Yule tree of height $t$ and let $X_s \coloneqq X^t_{t-s}(u)$ where $u$ is the leftmost particle alive at time $t-s$. Note that time is running in the opposite direction as for the Yule tree. 

This process $X$ can be described as follows: $X$ decreases linearly at speed $1$ until reaching $0$. At each branching time $s$ of the leftmost particle, $X_s$ is replaced by $X_s + \tilde{X}_s$, where $\tilde{X}_s$ is independent of $X_s$ with law $\mu_s$. This description leads to Definition~\ref{def:drp} which says that a \textit{DR process} is a non-negative Markov process $(X_t)_{t \geq 0}$ solution to the McKean--Vlasov type stochastic differential equation:

\begin{equation}
  \label{eqn:DRsde}
  \addtolength\arraycolsep{-0.125cm}
  \left\{
  \begin{array}{rl}
    X_t &= X_0 - \int_0^t \ind{X_s > 0} \diff s + \int_0^t \int_0^1 F^{-1}_s(u) N(\diff s,\diff u)\\
    F^{-1}_t(u) &= \inf\{ y \in [0,\infty) : \P(X_t \leq y) > u\},
  \end{array}
  \right.
\end{equation}
where $N$ is a Poisson point process on $\R_+ \times [0,1]$ with intensity $\diff s \diff u$.
The drift term $-\ind{X_s > 0} \diff s$ indicates the amount of paint lost per unit of length, and each atom $(s,u)$ of $N$ corresponds to a time $s$ at which the subtree gains a new branching time, in which case $F^{-1}_s(u)$ is an independent copy of $X_s$. The following result shows that the one-dimensional marginals of $X$ constitutes a family obtained from the tree-painting scheme.

\begin{proposition}
\label{prop:defMcKean}
Let $\mu_0$ be a probability distribution on $\R_+$. There exists a unique strong solution to  Equation \eqref{eqn:DRsde} such that $X_0$ has law $\mu_0$. Moreover, writing $\mu_t$ the law of $X_t$ for each $t \geq 0$, the family $(\mu_t)_{t \geq 0}$ is a  family of laws   obtained from the tree-painting scheme.
\end{proposition}

\begin{proof}
We first prove there exists a solution to \eqref{eqn:DRsde}. To do so, let $(\mu_t)_{t \geq 0}$ be a family obtained from the tree-painting scheme.
For all $t \geq 0$ and $u \in [0,1]$, we denote by $G_t(u) = \inf\{ y \in [0,\infty) : \mu_t([0,y]) > u\}$. We observe easily that there exists a unique strong solution to the SDE
\[
  X_t = X_0 - \int_0^t \ind{X_s > 0} \diff s + \int_0^t \int_0^1 G_s(u) N(\diff s,\diff u),
\]
by decomposing on the atoms of the point measure $N$. We are going to show that, if $X_0$ has law $\mu_0$, then $X_t$ has law $\mu_t$ for each $t > 0$. This will prove that $X$ is a strong solution of \eqref{eqn:DRsde}, as in that case we will have $G_s = F_s^{-1}$ for all $s \geq 0$.

Let $t > 0$, we denote by $(s_j,u_j)_{1 \leq j \leq n}$ the atoms of $N$ in $[0,t] \times [0,1]$, ranked in the increasing order of their first coordinate. For each $j \leq n$, we write $\tilde{X}_j = G_{s_j}(u_j)$. We note that conditionally on $(s_j)_{1 \leq j \leq n}$, $(\tilde{X}_j)$ are independent random variables, and $X_j$ has law $\mu_{s_j}$. As $(\mu_s)_{s\geq 0}$ is a DR model, one can construct a Yule tree $\tau_j$ and a decoration of the leaves with i.i.d.\@ random variables with law $\mu_0$, such that the amount of paint remaining at the root of $\tau_j$ is $\tilde{X}_j$.

We then define a tree of height $t$ with a decoration on its leaves as follows: we start with a line of height $t$, and for each $j \leq n$, we attach the decorated tree $\tau_j$ to the line at height $t - s_j$. To the top of the line we attach the random variable $X_0$. This procedure constructs a Yule tree of height $t$, such that each leaf is decorated with an i.i.d.\@ variable with law $\mu_0$ (see Figure \ref{fig:alternative_construction}(b)). Therefore, $X_t$ is obtained as the result of a tree-painting procedure, hence has law $\mu_t$. We conclude that $(X_t, t \geq 0)$ is a strong solution to \eqref{eqn:DRsde}.
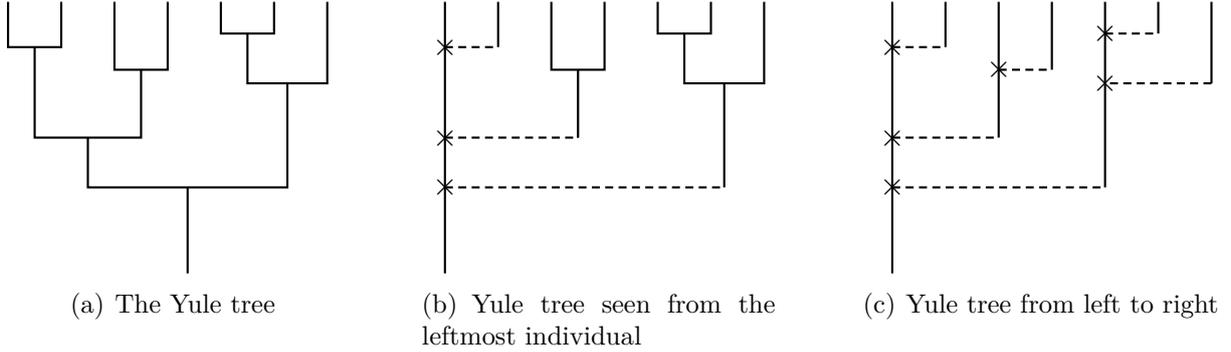
\begin{figure}[ht]
\centering
\subfigure[The Yule tree]{%
\begin{tikzpicture}[xscale=0.7,yscale=0.6]
  \draw [thick] (0,5) -- (0,4) -- (1,4) -- (1,5);
  \draw [thick] (2,5) -- (2,3.5) -- (3,3.5) -- (3,5);
  \draw [thick] (4,5) -- (4,4.3) -- (5,4.3) -- (5,5);
  \draw [thick] (6,5) -- (6,3.2) -- (4.5,3.2) -- (4.5,4.3);
  \draw [thick] (2.5,3.5) -- (2.5,2) -- (0.5,2) -- (0.5,4);
  \draw [thick] (5.25, 3.2) -- (5.25,0.9) -- (1.5,0.9) -- (1.5,2);
  \draw [thick] (3.375,.9) -- (3.375,-1);
\end{tikzpicture}
}%
\hfill%
\subfigure[Yule tree seen from the leftmost individual]{%
\begin{tikzpicture}[xscale=0.7,yscale=0.6]
  \draw [thick] (0,5) -- (0, -1);

  \draw (0,4) node{$\times$};
  \draw [thick, densely dashed] (0,4) -- (1,4);
  \draw [thick] (1,4) -- (1,5);

  \draw (0,2) node{$\times$};
  \draw [thick, densely dashed] (0,2) -- (2.5,2);
  \draw [thick] (2.5,3.5) -- (2.5,2);
  \draw [thick] (2,5) -- (2,3.5) -- (3,3.5) -- (3,5);
  
  \draw (0,0.9) node{$\times$};
  \draw [thick, densely dashed] (0,0.9) -- (5.25,0.9);
  \draw [thick] (4,5) -- (4,4.3) -- (5,4.3) -- (5,5);
  \draw [thick] (6,5) -- (6,3.2) -- (4.5,3.2) -- (4.5,4.3);
  \draw [thick] (2.5,3.5) -- (2.5,2);
  \draw [thick] (5.25, 3.2) -- (5.25,0.9);
\end{tikzpicture}
}%
\hfill%
\subfigure[Yule tree from left to right]{
\begin{tikzpicture}[xscale=0.7,yscale=0.6]
  \draw [thick] (0,5) -- (0, -1);

  \draw (0,4) node{$\times$};
  \draw [thick, densely dashed] (0,4) -- (1,4);
  \draw [thick] (1,4) -- (1,5);

  \draw (0,2) node{$\times$};
  \draw [thick, densely dashed] (0,2) -- (2,2);
  \draw [thick] (2,2) -- (2,5);
  \draw (2,3.5) node{$\times$};
  \draw [thick, densely dashed] (2,3.5) -- (3,3.5);
  \draw [thick] (3,3.5) -- (3,5);

  \draw (0,0.9) node{$\times$};
  \draw [thick, densely dashed] (0,0.9) -- (4,0.9);
  \draw [thick] (4,0.9) -- (4,5);
  \draw (4,4.3) node{$\times$};
  \draw [thick, densely dashed] (4,4.3) -- (5,4.3);
  \draw [thick] (5,4.3) -- (5,5);
  \draw (4,3.2) node{$\times$};
  \draw [thick, densely dashed] (6,3.2) -- (4,3.2);
  \draw [thick] (6,5) -- (6,3.2);
\end{tikzpicture}
}
\caption{Two alternative constructions of a Yule tree}
\label{fig:alternative_construction}
\end{figure}

We now prove strong uniqueness of the solution to \eqref{eqn:DRsde}. We first show uniqueness in law, from which we will immediately deduce pathwise uniqueness. Let $(X_t, t \geq 0)$ be a solution to \eqref{eqn:DRsde}, for all $t \geq 0$, we denote by $\nu_t$ the law of $X_t$, and we define $H_t(u) = \inf\{ y \geq 0 : \nu_t([0,y]) > u\}$. We first remark that there is strong existence and uniqueness of the solution of
\begin{equation}
  \label{eqn:mark2}
  X_t = X_0 - \int_0^t \ind{X_s > 0} \diff s + \int_0^t \int_0^1 H_s(u) N(\diff s,\diff u),
\end{equation}
and that $X$ is this solution.

Let $t > 0$, our aim is to prove that $\nu_t = \mu_t$. For this, we show that, using the same procedure as above, we can represent $X_t$ as the result of a tree-painting procedure on a Yule tree of height $t$, with leaves decorated with law $\mu_0$. Indeed, we denote once again $(s_j, u_j)_{1 \leq j \leq n}$ the atoms of the point process $N$ on $[0,t] \times [0,1]$ and $\tilde{X}_j \coloneqq H_{s_j}(u_j)$.
Observe that, conditionally on $(s_j)_{1\leq j \leq n}$, each $\tilde{X}_j$ is a copy of $X_{s_j}$, independent of $(\tilde{X}_i)_{i \neq j}$ and of $X_0$. We construct the tree of height $t$ as follows. We start with a line of height $t$, and for each $j \leq n$, we add a mark on this line at height $t-s_j$. At each of these marks is added the random variable $\tilde{X}_j$. 
As $\tilde{X}_j$ is an independent copy of $X_{s_j}$, it can be itself constructed as the unique strong solution of Equation \eqref{eqn:mark2}, where the point process $N$ and the initial value $X_0$ are replaced by i.i.d.\@ copies $N^{(j)}$ and $X_0^{(j)}$.

Hence, we can construct new lines and marks starting from each of the marks in the original line and proceed so on recursively. This constructs a Yule tree, defined in a slightly different way: indeed, in this tree, each line lives until time~$t$, giving independently birth to children at rate $1$  (this is the construction pictured in Figure \ref{fig:alternative_construction}(c)). Moreover, on each of the leaves at level $t$ are placed i.i.d.\@ random variables with law $\mu_0$, corresponding to the starting position for each of the stochastic differential equation. Therefore, $X_t$ is once again constructed as a tree-painting procedure on a Yule tree of height~$t$ and we can conclude that $\nu_t = \mu_t$. 

In particular, this proves that $H = G$ and, thus, that $X$ is a solution to the SDE
\[
  X_t = X_0 - \int_0^t \ind{X_s > 0} \diff s + \int_0^t \int_0^1 G_s(u) N(\diff s,\diff u),
\]
for which there is pathwise uniqueness of the solution associated to a given $X_0$ and $N$. Hence, given $X_0$ and $N$, there exists a unique strong solution to \eqref{eqn:DRsde}, and its unidimensional marginals form a  family  of laws  obtained from the tree-painting scheme.
\end{proof}

This result proves there is equivalence between the family of laws defined via the tree-painting scheme, or by the McKean-Vlasov type equation. To complete the proof of Theorem \ref{t:main1}, it is therefore enough to show that this family of law is the unique solution of the partial differential equation \eqref{eqn:DRsde}, which is done in the next section.

\subsection{The partial differential equation representation}
\label{sec:pde}

We observe here that a family of laws obtained from the tree-painting scheme can be rewritten as the weak solution of a partial differential equation, the key being the application of It\^o's formula to the solution of the stochastic differential equation \eqref{eqn:DRsde}. We then study the existence and uniqueness of the solutions of this partial differential equation using the tree-painting scheme, and in particular Lemma~\ref{lem:recursiveEquation}.

As in Definition \ref{def:drm}, we say that a family $(\nu_t)_{ t \geq 0}$ of finite measures on $\R$ is a weak solution to the partial differential equation
\begin{equation}
  \label{eqn:pde}
  \partial_t \nu_t(\diff x) = \partial_x (\ind{x>0}\nu_t(\diff x)) + (\nu_t \ast \nu_t)(\diff x) - \nu_t(\diff x),
\end{equation}
if $t \mapsto \nu_t$ is measurable and, for all $\mathcal{C}^1$ bounded functions $f$ with bounded derivative, we have
\begin{align}
  \label{eqn:pdeWeak}
  & \int_\R f(x) \nu_t(\diff x) - \int_\R f(x) \nu_0(\diff x) \nonumber \\
  &  = \int_0^t\left( - \int_\R (f'(x) \ind{x>0} + f(x)) \nu_s(\diff x) + \int_{\R^2} f(x+y)\nu_s(\diff x)\nu_s(\diff y)\right) \diff s. 
\end{align}
In particular, if $(\nu_t, t \geq 0)$ is a family of probability measures, this equation can be rewritten
\begin{align}
  \label{eqn:pdeWeak2}
  & \int_\R f(x) \nu_t(\diff x) - \int_\R f(x) \nu_0(\diff x) \nonumber \\
  & = \int_0^t\left( - \int_\R f'(x) \ind{x>0} \nu_s(\diff x) + \int_{\R^2} (f(x+y) - f(x)) \nu_s(\diff x)\nu_s(\diff y) \right) \diff s. 
\end{align}

The following result proves existence and uniqueness of weak solutions of \eqref{eqn:pde} in the case where the initial condition is a probability measure on $\R_+$. 
In other word, there exists a unique (continuous-time) DR model with initial distribution $\mu_0$, for all probability measure $\mu_0$ on $\R_+$.

\begin{proposition}
\label{prop:defedp}
Let $(\mu_t)_{t\ge 0}$  be a  family  of laws  obtained from the tree-painting scheme, then $(\mu_t)_{t\ge 0}$  is a weak solution to \eqref{eqn:pde}.

Reciprocally, if  $(\nu_t)_{t\ge 0}$ is a family of non-negative finite measures on $\R$ which is a weak solution to \eqref{eqn:pde} and $\nu_0$ is a probability distribution on $\R_+$, then  $(\nu_t)_{t\ge 0}$ is a  family  of laws  obtained from the tree-painting scheme.
\end{proposition}

\begin{proof}
From Proposition \ref{prop:defMcKean}, we observe that if $(\mu_t)_{t\ge 0}$  is a  family  of laws  obtained from the tree-painting scheme, we can construct a Markov process $(X_t)_{ t \geq 0}$ solution of \eqref{eqn:DRsde} such that the law of $X_t$ is $\mu_t$ for all $t \geq 0$. Therefore, for all $\mathcal{C}^1$ bounded function $f$, we have
\begin{align*}
  \E\left[ f(X_t) \right]
  &= \E\left[ f(X_0) - \int_0^t f'(X_s) \ind{X_s>0} \diff s\right]\\
  &\relphantom{=} {} + \E\left[\int_0^t \int_0^1 (f(X_{s-} + F_s^{-1}(u)) - f(X_{s-})) N(\diff s,\diff u) \right],
\end{align*}
by the It\^o formula. Therefore, using the linearity of the expectation and the Campbell's formula, we obtain
\[
 \E[f(X_t)] - \E[f(X_0)] = \int_0^t (-\E[f'(X_s) \ind{X_s>0}] + \E[f(X_s+Y_s)-f(X_s)]) \diff s,
\]
with $Y_s$ an independent copy of $X_s$, proving that $(\mu_t)_{t \geq 0}$ is a weak solution of \eqref{eqn:pde}.

Reciprocally, let $\nu_0$ be a probability distribution on $\R$ and $(\nu_t)_{t \geq 0}$ be a weak solution of \eqref{eqn:pde} (in particular, $\nu_t$ is finite). We first observe that for each $t > 0$, the total mass of $\nu_t$ remains equal to $1$. Indeed, using $f=1$, we have
\[
  \nu_t(\R) - \nu_0(\R) = \int_0^t (\nu_s(\R)^2 - \nu_s(\R)) \diff s.
\]
Hence, the total mass is the solution of the differential equation $y' = -y(1-y)$ with initial condition $y(0)=1$, which is $y(t) = 1$ for all $t \geq 0$, proving the total mass is preserved in \eqref{eqn:pde}. We conclude that $(\nu_t, t \geq 0)$ is a family of probability distributions.

We now prove that if $\nu_0$ put no mass on $(-\infty,0)$ then so does $\nu_t$ for all $t > 0$. To do so, we observe that for all $\mathcal{C}^1$ non-negative bounded function $f$ with support in $(-\infty,0)$, we have
\[
  \int f(x) \nu_t(\diff x)
   = \int_0^t \int (f(x+y)-f(x)) \nu_s(\diff x)\nu_s(\diff y) \diff s \geq - \int_0^t \int f(x) \nu_s(\diff x) \diff s,
\]
hence the function $t \mapsto\int f(x) \nu_t(\diff x)$ remains equal to $0$ at all times by Gronwall lemma. We conclude that $\nu_t((-\infty,0)) = 0$ for all $t > 0$.

Finally, we prove that $(\nu_t)_{t \geq 0}$ satisfies  Equation \eqref{eqn:recursiveEquation}. To do so, let $g$ be a $\mathcal{C}^1$ function with compact support in $(0,\infty)$. We define the function
\[
  G_g(a,b) \coloneqq \int g((x-a)_+) \nu_b(\diff x).
\]
By dominated convergence and \eqref{eqn:pdeWeak}, we observe that $G_g$ is a $\mathcal{C}^1$ function, and that
\begin{align*}
  \partial_a G_g(a,b) &= -\int g'(x-a) \ind{x > a} \nu_b(\diff x)\\
  \partial_b G_g(a,b) &= - \int \left(g'(x-a) \ind{x>a} + g((x-a)_+)\right) \nu_b(\diff x) \\
  & \relphantom{=} {} + \int g((x+y-a)_+) \nu_b(\diff x) \nu_b(\diff y).
\end{align*}
Combining both equations, it follows that
\[
  \int g(z) \tau_a(\nu_b\ast \nu_b)(\diff z) = \partial_b G_g(a,b) - \partial_a(G_g(a,b)) + \int g((x-a)_+) \nu_b(\diff x).
\]
and, therefore, we get
\begin{align*}
  G_g(0,t) - \e^{-t}G_g(t,0)
  &= \int_0^t \partial_s(\e^{-(t-s)}G_g(t-s,s)) \diff s\\
  &= \int_0^t \int_0^t \e^{-(t-s)} \left(G_g(t-s,s) - \partial_a G_g(t-s,s) + \partial_b G_g(t-s,s) \right)\diff s\\
  &= \int_0^t \e^{-s} g(z) \tau_s(\nu_{t-s} \ast \nu_{t-s})(\diff z) \diff s,
\end{align*}
using the change of variables $s \to t-s$ in the last line. This proves that $(\nu_t)_{t \geq 0}$ satisfies \eqref{eqn:recursiveEquation}, hence by Lemma \ref{lem:recursiveEquation}, this is a  family  of laws  obtained from the tree-painting scheme, concluding the proof.
\end{proof}

\begin{remark}
\label{rem:pde}
Note that by definition, there exists a unique family  $(\mu_t)_{t\ge 0}$  of laws  obtained from the tree-painting scheme and  associated to a given initial probability measure $\mu_0$. As a result, Proposition \ref{prop:defedp} shows there exists a unique family of probability measures which are weak solutions of \eqref{eqn:pde}.
\end{remark}

\begin{proof}[Proof of Theorem \ref{t:main1}]
The proof  follows immediately from  Propositions \ref{prop:defMcKean} and \ref{prop:defedp}.
\end{proof}

Thanks to Proposition \ref{prop:defedp}, we observe that if we can find a solution to \eqref{eqn:pde}, then this family of probability measures is a DR model. Despite obtaining few information from \eqref{eqn:pde} (for example, we were not able to prove the uniqueness of the solution in the space of distributions, or that a weak solution starting from a smooth initial condition remains smooth at all times), we were able to exhibit a set of measures, which is stable with respect to the continuous-time DR transformation.

We note that strong solutions to \eqref{eqn:pde} can be described in the following way. Let $p, \phi$ be two $\mathcal{C}^1$ functions, the family defined by
\[
  \mu_t(\diff x) = p(t) \delta_0(\diff x) + (1-p(t)) \phi(t,x) \diff x
\]
is a DR model if and only if $p(0) \in [0,1]$, $\phi(0,x) \geq 0$ and $\int \phi(0,x) \diff x =1$, and
\begin{equation} \label{edppetphi}
\addtolength\arraycolsep{-0.125cm}
  \left\{
  \begin{array}{rl}
    p'(t) & = (1 - p(t))(\phi(t,0) -p(t))\\
    \partial_t \phi(t,\cdot) & = \partial_x \phi(t,\cdot) +(1-p(t)) \phi(t,\cdot) \ast \phi(t,\cdot) + (p(t) +\phi(t, 0) - 1)\phi(t,\cdot).
  \end{array}
  \right.
\end{equation}
Finding solutions to this differential equation, or proving existence, uniqueness and regularity for the solutions appears to be an interesting open problem.

\section{Exactly solvable Derrida--Retaux model}\label{s:casexponentiel}

In this section, we study a continuous-time DR model starting from an initial measure
\begin{equation}
  \label{eqn:initialDistribution}
  \mu_0(\diff x) = p(0) \delta_0(\diff x) + (1 - p(0)) \lambda(0) \e^{-\lambda(0) x} \diff x, \quad p(0) \in [0,1], \lambda(0)>0, 
\end{equation}
which is a mixture of a Dirac mass at 0 and an exponential random variable (on $\R_+$). It turns out that the evolution of the DR model preserves the class of mixtures of exponential random variables and Dirac masses at $0$. In other words, for all $t > 0$, there exists $p(t) \in [0,1]$ and $\lambda(t)>0$ such that
\begin{equation}\label{eqn:mut}
  \mu_t(\diff x) = p(t) \delta_0(\diff x) + (1 -p(t)) \lambda(t) \e^{-\lambda(t) x} \diff x,
\end{equation}
and the functions $(p,\lambda)$ solve a system of ordinary differential equations. More precisely, the following result holds.
\begin{lemma}
\label{lem:systDiff}
The family $(\mu_t)_{t \geq 0}$ is a DR model if and only if $(p,\lambda)$ is solution of the system of ordinary differential equations
\begin{equation}
\label{eqn:systDiff}
  \addtolength\arraycolsep{-0.125cm}
  \left\{
  \begin{array}{rl}
    p' &= (1 - p)( \lambda - p)\\
    \lambda' &= - \lambda (1-p).
  \end{array}
  \right.
\end{equation}
\end{lemma}

\begin{proof}
By Proposition \ref{prop:defedp}, there exists a unique DR model with initial distribution $\mu_0$  given by \eqref{eqn:initialDistribution}. Therefore, 
letting $(p,\lambda)$ denote the maximal solution of the system of ordinary differential equations \eqref{eqn:systDiff} with initial condition $(p(0),\lambda(0))$ (unique by Cauchy--Lipschitz theorem), it is enough to prove that this solution is well-defined for all $t \geq 0$ and that $(\mu_t)_{t \geq 0}$ is a DR model.

We first observe that if $(p,\lambda)$ solve \eqref{eqn:systDiff}, then $\lambda$ is non-increasing, and if $\lambda(0) > 0$ then $\lambda(t) > 0$ as long as the solution is well-defined. Similarly, if $p(0) \leq 1$ then $p(t) \leq 1$, and as $p'(t) \geq - (1-p(t))p(t)$, $p(t) \geq 0$ for all $t$ such that the solution is well-defined. As a result, we deduce that $(p(t),\lambda(t)) \in [0,1] \times [0,\lambda_0]$. The solution remaining in a compact set, it cannot explode in finite time, therefore $(p,\lambda)$ is well-defined for all $t \geq 0$.

We now prove that $(\mu_t)_{t \geq 0}$ is a DR model. By Proposition \ref{prop:defedp}, it is enough to prove that the function is a weak solution to \eqref{eqn:pde}. Let $f$ be a $\mathcal{C}^1$ bounded function with a bounded derivative. We define
\[
  g(t) = \int_{\R_+} f(x) \mu_t(\diff x) 
  = p(t) f(0) + (1 - p(t)) \lambda(t) \int_0^\infty f(x) \e^{-\lambda(t) x} \diff x.
\]
We write $E_0(t) \coloneqq \lambda(t) \int_0^\infty f(x) \e^{-\lambda(t) x} \diff x$ and $E_1(t) \coloneqq \lambda(t) \int_0^\infty x f(x) \e^{-\lambda(t) x} \diff x$. By dominated convergence, we observe that, on the one hand, the derivative of $g$ is given by
\begin{align*}
  g' & = p' f(0) + \left(- p' + (1 - p) \tfrac{\lambda'}{\lambda}\right) E_0 - (1 - p) \lambda' E_1 \\
  & = (1-p) (\lambda-p) f(0) - (1-p)(1 + \lambda - 2p) E_0 + (1-p)^2 \lambda E_1.
\end{align*}
On the other hand, simple computations yield
\begin{align*}
  \int_{\R_+} (f'(x) \ind{x>0} + f(x)) \mu_t(\diff x)
  & = p(t) f(0) + \lambda(t)(1 - p(t)) \int_0^\infty (f'(x)  + f(x)) \e^{-\lambda(t) x} \diff x \\
  &= (p(t) - \lambda(t)(1-p(t))) f(0) + (1-p(t))( 1 + \lambda(t)) E_0(t),
\end{align*}
by integration by part, and
\begin{align*}
  &\int_{(\R_+)^2} f(x+y) \mu_t(\diff x) \mu_t(\diff y)\\ 
  & = p(t)^2 f(0) + 2 (1-p(t)) E_0(t) + (1 -p(t))^2 \lambda(t)^2 \int_{(\R_+)^2} f(x+y)\e^{-\lambda(t) (x+y)} \diff x \diff y\\
  & = p(t)^2 f(0) + 2p(t) (1-p(t)) E_0(t) + (1 -p(t))^2 \lambda(t)E_1(t).
\end{align*}
As a result, we obtain that
\begin{align*}
  - \int_{\R_+} (f'(x) \ind{x>0} + f(x)) \mu_t(\diff x) + \int_{(\R_+)^2} f(x+y) \mu_t(\diff x) \mu_t(\diff y) = g'(t)
\end{align*}
showing that $(\mu_t)_{t \geq 0}$ is a DR model, which concludes the proof.
\end{proof}

As a result of Lemma \ref{lem:systDiff}, to study the DR model with initial law $\mu_0$ given in \eqref{eqn:initialDistribution}, it is enough to study the system of ordinary differential equations \eqref{eqn:systDiff}. This is what is done in the rest of the section, allowing us in particular to prove the DR conjecture in this particular setting.

\begin{remark}
\label{rem:systDiff}
Observe that, similarly, the evolution of a DR model starting by any mixture of Dirac mass at $0$ and exponential random variables at time $0$ can be represented by a system of differential equations. More precisely, for any $n \in \N$ the measure
\[
  \nu_t(\diff x) = p(t) \delta_0 + (1 - p(t))  \sum_{j=1}^n q_j(t) \lambda_j(t) \e^{-\lambda_j(t)x} \diff x
\]
is a DR model if and only if the functions $p$, $q_j$, $\lambda_j$ satisfy
\begin{equation}
  \addtolength\arraycolsep{-0.125cm}
  \left\{
  \begin{array}{rl}
    p' & = (1-p) \left( \sum_{k=1}^n q_k \lambda_k - p \right) \\
    q_j' & =- q_j \left(  (1-p) (1- q_j + 2 \sum_{k \neq j} \frac{q_k \lambda_k}{\lambda_j - \lambda_k} ) + \lambda_j - \sum_{k=1}^n q_k \lambda_k \right)\\
    \lambda_j' & = -(1-p) q_j \lambda_j,
  \end{array}
  \right. 
\end{equation}
and if $p(0),q_j(0) \in [0,1]$, $\lambda_j(0) \geq 0$ and $\sum_{j=1}^n q_j(0) = 1$. However, we concentrate on the simple mixture of a Dirac mass at $0$ and one exponential random variable in this article, as this simple model exhibits the main features believed to hold for any DR model.
\end{remark}

In the next section, we give the phase diagram of the differential system \eqref{eqn:systDiff}, with a description of all stable and unstable equilibrium points, as well as their zone of attraction. In Section~\ref{subsec:asymptoticBehaviourSystDiff}, we refine the study by giving asymptotic equivalents for the solutions of \eqref{eqn:systDiff}. Finally, we use these results to prove in Section~\ref{subsec:proofDR} the DR conjecture for continuous models with exponential initial condition.

\subsection{Phase diagram of the model}

We first observe that some quantities are preserved under the dynamics \eqref{eqn:systDiff}.
\begin{lemma}
\label{lem:integrable}
The function $t \mapsto \frac{p(t)}{\lambda(t)} + \log \lambda(t)$ is a constant.
\end{lemma}

\begin{proof}
We compute the derivative of the function $H :t \mapsto  \frac{p(t)}{\lambda(t)} + \log \lambda(t)$, we have
\begin{align*}
    H'(t) &= \frac{p'(t)}{\lambda(t)} - \frac{p(t) \lambda'(t)}{\lambda(t)^2} + \frac{\lambda'(t)}{\lambda(t)}\\
    &= \frac{(1-p(t))(\lambda(t) - p(t))}{\lambda(t)} + \frac{p(t)\lambda(t) (1-p(t))}{\lambda(t)^2} - \frac{\lambda(t)(1-p(t))}{\lambda(t)} = 0,
\end{align*}
which concludes the proof.
\end{proof}

As a straightforward consequence of this remark, we deduce that $p$ can be written as an explicit function of $\lambda$, for all initial condition.
\begin{corollary}
\label{cor:differentialEquation}
Let $H \coloneqq \frac{p(0)}{\lambda(0)} + \log \lambda(0)$, for all $t \geq 0$, we have
\begin{equation} \label{eq:p_as_function_of_lambda}
p(t) = H \lambda(t) - \lambda(t) \log \lambda(t).
\end{equation}
In particular, $\lambda$ is the solution of the ordinary differential equation
\begin{equation}
  \label{eqn:diffEquation}
  \lambda' = - \lambda(1 - H \lambda + \lambda \log \lambda).
\end{equation}
\end{corollary}

Using Lemma \ref{lem:integrable}, we can draw the phase diagram of the differential system \eqref{eqn:systDiff}, obtain the asymptotic behavior of $(p,\lambda)$ as $t \to \infty$ and describe all the equilibrium points of \eqref{eqn:systDiff}, with their stability properties.

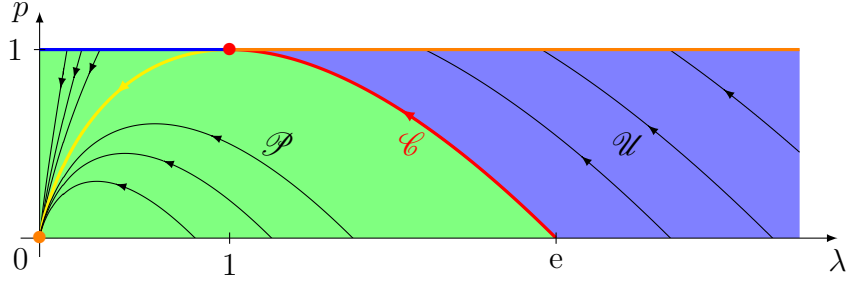
\begin{figure}[ht]
\centering
\begin{tikzpicture}[scale=2.5]
  \fill [color=green!50] (0,0) -- (0,1) -- (1,1) -- plot[domain = 1:2.71] (\x,{\x - \x*ln(\x)}) -- (2.71,0) -- cycle;
  \fill [color =blue!50] (1,1) -- plot[domain = 1:2.71] (\x,{\x - \x*ln(\x)}) -- (2.71,0) -- (4,0) -- (4,1) -- cycle;

  \draw [domain = 0.001:0.261,samples = 200] plot(\x,{2*\x-\x *ln(\x)});
  \draw [<-,>=latex,domain = 0.26:0.35,samples = 200] plot(\x,{2*\x-\x *ln(\x)});
  \draw [domain = 0.001:0.181,samples = 200] plot(\x,{3*\x-\x *ln(\x)});
  \draw [<-,>=latex,domain = 0.18:0.25,samples = 200] plot(\x,{3*\x-\x *ln(\x)});
  \draw [domain = 0.001:0.111,samples = 200] plot(\x,{5*\x-\x *ln(\x)});
  \draw [<-,>=latex,domain = 0.11:0.15,samples = 200] plot(\x,{5*\x-\x *ln(\x)});
  
  \draw [domain = 3.2:3.61,samples = 200] plot(\x,{1.5*\x-\x *ln(\x)});
  \draw [<-,>=latex,domain = 3.6:4,samples = 200] plot(\x,{1.5*\x-\x *ln(\x)});
  \draw [domain = 2.5:3.21,samples = 200] plot(\x,{1.35*\x-\x *ln(\x)});
  \draw [<-,>=latex,domain = 3.2:3.9,samples = 200] plot(\x,{1.35*\x-\x *ln(\x)});
  \draw [domain = 2:2.86,samples = 200] plot(\x,{1.2*\x-\x *ln(\x)});
  \draw [<-,>=latex,domain = 2.85:3.4,samples = 200] plot(\x,{1.2*\x-\x *ln(\x)});
  
  \draw [color=red, very thick, domain = 1: 2.71828, samples=200] plot (\x,{\x-\x*ln(\x)}); 
  \draw [<-,>=latex,thick,color=red,domain=1.9:2.71828,samples=200] plot (\x,{\x-\x*ln(\x)}); 
  \draw [color=yellow, very thick, domain = 0.001:1, samples=200] plot (\x,{\x-\x*ln(\x)});
  \draw [<-,>=latex,thick,color=yellow,domain=0.4:1,samples=200] plot (\x,{\x-\x*ln(\x)});  
  
  \draw [domain = 0.001:0.91,samples = 200] plot(\x,{0.5*\x-\x *ln(\x)});
  \draw [<-,>=latex,domain = 0.9:1.7,samples = 200] plot(\x,{0.5*\x-\x *ln(\x)});
  \draw [domain = 0.001:0.651,samples = 200] plot(\x,{0.2*\x-\x *ln(\x)});
  \draw [<-,>=latex,domain = 0.65:1.3,samples = 200] plot(\x,{0.2*\x-\x *ln(\x)});
  \draw [domain = 0.001:0.41,samples = 200] plot(\x,{-0.2*\x-\x *ln(\x)});
  \draw [<-,>=latex,domain = 0.4:0.9,samples = 200] plot(\x,{-0.2*\x-\x *ln(\x)});

  \fill[color=white] (0.1,1.001)--(3.5,1.001) -- (3.5,1.2) -- (0.1,1.2) -- cycle;
  \fill[color=white] (0.1,-0.001)--(4,-0.001) -- (4,-0.1) -- (0.1,-0.1) -- cycle;

  \draw [->,>=latex] (-0.1,0) -- (4.2,0) node[below] {$\lambda$};
  \draw (0,0) node[below left] {0};
  \draw [->,>=latex] (0,-0.1) -- (0,1.2) node[left] {$p$};
  \draw [very thick ,color=orange] (1,1) -- (4,1);
  \draw [very thick ,color=blue] (0,1) -- (1,1);
  \draw [color=orange] (0,0) node {$\bullet$};
  \draw [color=red] (1,1) node {$\bullet$};
  
  \draw (0.03,1) -- (-0.03,1) node[left]{$1$};
  \draw (1,0.03) -- (1,-0.03) node[below]{$1$};
  \draw (2.71828,0.03) -- (2.71828,-0.03) node[below]{$\mathrm{e}$};
  
  \draw (1+0.25, 0.5) node {$\mathscr{P}$};
  
  \draw[color=red] (1.96,0.5) node {$\mathscr{C}$};
  
  \draw (3.3-0.2,0.5) node{$\mathscr{U}$};
\end{tikzpicture}
\caption{The phase diagram of the differential system \eqref{eqn:systDiff}}
\label{fig:phaseDiagramPasIntro}
\end{figure}

Recall the three sets $ {\mathscr P},  {\mathscr C}$ and ${\mathscr U}$ defined in the introduction, respectively the supercritical, critical and subcritical zones, which we recall on Figure \ref{fig:phaseDiagramPasIntro}. We added on Figure \ref{fig:phaseDiagramIntro}, the trajectories of some solutions of the equation (in black). Some interesting additional sets are represented on that figure, such as the supercritical branch (in yellow) of the critical curve
\[
  \widehat{{\mathscr C}} = \left\{ (p,\lambda) \in [0,1) \times [0,1) : p = \lambda - \lambda \log \lambda \right\},
\]
as well as the attractive, critical and repulsive points of the dynamics (in orange, red and dark blue respectively):
\[
 {\mathscr A}_p = \{(0,0)\} \cup \left(\{1\} \times (1,\infty)\right), \quad {\mathscr C}_p = \{(1,1)\} \quad \text{and} \quad {\mathscr R}_p = \{1\} \times [0,1) .
\]

We now describe the asymptotic behavior of this DR model as $t \to \infty$.
\begin{proposition}
\label{thm:asympBehaviourIntro}
Let $(p,\lambda)$ be a solution of \eqref{eqn:systDiff} on $\R_+$. If $(p(0),\lambda(0)) \in [0,1] \times [0,\infty)$, then the domain of definition of $(p,\lambda)$ includes $[0,\infty)$, and moreover,
\begin{enumerate}
  \item if $(p(0),\lambda(0)) \in {\mathscr P}$, then $\lim_{t \to \infty} (p(t),\lambda(t)) = (0,0)$.
  \item if $(p(0),\lambda(0)) \in {\mathscr U}$, then $\lim_{t \to \infty} (p(t),\lambda(t)) = (1,x)$, where $x$ is the unique solution larger than 1 of the equation $Hx - x\log x = 1$.
  \item if $(p(0),\lambda(0)) \in {\mathscr C}$, then $\lim_{t \to \infty} (p(t),\lambda(t)) = (1,1)$.
\end{enumerate}
\end{proposition}

\begin{proof}
We first remark from \eqref{eqn:systDiff} that the function $\lambda$ is non-increasing. 
 Moreover, we observe the following behavior for the solutions of \eqref{eqn:systDiff} starting from the boundary of the domain:
\begin{itemize}
  \item if $p(0)=1$, then $(p(t), \lambda(t)) = (1,\lambda(0))$ for all $t \in \R$,
  \item if $\lambda(0)=0$ and $p(0) \in (0,1)$, then $(p(t),\lambda(t)) = (\frac{1}{(p(0)^{-1}-1)\e^t +1},0)$ for all $t \in \R$,
  \item if $\lambda(0)=p(0)=0$ then $(p(t),\lambda(t)) = (0,0)$ for all $t \in \R$,
  \item if $\lambda(0)>0$ and $p(0) = 0$, then for all $h>0$ small enough, we have $(p(h),\lambda(h)) \in (0,1) \times (0,\infty)$.
\end{itemize}
The possible values of $\lim_{t \to \infty} (p(t),\lambda(t))$ then belong to the set of equilibrium points of the dynamics, defined by the system of equation
\[
\addtolength\arraycolsep{-0.125cm}
  \left\{
  \begin{array}{rl}
    0 &= (1-p)( \lambda - p)\\
    0 &= - \lambda (1-p),
  \end{array}
  \right.
\]
which is exactly ${\mathscr A}_p \cup {\mathscr C}_p \cup {\mathscr R}_p$. Moreover, recall that $(p(t),\lambda(t))$ move along the trajectory of $x \mapsto Hx - x \log x$, as $\lambda(t)$ decreases.

Hence, if $(p(0),\lambda(0)) \in {\mathscr P}$, then either $H < 1$, and the only equilibrium point belonging to $x \mapsto Hx - x \log x$ is $(0,0)$, or $H>1$ and this is the only equilibrium point left of the starting point. 
If $(p(0),\lambda(0))$ in ${\mathscr U}$, the first equilibrium point met by the trajectory is the one described in the second point.
Finally, if $(p(0),\lambda(0)) \in {\mathscr C}$, then the trajectory of $(p,\lambda)$ would go through the critical point $(1,1)$ which is impossible, proving the third case. 
\end{proof}

\begin{remark}
Note that $(p(0),\lambda(0)) \in [0,1] \times [0,1]$ and $H\geq 1$ is a necessary and sufficient condition for the existence of a meaningful solution of \eqref{eqn:diffEquation} defined on~$\R$. From any other starting position, there exists a finite $t \in (-\infty,0)$ such that $p(t) < 0$, therefore the DR model cannot be extended as a process on~$\R$ in general.
\end{remark}

\subsection{Asymptotic behavior of the solutions}
\label{subsec:asymptoticBehaviourSystDiff}

In this section, we study the asymptotic behavior of the solution $(p(t), \lambda(t))$ as $t \to \infty$, in the three different phases.    In the critical case, this corresponds to Conjectures 1.3 and 1.4 of Chen et al.\@  \cite{CDHLS17} in the discrete setting.
 Using the properties of the differential system, we were able to study the asymptotic behavior in the critical case beyond the computation of an equivalent. 
\begin{proposition} \label{prop:asymptotics_of_p_and_lambda}
Let $(p,\lambda)$ be a solution of \eqref{eqn:systDiff}.
\begin{enumerate}
\item If $(p(0),\lambda(0)) \in {\mathscr P}$ with $\lambda(0) > 0$, then there exists $K > 0$ such that, as $t \to \infty$,
\[
  \lambda(t) \sim K \e^{-t} 
  \quad \text{and} \quad 
  p(t) \sim Kt\e^{-t}.
\] 
\item If $(p(0),\lambda(0)) \in {\mathscr U}$, then there exists $K>0$ such that, as $t \to \infty$,
\[
  \lambda(t) - x \sim K \e^{-(x-1)t} 
  \quad \text{and} \quad 
  1-p(t) \sim \left( 1 - \frac{1}{x} \right) K \e^{-(x-1)t},
\]
where $x$ is the unique solution larger than 1 of the equation $Hx - x\log x = 1$.
\item If $(p(0),\lambda(0)) \in {\mathscr C}$, then we have, as $t \to \infty$,
\[
  \lambda(t) = 1 + \frac{2}{t} - \frac{8 \log t}{3 t^2} + o\left( \frac{\log t}{t^2} \right) 
  \quad \text{and} \quad
  p(t) = 1 - \frac{2}{t^2} + \frac{16 \log t}{3 t^3} + o\left( \frac{\log t}{t^3} \right).
\]
\end{enumerate}
\end{proposition}

\begin{proof}
Let start with the supercritical case $(p(0),\lambda(0)) \in {\mathscr P}$. Recall from Theorem \ref{thm:asympBehaviour} that $\lambda(t) \to 0$. Using \eqref{eqn:diffEquation}, we get that $\lambda'(t)/\lambda(t) \to -1$ and therefore 
\[
  \log \lambda(t) - \log \lambda(0) 
  = \int_0^t \frac{\lambda'(s)}{\lambda(s)} \diff s 
  \sim -t,
\]
hence $\lambda(t) = \e^{-t + o(t)}$. Using again \eqref{eqn:diffEquation}, we get that 
\[
\frac{\lambda'(t)}{\lambda(t)} + 1 = - \lambda(t) \log \lambda(t) (1+o(1)) = -\e^{-t + o(t)},
\]
which is integrable at infinity. Therefore, it follows that
\[
\log \lambda(t) - \log \lambda(0) + t 
= \int_0^t \left( \frac{\lambda'(s)}{\lambda(s)} + 1 \right) \diff s 
\xrightarrow[t\to\infty]{} 
\int_0^\infty \left( \frac{\lambda'(s)}{\lambda(s)} + 1 \right) \diff s,
\]
where the last integral is finite. 
This proves that $\lambda(t) \sim K \e^{-t} $ for some constant $K > 0$. 
Then, the behavior of $p(t)$ follows directly from \eqref{eq:p_as_function_of_lambda}.

We now deal with the subcritical case $(p(0),\lambda(0)) \in {\mathscr U}$. 
Recall from Theorem \ref{thm:asympBehaviour} that $\lambda(t) \to x$, where $x > 1$ satisfies $Hx - x \log x = 1$. 
We introduce the auxiliary function $u(t) \coloneqq \lambda(t) - x$, which is positive for all $t \in \R_+$.
Using \eqref{eqn:diffEquation}, we get
\begin{align*}
u' 
& = -(x+u) (1 - H(x+u)+ (x+u) \log(x+u)) \\
& = -(x+u) \left(-Hu + x \log \left( \frac{x+u}{x} \right) + u\log(x+u) \right),
\end{align*}
using that $Hx - x \log x = 1$.
Since $u(t) \to 0$, we get the following asymptotic development
\begin{equation} \label{nl}
u'(t) 
= -x (-H+1+\log x) u(t) + O(u(t)^2)
= -(x-1) u(t) + O(u(t)^2).
\end{equation}
Proceeding as in the supercritical case, we first deduce from \eqref{nl} that $u(t) = \e^{-(x-1)t + o(t)}$ and, then, that $u(t) \sim K \e^{-(x-1)t}$.
Plugging this in \eqref{eq:p_as_function_of_lambda}, the behavior of $p(t)$ follows.

Finally, we deal with the critical case $(p(0),\lambda(0)) \in {\mathscr C}$.
Recall from Theorem \ref{thm:asympBehaviour} that $\lambda(t) \to 1$, so we introduce the auxiliary function $v(t) \coloneqq \lambda(t) - 1$, which is positive for all $t \in \R_+$.
Using \eqref{eqn:diffEquation} with here $H=1$, we get
\[
v' = -(1+v) (1 - (1+v)+ (1+v) \log(1+v)),
\]
and, using that $v(t) \to 0$, we get
\begin{align}
v'(t) = -\frac{v(t)^2}{2} -\frac{v(t)^3}{3} + O(v(t)^4). \label{nm}
\end{align}
First, we deduce from \eqref{nm} that $(1/v)'(t) \to 1/2$ as $t \to \infty$ and therefore,
\[
  \left( \frac{1}{v(t)} - \frac{1}{v(0)} \right)
  = \int_0^t \left( \frac{1}{v} \right)'(s) \diff s
 \sim \frac{t}{2}.
\]
It proves that $v(t) \sim 2/t$ as $t \to \infty$.
Then, using again \eqref{nm}, we have 
\[
  \left( \frac{1}{v} \right)'(t) - \frac{1}{2}
 = \frac{v(t)}{3} + o(v(t))
  \sim \frac{2}{3t},
\]
as $t\to\infty$, and it follows that
\[
\int_1^t \left( \left( \frac{1}{v} \right)'(s) - \frac{1}{2} \right) \diff s
\sim \int_1^t \frac{2}{3s} \diff s = \frac{2}{3} \log t,
\]
and, finally,
\begin{equation} \label{nn}
v(t) = \left( \frac{t}{2} + \frac{2}{3} \log t + o(\log t) \right)^{-1}
= \frac{2}{t} - \frac{8 \log t}{3 t^2} + o\left( \frac{\log t}{t^2} \right).
\end{equation}
Then, it follows from \eqref{eq:p_as_function_of_lambda} that 
\[
  p(t) = (1+v(t)) - (1+v(t)) \log (1+v(t)) = 1 - \frac{v(t)^2}{2} + O(v(t)^3),
\]
and plugging \eqref{nn} in this equation gives the desired result.
\end{proof}

\begin{proof}[Proof of Theorem \ref{thm:asympBehaviour}]
Recalling from \eqref{eqn:mut} that 
\[\mu_t(\diff x) = p(t) \delta_0(\diff x) + (1 -p(t)) \lambda(t) \e^{-\lambda(t) x} \diff x.\]
Note  that $\int_{\R_+} x \mu_t(\diff x) = \frac{1-p(t)}{\lambda(t)}$. Theorem \ref{thm:asympBehaviour} is now an immediate consequence of  the asymptotic behavior of $(p(t), \lambda(t))$ stated in Proposition \ref{prop:asymptotics_of_p_and_lambda}. 
\end{proof}

\subsection{Proof of the Derrida--Retaux conjectures}
\label{subsec:proofDR}

Using the fact that the DR model we are considering consists in a mixture of exponential random variables and Dirac masses at $0$ and is integrable, we obtain a straightforward proof of the DR conjecture for this model, which states that the free energy has an infinite order phase transition between the supercritical and the subcritical phases. We begin by observing that if we write
\[
  \mu_0(\diff x) 
  = p(0) \delta_0(\diff x) + (1 - p(0)) \lambda(0) \e^{-\lambda(0)x} \diff x,
\]
then the free energy can be rewritten
\begin{equation}
  \label{eqn:alternateExpression}
  F_\infty(\mu_0) = F_\infty(p(0),\lambda(0)) = \lim_{t \to \infty} \e^{-t} \frac{1-p(t)}{\lambda(t)}.
\end{equation}
\begin{remark}
The proof of Theorem \ref{thm:DRConjectureSimple} is crucially based on Corollary \ref{cor:differentialEquation}, which states that $\lambda$ satisfies an ordinary differential equation. 
In the case $\lambda \in [1,e)$, note that as $p(0)$ approaches $p_c$, the trajectory followed by $(p,\lambda)$ approaches closer and closer to the critical point $(1,1)$ (see Figure \ref{fig:phaseDiagramPasIntro}). Moreover, trajectories typically stay a long time around this critical point, as both $p'$ and $\lambda'$ become small (of order $(1-p)$). The main asymptotic behavior of the free energy turns out to be driven by the amount of time ``lost'' by the trajectory $(p,\lambda)$ in the neighborhood of $(1,1)$ before eventually converging toward $(0,0)$.
\end{remark}
\begin{proof}[Proof of Theorem \ref{thm:DRConjectureSimple}] 
We fix some $\lambda \in (0,e)$ and set $p_c \coloneqq \lambda - \lambda \log \lambda$ if $\lambda > 1$ and, by extension, $p_c \coloneqq 1$ if $\lambda \leq 1$.
We work with $p$ such that $(p,\lambda) \in {\mathscr P}$, which means exactly $p < p_c$.
Let $(p,\lambda)$ be a solution to \eqref{eqn:systDiff} with initial conditions $\lambda(0) = \lambda$ and $p(0) = p$.
Recall that setting $H \coloneqq \frac{p}{\lambda} + \log \lambda$, the function $\lambda$ is solution to the differential equation
\[
  \forall t \geq 0, \quad 
  \lambda'(t) = - \lambda(t) ( 1 - H\lambda(t) + \lambda(t) \log \lambda(t)),
\]
which can be rewritten as: for all $t \geq 0$, we have
\[
  \int_{\lambda(t)}^{\lambda(0)} \frac{\diff y}{y(1 - Hy + y\log y)} = t.
\]
We take interest in the asymptotic behavior, as $x \to 0$ of the function
\[
  G_H \colon x \mapsto \int_x^{\lambda} \frac{\diff y}{y (1 - H y + y \log y)}.
\]
Indeed, since $(p,\lambda) \in {\mathscr P}$, we have $p(t) \to 0$ and $\lambda(t) \to 0$ by Theorem \ref{thm:asympBehaviour} and it follows from \eqref{eqn:alternateExpression} that
\begin{align} \label{ol}
F_\infty(p,\lambda) = \lim_{t \to \infty} \frac{\e^{-t}}{\lambda(t)} = \exp\left( - \lim_{x \to 0} (G_H(x) + \log x )\right),
\end{align}
by composition of limits. 
Thus, in order to prove the result, it is enough to compute the asymptotic behavior, as $p \to p_c$ or equivalently $H \to H_c \coloneqq \frac{p_c}{\lambda} + \log \lambda$, of
\begin{align}
  \lim_{x \to 0} \left( G_H(x) + \log x\right) 
  & = \log \lambda + \int_0^{\lambda} \left( \frac{1}{y (1 - Hy + y \log y)} - \frac{1}{y} \right) \diff y \nonumber \\
  & = \log \lambda + \int_0^{\lambda} \frac{H - \log y}{1 - Hy + y \log y} \diff y, \label{om}
\end{align}
in the three different cases $\lambda \in (1,e)$, $\lambda = 1$ and $\lambda \in (0,1)$.

We start with the case $\lambda \in [1,e)$ (we will distinguish later between $\lambda = 1$ and $\lambda \neq 1$), in which case $H_c =1$. We first observe that
\begin{align*}
C_1 \coloneqq \lim_{H \to 1} \int_0^{\lambda} \left( \frac{H - \log y}{1 - Hy + y \log y} - \frac{H+1-y}{(1-H) y + \frac{(y-1)^2}{2} - \frac{(y-1)^3}{6}} \right) \diff y
\end{align*}
exists and is finite, by dominated convergence\footnote{Note that the main issue to be addressed is the divergence at $1$ as $(y-1)^{-2}$ in the limit. For this, we replaced the $\log y$ in the numerator by $y-1$ and the $1 + y\log y$ in the denominator by its Taylor expansion $1+ (y-1) + \frac{(y-1)^2}{2} - \frac{(y-1)^3}{6}$. For the domination, one can use the triangle inequality, in order to make one replacement after the other, and then deal with them separately, using Taylor--Lagrange inequality.}. Setting $\delta \coloneqq 1-H$, it follows that
\begin{equation} \label{on}
  \lim_{x \to 0} \left( G_H(x) + \log x \right) 
    = \log \lambda + C_1
	+ \int_0^{\lambda} \frac{2-\delta-y}{\delta y + \frac{(y-1)^2}{2} - \frac{(y-1)^3}{6}} \diff y
	+o(1),
\end{equation}
where the $o(1)$ holds as $\delta \to 0$, or equivalently as $p \to p_c$. 
In order to compute the integral, we need to compute the partial fraction decomposition of the rational function in the integral.
The polynomial $\delta y + \frac{(y-1)^2}{2} + \frac{(y-1)^3}{6}$ has a real root $x_0$ and two complex conjugate roots $x_1 + i y_1$ and $x_1 - i y_1$.
These roots are explicit and one can check that $x_0 = 4 + O(\delta)$, $x_1 = 1 - O(\delta)$ and $y_1 = \sqrt{2 \delta} + O(\delta)$. 
Then, we have the following decomposition 
\begin{align*}
\frac{2-\delta-y}{\delta y + \frac{(y-1)^2}{2} - \frac{(y-1)^3}{6}}
= \frac{a}{y-x_0} + \frac{b-ay}{(y-x_1)^2+y_1^2},
\end{align*}
with $a = \frac{4}{3} + O(\delta)$ and $b = \frac{10}{3} + O(\delta)$. 
Therefore, we can finally compute
\begin{align*}
  & \int_0^{\lambda} \frac{2-\delta-y}{\delta y + \frac{(y-1)^2}{2} - \frac{(y-1)^3}{6}} \diff y \\
  & = a \left( \ln(x_0-\lambda) - \ln(x_0) \right)
	+ \frac{b-ax_1}{y_1} \left( 
		\arctan \left( \frac{\lambda-x_1}{y_1} \right)
		- \arctan \left( \frac{-x_1}{y_1} \right) \right) \\
  & \hphantom{={}} {}
	- \frac{a}{2} \left(
		\ln \left( (\lambda-x_1)^2 + y_1^2 \right)
		- \ln \left( x_1^2 + y_1^2 \right) \right).
\end{align*}
From now on, we distinguish $\lambda > 1$ and $\lambda = 1$.
We start with $\lambda > 1$: using the development of the different constants and that $\arctan(t) = \frac{\pi}{2} + O(\frac{1}{t})$ as $t \to \infty$, we get
\begin{align*}
\int_0^{\lambda} \frac{2-\delta-y}{\delta y + \frac{(y-1)^2}{2} - \frac{(y-1)^3}{6}} \diff y 
& =  \pi \sqrt{\frac{2}{\delta}} + C_2 + o(1),
\end{align*}
where $C_2$ is a constant depending on $\lambda$.
Recalling that $\delta = 1 - \frac{p}{\lambda} - \log \lambda = \frac{p_c - p}{\lambda}$, coming back to \eqref{ol}, \eqref{om} and \eqref{on}, we get
\[
F_\infty(p,\lambda) 
= \exp\left( - \pi \sqrt{\frac{2 \lambda}{p_c-p}} - \log \lambda - C_1 - C_2 + o(1) \right)
\]
and it concludes the proof of the first point.

We now consider the case $\lambda = 1$. In that case $\lambda -x_1 = O(\delta)$, so the development of the integral is slightly different:
\begin{equation*}
\int_0^{\lambda} \frac{2-\delta-y}{\delta y + \frac{(y-1)^2}{2} - \frac{(y-1)^3}{6}} \diff y 
 = \frac{\pi}{\sqrt{2\delta}} - \frac{2}{3} \ln \delta + C_3 + o(1),
\end{equation*}
and, in the same way,
\[
F_\infty(p,\lambda) 
= \exp\left( - \frac{\pi}{\sqrt{2 (p_c-p)}} + \frac{2}{3} \ln(p_c-p) - \log \lambda - C_1 - C_3 + o(1) \right),
\]
which proves the second point.

Finally, we deal with $\lambda \in (0,1)$ and, therefore, $p_c = 1$. 
Recall from \eqref{ol} and \eqref{om}, that our aim is to compute the asymptotic behavior, as $p \to 1$, of
\begin{align*} 
\int_0^{\lambda} \frac{H - \log y}{1 - Hy + y \log y} \diff y
&= \int_0^{\lambda} 
	\frac{p - \lambda \log \frac{y}{\lambda}}
		{\lambda - p y + \lambda y \log \frac{y}{\lambda}} 
	\diff y,
\end{align*}
where we used that $H = \frac{p}{\lambda} + \log \lambda$. Note that
\begin{align*}
C_4 \coloneqq \lim_{p \to 1} 
\int_0^{\lambda} \left(
	\frac{p - \lambda \log \frac{y}{\lambda}}
		{\lambda - p y + \lambda y \log \frac{y}{\lambda}} 
	- \frac{p}{\lambda (1-\lambda) - y (p- \lambda)} 
	\right) \diff y
\end{align*}
exists and is finite, by dominated convergence\footnote{Here, the issue is the divergence at $\lambda$ as $(y-\lambda)^{-1}$ in the limit and, therefore, we replaced the $\log \frac{y}{\lambda}$ in the numerator by $0$ and the $y\log \frac{y}{\lambda}$ in the denumerator by $\lambda (\frac{y}{\lambda} - 1)$.}. Then, we compute
\begin{align*}
\int_0^{\lambda} 
	\frac{p}{\lambda (1-\lambda) - y (p- \lambda)} \diff y
& = - \frac{p}{p-\lambda} \log \left( \frac{1-p}{1-\lambda} \right),
\end{align*}
and conclude that 
\[
F_\infty(p,\lambda) 
= \exp \left( 
\frac{1}{1-\lambda} \log \left( \frac{1-p}{1-\lambda} \right)
-\log \lambda -C_4 + o(1)
\right),
\]
which proves the third point.
\end{proof}

\section{Critical Derrida--Retaux tree conditioned on survival}\label{s:criticalDR}

As observed in the previous section, a particular trajectory in the DR model exhibits some interesting features. It is the critical line, followed by a DR model with initial law $\mu_0$ given by  
\[
  \mu_0(\diff x) = p(0) \delta_0 + (1-p(0)) \lambda(0) \e^{-\lambda(0) x} \diff x.
\]
with $(p(0), \lambda(0)) \in {\mathscr C}$. From this starting condition, the law $\mu_t$ is at all times $t >0$ given by a mixture of an exponential random variable and a Dirac mass at $0$. 

In this section, we consider a DR tree $X^t=(X^t_s(u), s \in [0,t], u \in \mathcal{N}_s)$ with initial law $\mu_0$, as defined in Section~\ref{subsec:Yule}, recalling that $\mathcal{N}_s$ is the set of particles of the Yule tree alive at time $s$ and $X^t_s(u)$ the amount of paint remaining at level $s$ in the branch $u$. 
The main objective of the section is to construct, describe and analyze the law of the portion of the Yule tree covered in paint, conditionally on the event that there is ``paint'' remaining at the origin. The conditioned tree turns out to behave as a time-inhomogeneous growth-fragmentation process, which converges to some limiting process. It is believed that the limiting tree is an universal characteristic, and would appear in numerous models related to DR at criticality. We have some anecdotal evidences of this universality, based on simulations. In Section \ref{section:asymptotic_mass_leaves}, we take particular interest in the asymptotic behavior of the number and the mass on the leaves of that tree.

On the event $\{ X^t_0(\varnothing) > 0 \}$, we say that a particle $u \in \mathcal{N}_s$ for $s \leq t$ is \textit{red} if it belongs to the painted connected component of the root: more formally, for any $r \leq s$ and $v \in \mathcal{N}_r$ such that $v \leq u$, $X_r^t(v) > 0$. 
Let $(\mathfrak{X}^t_s(v), v \in \mathfrak{N}_s^t, s \leq t)$ denote the restriction of the process $X^t$ to the painted connected component of the root. 
Note that the labeling of particles is modified between the processes $X^t$ and $\mathfrak{X}^t$.
Indeed, if at the first branching time of $X^t$, all the paint comes from one of the children, the process $\mathfrak{X}^t$ does not ``see'' this branching event, hence keep calling the line $\varnothing$.
The reindexed portion of the Yule tree, denoted by $(\mathfrak{N}_s^t, s \leq t)$, is called the \textit{red tree} (see Figure \ref{fig:construction_red_tree}); note that it depends on $t$, unlike the Yule tree used to define $X^t$.
By analogy with the definition of the Yule tree in Subsection \ref{subsec:Yule}, we denote by $\mathfrak{e}_u^t$, $\mathfrak{b}_u^t$ and $\mathfrak{d}_u^t$ respectively the lifetime, birth time and death time of $u$ in the red tree $(\mathfrak{N}_s^t, s \leq t)$.
By convention, all particles die at time $t$ and all their offspring have zero lifetime.

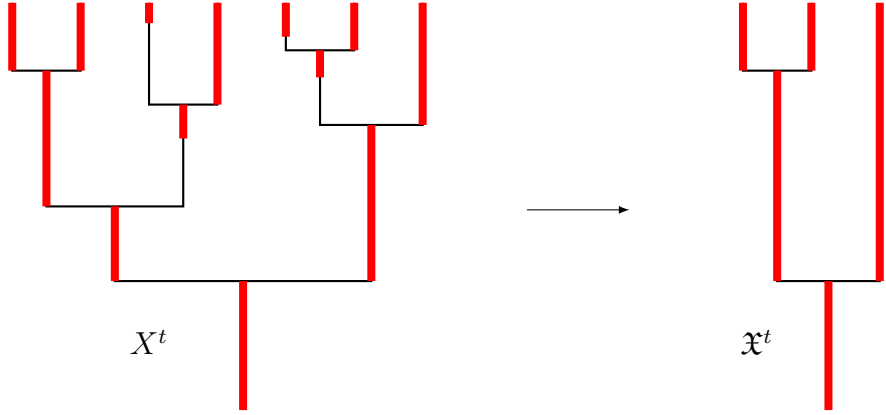
\begin{figure}[ht]
\centering
\subfigure{
\begin{tikzpicture}[scale=0.9]
  \draw [thick] (0,5) -- (0,4) -- (1,4) -- (1,5);
  \draw [thick] (2,5) -- (2,3.5) -- (3,3.5) -- (3,5);
  \draw [thick] (4,5) -- (4,4.3) -- (5,4.3) -- (5,5);
  \draw [thick] (6,5) -- (6,3.2) -- (4.5,3.2) -- (4.5,4.3);
  \draw [thick] (2.5,3.5) -- (2.5,2) -- (0.5,2) -- (0.5,4);
  \draw [thick] (5.25, 3.2) -- (5.25,0.9) -- (1.5,0.9) -- (1.5,2);
  \draw [thick] (3.375,.9) -- (3.375,-1);

  \draw [line width=3,color=red] (0,5) -- (0,4) ++ (1,0) -- (1,5);
  \draw [line width=3,color=red] (2,5) -- ++ (0,-0.3);
  \draw [line width=3,color=red] (3,5) -- ++ (0,-1.5);
  \draw [line width=3,color=red] (2.5,3.5) -- (2.5,3);
  \draw [line width=3,color=red] (4,5) -- ++ (0,-0.5);
  \draw [line width=3,color=red] (5,4.3) -- (5,5);
  \draw [line width=3,color=red] (6,5) -- (6,3.2);
  \draw [line width=3,color=red] (4.5,3.9) -- (4.5,4.3);
  \draw [line width=3,color=red] (0.5,2) -- (0.5,4);
  \draw [line width=3,color=red] (5.25, 3.2) -- (5.25,0.9) ++ (-3.75,0) -- (1.5,2);
  \draw [line width=3,color=red] (3.375,.9) -- (3.375,-1);
  
  \draw (2,0) node{$X^t$};
\end{tikzpicture}
}
\subfigure{
\begin{tikzpicture}[scale=0.9]
\draw[white] (0,-1) -- (0,5);
\draw[white] (3.5,-1) -- (3.5,5);
\draw[->,>=latex] (1,2) -- (2.5,2);
\end{tikzpicture}
}
\subfigure{
\begin{tikzpicture}[scale=0.9]
  \draw [thick] (1,5) -- (1,4) -- (2,4) -- (2,5);
  \draw [thick] (3, 3.2) -- (3,0.9) -- (1.5,0.9) -- (1.5,4);
  \draw [thick] (2.25,.9) -- (2.25,-1);

  \draw [line width=3,color=red] (1,5) -- (1,4);
  \draw [line width=3,color=red] (2,4) -- (2,5);
  \draw [line width=3,color=red] (3,5) -- (3,0.9);
  \draw [line width=3,color=red] (1.5,0.9) -- (1.5,4);
  \draw [line width=3,color=red] (2.25,.9) -- (2.25,-1);
  
  \draw (1.2,0) node{$\mathfrak{X}^t$};
\end{tikzpicture}
}
\caption{Illustration of the construction of the red tree, by pruning the Yule tree at points where there is no paint.} \label{fig:construction_red_tree}
\end{figure}

\subsection{The red tree as a Markovian branching process}

The purpose of this subsection is the description of the law of $(\mathfrak{X}^t_s(v), v \in \mathfrak{N}_s^t, s \leq t)$ conditioned by $X^t_0(\varnothing) = x$ for some $x > 0$.
However, the event $\{ X^t_0(\varnothing) = x \}$ has zero probability, so the question is not well-posed.
But, we give a particular description, as a Markovian branching process, of the law of $\mathfrak{X}^t$ given $X^t_0(\varnothing)$, such that, for this description, it makes sense to consider the process ``starting from $x$''.

Let $(p(t),\lambda(t), t \geq 0)$ be the solution of \eqref{eqn:systDiff} starting from initial condition $p(0)=p$, $\lambda(0)=\lambda$ for some $(p,\lambda) \in [0,1) \times (0,\infty)$. 
Note that in this section, we \textit{do not} restrict ourselves to the critical case.
We introduce the auxiliary function
\[
  \rho(s) \coloneqq (1 - p(s)) \lambda(s), \quad s \geq 0,
\]
which will appear regularly in the sequel.
\begin{proposition} \label{prop:law_of_the_red_tree}
Assume that the initial law $\mu_0$ is given by  \eqref{eqn:initialDistribution} with  $(p(0), \lambda(0)) \in (0,1) \times (0,\infty)$.
Given $X^t_0(\varnothing)$, on the event $\{ X^t_0(\varnothing) > 0 \}$, $\mathfrak{X}^t$ has the law of a time-inhomogeneous branching Markov process, which can be described as follows:
\begin{enumerate}
  \item The process starts at time $0$ with a unique particle of mass $\mathfrak{X}^t_0(\varnothing) = X^t_0(\varnothing)$.
  \item The mass associated to each particle grows linearly at speed $1$.
  \item A particle of mass $m$ at time $s$ splits at rate $\rho(t-s) m$ into two children, the mass $m$ being split uniformly between the two children.
  \item Particles behave independently after their splitting time.
\end{enumerate}
\end{proposition}
This proposition allows us to define the distribution of $\mathfrak{X}^t$ ``starting from $x$''. 
Indeed, for any $x > 0$, let $\P_x$ denote a probability measure such that, under $\P_x$, the process $\mathfrak{X}^t$ is the time-inhomogeneous branching Markov process starting with $\mathfrak{X}^t_0(\varnothing) = x$ and satisfying Points 2, 3 and 4 of Proposition \ref{prop:law_of_the_red_tree}.
Then, informally, the law of $\mathfrak{X}^t$ under $\P$ conditioned by $X^t_0(\varnothing) = x$ is the law of $\mathfrak{X}^t$ under $\P_x$.
Before proving Proposition \ref{prop:law_of_the_red_tree}, the following lemma shows that the whole process $X^t$ satisfies also a branching property. 
\begin{lemma} \label{lem:branching_property_X^t}
Let $s \in (0,t)$. Given the Yule tree and the process $X^t$ until time~$s$, the $(X^t_{s+r}(v), r \in [0,t-s], v \in \mathcal{N}_{s+r} \text{ s.t.\@ } v \geq u)$ for $u \in \mathcal{N}_s$ are independent with respectively the same law as $X^{t-s}$ starting from $X^t_s(u)$.
\end{lemma}
\begin{proof} 
Let $\T_s \coloneqq (\mathcal{N}_r,r\in[0,s])$ denote the Yule tree until time $s$.
Our aim is to prove that, for any family of nonnegative measurable functions $(F_u)_{u \in \T}$ and any r.v.\@ $\Upsilon_s$ measurable w.r.t.\@ the $\sigma$-field generated by $\T_s$ and the process $X^t$ until time~$s$, 
\begin{align}
& \E \left[
\Upsilon_s \prod_{u \in \mathcal{N}_s}
F_u \left( X^t_{s+r}(v), r \in [0,t-s], v \in \mathcal{N}_{s+r} \text{ s.t.\@ } v \geq u \right)
\right] \nonumber \\
& = \E \left[
\Upsilon_s \prod_{u \in \mathcal{N}_s} \tilde{F}_u \left( X^t_s(u) \right)
\right]. \label{pl}
\end{align}
where, for $u \in \T$, $\tilde{F}_u \colon \R_+ \to \R_+$ is the measurable function (unique up to a Lebesgue-negligible set) such that
\[
\tilde{G}_u(X^{t-s}_0(\varnothing))
\coloneqq \E \left[
F_u \left( X^{t-s}_r(v), r \in [0,t-s], v \in \mathcal{N}_r \right)
\middle| X^{t-s}_0(\varnothing) \right].
\]
The process $X^t$ until time~$s$ is a deterministic function of $\T_s$ and of the $X_u(s)$ for $u \in \mathcal{N}_s$. 
Therefore, it suffices to prove \eqref{pl} for $\Upsilon_s$ of the form $\Upsilon_s = g(\T_s) \prod_{u\in \mathcal{N}_s} h_u(X_u(s))$, with $g$ and $(h_u)_{u\in\T}$ measurable functions.
But, each factor $h_u(X_u(s))$ can be incorporated in $F_u \left( \cdots \right)$, with an appropriate change of the function $F_u$.
Hence, it is finally enough to consider $\Upsilon_s = g(\T_s)$.

Given $\T_s$, the subtrees rooted in the $u$'s for $u \in \mathcal{N}_s$ are independent Yule trees of height $t-s$ and the labels on their leaves are independent of each other. 
Therefore, given $\T_s$, the $(X^t_{s+r}(v), r \in [0,t-s], v \in \mathcal{N}_{s+r} \text{ s.t.\@ } v \geq u)$ for $u \in \mathcal{N}_s$ are independent copies of $X^{t-s}$ and it proves \eqref{pl} in the case $\Upsilon_s = g(\T_s)$.
This concludes the proof.
\end{proof}
\begin{remark}
Note that this proof does not use the fact that we are working with the exponential model. Therefore, for any initial law $\mu_0$, $X^t$ is a time-inhomogeneous branching Markov process.
\end{remark}
\begin{proof}[Proof of Proposition \ref{prop:law_of_the_red_tree}]
Points 1 and 2 are clearly true.
Point 4 follows directly from the branching property for $X^t$ established in Lemma \ref{lem:branching_property_X^t}.
We focus now on Point~3.
By the branching property, it is sufficient to prove it for the root of the tree:
hence, our aim is to prove that, for any $0 \leq s < t$ and any measurable functions $f,g_1,g_2 \colon \R \to \R_+$,
\begin{align}
& \E \left[ f(X^t_0(\varnothing)) \1_{X^t_0(\varnothing) > 0}
\1_{\mathfrak{e}^t_\varnothing \leq s} 
g_1(\mathfrak{X}^t_{\mathfrak{e}^t_\varnothing}(1)) g_2(\mathfrak{X}^t_{\mathfrak{e}^t_\varnothing}(2)) 
\right] \nonumber\\
\begin{split}
& = \int_0^\infty f(x) \rho(t) \e^{-\lambda(t) x}
\int_0^s 
\rho(t-r) (x+r) \exp \left( - \int_0^r \rho(t-v) (x+v) \diff v \right) \\
& \hphantom{= \int_0^\infty f(x) \rho(t) \e^{-\lambda(t) x} \int_0^s} {}
\times \int_0^{x+r} g_1(y) g_2(x+r-y) \frac{\diff y}{x+r}
\diff r \diff x,
\end{split} \label{eq:density_branching_and_splitting}
\end{align}
recalling that $\mathfrak{e}^t_\varnothing$ denotes the lifetime of $\varnothing$ in the red tree.
First note that $X^t_0(\varnothing) + \mathfrak{e}^t_\varnothing = \mathfrak{X}^t_{\mathfrak{e}^t_\varnothing}(1) + \mathfrak{X}^t_{\mathfrak{e}^t_\varnothing}(2)$ and, therefore, the left-hand side of \eqref{eq:density_branching_and_splitting} is equal to
\begin{align}
\E \left[ 
f(\mathfrak{X}^t_{\mathfrak{e}^t_\varnothing}(1) + \mathfrak{X}^t_{\mathfrak{e}^t_\varnothing}(2)-\mathfrak{e}^t_\varnothing) 
\1_{\mathfrak{X}^t_{\mathfrak{e}^t_\varnothing}(1) + \mathfrak{X}^t_{\mathfrak{e}^t_\varnothing}(2) > \mathfrak{e}^t_\varnothing}
\1_{\mathfrak{e}^t_\varnothing \leq s} 
g_1(\mathfrak{X}^t_{\mathfrak{e}^t_\varnothing}(1)) g_2(\mathfrak{X}^t_{\mathfrak{e}^t_\varnothing}(2)) \right].
\label{ql}
\end{align}
Then, we distinguish on the first particle $u$ starting from the root, that splits its mass into two \textit{positive} masses between its children.
Denoting $\varnothing = u_0 < u_1 < \dots < u_k = u$ the lineage of $u$ and, for $v \in \T \setminus \{ \varnothing \}$, $\overline{v}$ the single brother of $v$, we get that \eqref{ql} is equal to
\begin{align}
\begin{split}
& \sum_{k\geq0} \sum_{\lvert u \rvert = k} 
\E \Biggl[ \1_{\forall i \in \llbracket 1,k \rrbracket, X_{b_{u_i}}^t (\overline{u_i})= 0} 
\1_{X^t_{d_u}(u1),X^t_{d_u}(u2)>0} 
f(X^t_{d_u}(u1) + X^t_{d_u}(u2) - d_u) \\
& \hphantom{\sum_{k\geq0} \sum_{\lvert u \rvert = k} \E \Biggl[ }
\times \1_{X^t_{d_u}(u1) + X^t_{d_u}(u2) > d_u}
\1_{d_u \leq s} 
g_1(X^t_{d_u}(u1)) g_2(X^t_{d_u}(u2)) \Biggr]. 
\end{split} \label{qm}
\end{align}
Using the branching property, on the event $d_u \leq s$ and given $b_{u_1}, \dots, b_{u_k},d_u$, the Yule trees rooted in $\overline{u_1},\dots,\overline{u_k},u1,u2$ and the labels on their leaves are independent and, therefore, $X_{b_{u_1}}^t (\overline{u_1}), \dots, X_{b_{u_k}}^t (\overline{u_k}), X^t_{d_u}(u1), X^t_{d_u}(u2)$ are independent with laws $\mu_{b_{u_1}}, \dots \mu_{b_{u_k}}, \mu_{d_u}, \mu_{d_u}$ respectively.
Thus, taking the conditional expectation given $b_{u_1}, \dots, b_{u_k},d_u$, \eqref{qm} is equal to
\begin{align}
\begin{split}
& \sum_{k\geq0} \sum_{\lvert u \rvert = k} 
\E \Biggl[ 
\left( \prod_{i=1}^k p(t-b_{u_i}) \right)
\1_{d_u \leq s} (1-p(t-d_u))^2 \lambda(t-d_u)^2 \\
& \hphantom{\sum_{k\geq0} \sum_{\lvert u \rvert = k} \E \Biggl[ }
\times 
\int_0^\infty \int_0^\infty
 \e^{-\lambda(t-d_u)(y+z)}
f(y+z-d_u) \1_{y+z>d_u} g_1(y) g_2(z) \diff z \diff y \Biggr].
\end{split} \label{qn}
\end{align}
On the other hand, given $d_u$, the vector $(b_{u_i})_{1\leq i \leq k}$ has the same law as the ordered vector $(U_i)_{1\leq i \leq k}$, where the $U_i$'s are i.i.d.\@ uniform r.v.\@ on $[0,d_u]$, so, for any measurable positive function $h$ and any $\lvert u \rvert = k$, we have
\begin{align*}
\E \Biggl[ \left( \prod_{i=1}^k p(t-b_{u_i}) \right)
\1_{d_u \leq s} h(d_u) \Biggr]
& = \E \Biggl[ 
\left( \int_0^{d_u} p(t-v) \frac{\diff v}{d_u} \right)^k
\1_{d_u \leq s} h(d_u) \Biggr] \\
& = \int_0^s \left( \int_0^r p(t-v) \frac{\diff v}{r} \right)^k
h(r) \frac{r^k \e^{-r}}{k!} \diff r,
\end{align*}
using that $d_u$ follows the $\Gamma(k+1,1)$ distribution.
Noting that $\#\{ \lvert u \rvert = k \} = 2^k$ and summing over $k$, we get that \eqref{qn} is equal to
\begin{align}
\begin{split}
& \int_{r=0}^s \exp \left( \int_0^r 2 p(t-v) \diff v \right) 
(1-p(t-r))^2 \lambda(t-r)^2 \\
& \hphantom{\int_{r=0}^s} {} \times
\int_{y=0}^\infty \int_{z=0}^\infty
 \e^{-\lambda(t-r)(y+z)}
f(y+z-r) \1_{y+z>r} g_1(y) g_2(z) \diff z \diff y
\e^{-r} \diff r.
\end{split} \label{qo}
\end{align}
Changing $z$ into $x=y+z-r$, using Fubini's theorem and recalling that $\rho(t) = (1-p(t)) \lambda(t)$, \eqref{qo} is equal to
\begin{align*}
\begin{split}
& \int_{x=0}^\infty f(x)
\int_{r=0}^s \exp \left( \int_0^r (2 p(t-v)-1) \diff v \right) 
\rho(t-r)^2 \e^{-\lambda(t-r)(x+r)} \\
& \hphantom{\int_{x=0}^\infty f(x) \int_{r=0}^s} {} \times
\int_{y=0}^{x+r} g_1(y) g_2(x+r-y) \diff y
\diff r \diff x.
\end{split}
\end{align*}
Now, in order to prove \eqref{eq:density_branching_and_splitting}, it is sufficient to prove that
\begin{align}
& \rho(t-r) 	\e^{-\lambda(t-r) (x+r)}   
\exp \left( -\int_0^r (1-2 p(t-v)) \diff v \right) \nonumber \\
& = \rho(t) \e^{-\lambda(t) x} 
\exp \left( -\int_0^r \rho(t-v) (x+v) \diff v \right).
\label{eq:last_formula}
\end{align}
For this, note that $(\log \rho)' = 2p - 1 - \lambda$ by \eqref{eqn:systDiff} and therefore
\begin{align*}
\frac{\rho(t-r)}{\rho(t)} 
& = \exp \left( \log \rho(t-r) - \log\rho(t) \right) 
= \exp \left( -\int_0^r (2p(t-v)-1-\lambda(t-v)) \diff v \right).
\end{align*}
On the other hand, using that $\frac{\diff}{\diff v} \big(\lambda(t-v)(x+v)\big) = \rho(t-v)(x+v) + \lambda(t-v)$, we get
\begin{align*}
\frac{\e^{- \lambda(t-r)(x+r)}}{\e^{-\lambda(t)x}}
= \exp \left( -\int_0^r (\rho(t-v)(x+v) + \lambda(t-v)) \diff v \right).
\end{align*}
Combining the two previous equations, it proves \eqref{eq:last_formula} and concludes the proof.
\end{proof}

\begin{remark}
Using \cite{harrisroberts2014}, one can prove a many-to-one lemma for the red tree. More precisely, the following result holds: For all measurable bounded function $f$, we have
\[
  \E\left( \sum_{u \in \mathcal{N}_t} f(Y^{t,x}_s(u), s \leq t) \right) = \E\left( \e^{\int_0^t \rho(t-s) Z^{t,x}_s \diff s} f(Z^{t,x}_s, s \leq t)\right),
\]
where $Z^{t,x}$ is a time-inhomogeneous Markov process with $Z^{t,x}_0 = x$, with a positive drift $1$, that jumps at rate $2\rho(t-s) Z_s$ to position $UZ^{t,x}$ with $U$ an independent uniform random variable.
\end{remark}

\subsection{Scaling limit of the red tree}

In this subsection, we consider the critical case by taking as initial law $\mu_0$ given by  \eqref{eqn:initialDistribution} with  $(p(0), \lambda(0)) \in {\mathscr C}$.  In that case, it follows from Proposition \ref{prop:asymptotics_of_p_and_lambda} that $\rho(s) = \lambda(s)(1-p(s)) \sim 2/s^2$ as $s \to \infty$.
Using this asymptotic result, one can take the limit $t \to \infty$ in the description of the red tree of Proposition \ref{prop:law_of_the_red_tree}, as shown in the next result.

Before stating and proving this result, we need some additional formalism.
The set of binary continuous trees is $(\R_+)^\T$, where a tree is represented by the lengths of its branches, and the set of masses is $\mathcal{C}(\R_+)^\T$, where to  each particle of the tree is associated the continuous function of its mass during its lifetime.
Then, the process $(\mathfrak{X}^t_s(v), v \in \mathfrak{N}_s^t, s \leq t)$ is encoded as the following random variable in $(\R_+)^\T \times \mathcal{C}(\R_+)^\T$:
\[
\left(
(\mathfrak{e}_u^t)_{u \in \T}, 
(\mathfrak{X}^t_{(\mathfrak{b}_u^t + s) \wedge \mathfrak{d}_u^t} (u), s \geq 0)_{u \in \T}
\right), 
\]
where we recall that for any $u \in \T$,  $\mathfrak{e}_u^t$, $\mathfrak{b}_u^t$ and $\mathfrak{d}_u^t$ denotes respectively the lifetime, birth time and death time of $u$.  
The set $\mathcal{C}(\R_+)$ is endowed with the topology of uniform convergence on every compact set and $(\R_+)^\T \times \mathcal{C}(\R_+)^\T$ with the product topology.

Now, for $T \geq 0$, we define the restriction of the tree to the time interval $[0,T]$ by
\[
\left(
((\mathfrak{d}_u^t \wedge T) - (\mathfrak{b}_u^t \wedge T))_{u \in \T}, 
(\mathfrak{X}^t_{(\mathfrak{b}_u^t + s) \wedge \mathfrak{d}_u^t \wedge T} (u) \1_{\mathfrak{b}_u^t \leq T}, s \geq 0)_{u \in \T}
\right),
\]
which is still an element of $(\R_+)^\T \times \mathcal{C}(\R_+)^\T$.

\begin{theorem} \label{th:scaling_limit_of_the_red_tree} 
Let $\mathfrak{X} = (\mathfrak{X}_s(u), u \in \mathfrak{N}(s), s\in [0,1))$ be a time-inhomogeneous branching Markov process such that under $\P_x$:
\begin{enumerate}
  \item The process starts at time $0$ with a unique particle of mass $x$: $\mathfrak{X}_0(\varnothing) = x$.
  \item The mass associated to each particle grows linearly at speed $1$.
  \item A particle of mass $m$ at time $s$ splits at rate $2m/(1-s)^2$ into two children, the mass $m$ being split uniformly between the two children.
  \item Particles behave independently after their splitting time.
\end{enumerate}
Assume that the initial law $\mu_0$ is given by \eqref{eqn:initialDistribution} with $(p(0), \lambda(0)) \in {\mathscr C}$. 
Let $(x_t)_{t\geq 0}$ be a family of positive real numbers such that $x_t /t \to x \in \R_+$ as $t \to \infty$.
Then, for any $\varepsilon \in (0,1)$, the restriction of $(\mathfrak{X}^t_{st}(v)/t, v \in \mathfrak{N}_{st}^t, s \in [0,1))$ to $[0,1-\varepsilon]$ under $\P_{x_t}$ (defined below Proposition~\ref{prop:law_of_the_red_tree}) converges in law toward the restriction of $\mathfrak{X}$ to $[0,1-\varepsilon]$ under $\P_x$.
\end{theorem} 
\begin{proof}
The first step is to prove that we can deal only with a finite number of particles.
Let $m_t \coloneqq \max_{u \in \mathfrak{N}_{(1-\varepsilon)t}} \lvert u \rvert$ be the maximal generation among the red particles alive at time $(1-\varepsilon)t$, we are going to show that $(M_t)_{t \geq 0}$ is tight in $\R_+$.
By Proposition \ref{prop:asymptotics_of_p_and_lambda}, for any $t$ large enough and $s \in [0,(1-\varepsilon)t]$, we have $\rho(t-s) \leq 4/(t-s)^2 \leq 4/(\varepsilon t)^2$.
On the other hand, note that under $\P_{x_t}$, the mass of any particle at any time is bounded by $x_t + t$, and therefore by $(x+2)t$ for $t$ large enough.
Therefore, by Proposition \ref{prop:law_of_the_red_tree}, the branching rate of any particle on the time interval $[0,(1-\varepsilon)t]$ can be upper bounded by $4(x+2)/(\varepsilon^2 t)$.
We conclude that $m_t$ is stochastically dominated by the maximal generation among the particles alive at time $(1-\varepsilon)$ in a Yule tree with parameter $4(x+2)/\varepsilon^2$ and, thus, $(m_t)_{t \geq 0}$ is tight.

Hence, it is now sufficient to prove, for any $m \in \N$, the following convergence in law
\begin{align*}
& \left(
(\mathfrak{e}_u^t/t)_{u \in \T, \lvert u \rvert \leq m}, 
(\mathfrak{X}^t_{(\mathfrak{b}_u^t + st) \wedge \mathfrak{d}_u^t} (u)/t, s \geq 0)_{u \in \T, \lvert u \rvert \leq m}
\right)  \text{ under } \P_{x_t} \\
& \xrightarrow[t\to\infty]{\mathrm{(d)}}
\left(
(\mathfrak{e}_u)_{u \in \T, \lvert u \rvert \leq m}, 
(\mathfrak{X}_{(\mathfrak{b}_u + s) \wedge \mathfrak{d}_u} (u), s \geq 0)_{u \in \T, \lvert u \rvert \leq m}
\right) \text{ under } \P_x.
\end{align*}
By recurrence on $m$ and using the Markov property of $\mathfrak{X}^t$ (Point~4 of Proposition~\ref{prop:law_of_the_red_tree}), it is sufficient to prove that,  
\begin{align} \label{qs}
\left(
 \mathfrak{e}_\varnothing^t/t, ( \mathfrak{X}^t_{(st) \wedge \mathfrak{d}_\varnothing^t} (\varnothing) / t, s \geq 0)
\right) \text{ under } \P_{x_t}
\xrightarrow[t\to\infty]{\mathrm{(d)}}
\left(
 \mathfrak{e}_\varnothing , 
(\mathfrak{X}_{s \wedge \mathfrak{d}_\varnothing} (\varnothing), s \geq 0) 
\right) \text{ under } \P_x
\end{align}
But, $\mathfrak{X}^t_{(st) \wedge \mathfrak{d}_\varnothing^t} (\varnothing) = x_t + ((st) \wedge \mathfrak{e}_\varnothing^t)$ and $\mathfrak{X}_{s \wedge \mathfrak{d}_\varnothing} (\varnothing) = x + (s \wedge \mathfrak{e}_\varnothing)$, so \eqref{qs} reduces to the proof of the convergence in distribution of $\mathfrak{e}_\varnothing^t/t$ under $\P_{x_t}$ toward $\mathfrak{e}_\varnothing$ under $\P_x$.
By Point~3 of Proposition~\ref{prop:law_of_the_red_tree}, for any $r \in (0,1)$, as $t\to\infty$, 
\begin{align} \label{qt}
\P_{x_t} \left( \mathfrak{e}_\varnothing^t/t >r \right)
= \exp\left( - \int_0^{rt} \varrho(t-u) (x_t+u)  \diff u\right)
\xrightarrow[t\to\infty]{}
\exp\left( -\int_0^r \frac{2 (x+v)  \diff v}{(1-v)^2} \right),
\end{align}
using that $\rho(s) \sim 2/s^2$ as $s\to \infty$ by Proposition \ref{prop:asymptotics_of_p_and_lambda} and $x_t/t \to x$ as $t \to \infty$. 
Finally, note that, by definition of the limit $\mathfrak{X}$, the right-hand side of \eqref{qt} is equal to $\P_x(\mathfrak{e}_\varnothing>r)$, which completes the proof.  
\end{proof}

\begin{remark}
We believe that the limit $\mathfrak{X}$ should be universal. 
Defining similarly the red tree for the discrete DR model, we expect the same scaling limit to hold in the critical case.
This is supported by computer calculations of the distribution of first branching time in the discrete red tree, which can be compared with the distribution of $\mathfrak{e}_\varnothing$ under $\P_0$ given by
\[
\forall r \in [0,1], \quad
\P_0 \left(\mathfrak{e}_\varnothing \leq r \right) 
= 1 - \exp\left( -\int_0^r \frac{2 v \diff v}{(1-v)^2} \right)
= 1 - \frac{1}{(1-r)^2} \exp\left( -\frac{2r}{1-r} \right),
\]
see Figure \ref{fig:simu_first_branching}. Similarly, computer calculations show that the mass in the red tree of height~$n$ is split uniformly among children as $n \to \infty$.
\begin{figure}[ht]
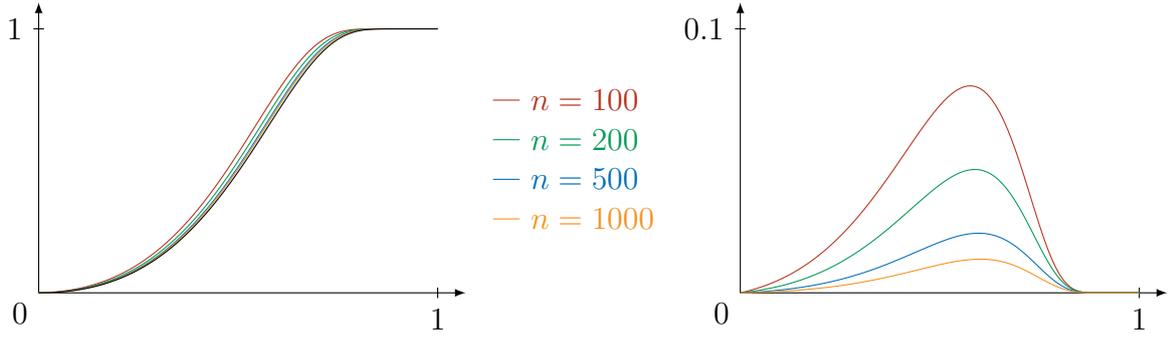
  
\hspace{-0.2cm}%
\subfigure{
\centering
 
}
\caption{For each $n \in \{100,200,500,1000\}$, the repartition function $F_n$ of the first branching time of the discrete red tree of height $n$ with critical initial law $\frac{4}{5} \delta_0 + \frac{1}{5} \delta_2$ has been computed. On the left, the $F_n$'s are drawn in color and the repartition function $F$ of $\mathfrak{e}_\varnothing$ under $\P_0$ is drawn in black. On the right, the $F_n-F$'s are drawn.}
\label{fig:simu_first_branching}
\end{figure}
\end{remark}

\section{Asymptotic behavior of the number and mass of red leaves}\label{s:nombre}
\label{section:asymptotic_mass_leaves}

The red tree, studied in the previous section, is a tool for understanding how mass can be brought from the leaves to the root. We focus here on the critical case, with initial law $\mu_0$ given by \eqref{eqn:initialDistribution} with $(p(0), \lambda(0)) \in {\mathscr C}$, and answer the following questions: typically, when mass is brought to the root, how many leaves have contributed to this amount of mass~? How much mass do these leaves carry~? 

In order to answer these	questions, let $N_t$ be the number of red leaves and $M_t$ the total mass of the red leaves:
\[
N_t \coloneqq \# \mathfrak{N}^t_t
\quad \text{and} \quad
M_t \coloneqq \sum_{u \in \mathfrak{N}^t_t} \mathfrak{X}^t_t(u).
\]
The joint asymptotic behavior of $N_t$ and $M_t$ is given by the following result.
\begin{theorem} \label{th:total_number_mass_of_red_leaves}  Let    $(p(0), \lambda(0)) \in {\mathscr C}$.  
There exist $\gamma_1,\gamma_2 > 0$ such that, for any family $(x_t)_{t\geq 0}$ of positive real numbers such that $x_t /t \to x \in \R_+$ as $t \to \infty$,
$(t^{-2}N_t, t^{-2}M_t)$ under $\P_{x_t}$ converges in distribution to 
$(\gamma_1 \eta,\gamma_2 \eta)$, where $\eta \coloneqq \frac{3}{2} \int_0^1 r^2(s) \diff s$, with $r(\cdot)$ denoting a $4$-dimensional Bessel bridge from $0$ to $2 \sqrt{x}$.
\end{theorem}
Constants $\gamma_1$ and $\gamma_2$ are defined implicitly in Lemma \ref{lem:function_a} and depend on the initial condition $(p(0), \lambda(0))$.
This section is devoted to the proof of Theorem \ref{th:total_number_mass_of_red_leaves}, via the analytic study of the Laplace transform of $(N_t,M_t)$.
In the sequel, $(p(0), \lambda(0)) \in \mathscr{C}$ is fixed and thus also the auxiliary function $\rho(s) = (1 - p(s)) \lambda(s)$.

A further question that we do not answer here is the following: what is the typical mass of a red leaf~? Thereupon, we state the following conjecture.
\begin{conjecture} \label{c:nombreetmasse}
For any family $(x_t)_{t\geq 0}$ of positive real numbers such that $x_t /t \to x \in \R_+$ as $t \to \infty$, the random measure
\[
\frac{1}{N_t} \sum_{u \in \mathfrak{N}^t_t} \delta_{\mathfrak{X}^t_t(u)} 
\text{ under } \P_{x_t}
\]
converges in law for the topology of vague convergence and the limit does not depend on~$x$.
\end{conjecture}

\subsection{The Laplace transform of the mass and number of red leaves}
\label{subsection:laplace_transform_of_mass_and_number}

For $\varepsilon_1, \varepsilon_2 > 0$, we introduce the following function
\begin{equation} \label{eq:def_phi}
\phi(t,x) = \phi_{\varepsilon_1,\varepsilon_2}(t,x) 
\coloneqq \E_x \left[ \e^{- \varepsilon_1 N_t - \varepsilon_2 M_t } \right], 
\quad t,x \geq 0.
\end{equation}
As a function of $(\varepsilon_1,\varepsilon_2)$, $\phi_{\varepsilon_1,\varepsilon_2}(t,x)$ is the Laplace transform of $(N_t,M_t)$ under $\P_x$.
However, in this subsection, we fix $(\varepsilon_1,\varepsilon_2)$ and study analytically $\phi_{\varepsilon_1,\varepsilon_2}$ as a function of $(t,x)$: therefore, we forget the dependence in $\varepsilon_1,\varepsilon_2$ and write simply~$\phi$.
The aim of this subsection is to prove the following result.
\begin{lemma} \label{lem:function_phi}
Let $\varepsilon_1, \varepsilon_2 > 0$.
Then, for any $t,x \geq 0$, we have 
\begin{equation}\label{phitx} \phi(t,x) = \e^{-(\Theta(t) +x \Theta'(t))},
\end{equation}
where $\Theta$ is the unique solution on $\R_+$ to the following differential equation
\begin{equation}
\addtolength\arraycolsep{-0.125cm}
\left\{\begin{array}{rl} 
\Theta'' & = \rho\,  (1-\e^{-\Theta}), \\
\Theta(0) & = \varepsilon_1, \\
\Theta'(0) & = \varepsilon_2.
\end{array}\right. \label{eq:equa_diff_q_mu}
\end{equation}
Moreover, for any $t \geq 0$, we have $\Theta(t) > 0$ and $\Theta'(t) > 0$.
\end{lemma}
\begin{proof}
Recall that $\mathfrak{e}^t_\varnothing$ is the first branching time in the red tree.
Distinguishing on the cases $\mathfrak{e}^t_\varnothing = t$ and $\mathfrak{e}^t_\varnothing < t$ and using Proposition \ref{prop:law_of_the_red_tree}, we have
\begin{align*}
\phi(t,x)
& = \P \left( \mathfrak{e}^t_\varnothing = t \right) \phi(0,x+t)
+ \int_0^t \P \left( \mathfrak{e}^t_\varnothing \in \diff s \right)
\int_0^{x+s} \phi(t-s,y) \phi(t-s,x+s-y) \frac{\diff y}{x+s}
\end{align*}
and, moreover, the law of $\mathfrak{e}^t_\varnothing$ is given by
\begin{align*}
\P \left( \mathfrak{e}^t_\varnothing \geq s \right)
= \exp \left( -\int_0^s \rho(t-r) (x+r) \diff r \right)
= \exp \left( -\int_{t-s}^t \rho(r) (x+t-r) \diff r \right),
\end{align*}
for any $s \in [0,t]$.
Therefore, replacing $s$ by $t-s$ in the first integral, we get the following integral equation for $\phi$:
\begin{align}
\begin{split}
\phi(t,x)
& = \exp \left( -\int_0^t \rho(r) (x+t-r) \diff r \right) \phi(0,x+t) \\
& \hphantom{{} = {}} 
+ \int_0^t \rho(s) \exp \left( -\int_s^t \rho(r) (x+t-r) \diff r \right)
\int_0^{x+t-s} \phi(s,y) \phi(s,x+t-s-y) \diff y \diff s.
\end{split} \label{eq:integral_equation_for_phi}
\end{align}

Let first check that the function $\phi$ suggested in the lemma is a solution of Equation \eqref{eq:integral_equation_for_phi}.
By Cauchy-Lipschitz theorem, there exists a unique maximal $\mathcal{C}^2$ solution~$\Theta$ to Equation \eqref{eq:equa_diff_q_mu} and $(\Theta,\Theta')$ has to stay in $(0,\infty)\times(0,\infty)$ at positive times.
Therefore, $\rho (1-\e^{-\Theta})$ remains bounded and it follows that $\Theta$ is well-defined on $\R_+$.
Hence, for $t,x \geq 0$, we can define $\phi(t,x) = \e^{-(\Theta(t) +x \Theta'(t))}$ and the second term in the right-hand side of \eqref{eq:integral_equation_for_phi} is equal to
\begin{align*}
& \int_0^t \rho(s) \exp \left( -\int_s^t \rho(r) (x+t-r) \diff r \right)
(x+t-s) \e^{-(2\Theta(s) + (x+t-s) \Theta'(s))} \diff s \\
& = \int_0^t \frac{\diff}{\diff s} \left[
\exp \left( -\int_s^t \rho(r) (x+t-r) \diff r \right) 
\e^{-(\Theta(s) + (x+t-s) \Theta'(s))} 
\right] \diff s \\
& = \phi(t,x) 
- \exp \left( -\int_0^t \rho(r) (x+t-r) \diff r \right) \phi(0,x+t),
\end{align*}
using \eqref{eq:equa_diff_q_mu} in order to compute the derivative.
This proves that $(x,t) \mapsto \e^{-(\Theta(t) +x \Theta'(t))}$ is solution of \eqref{eq:integral_equation_for_phi}.

We now prove uniqueness for Equation \eqref{eq:integral_equation_for_phi} among the set of measurable locally bounded functions from $(\R_+)^2 \to \R$.
Let $\phi_1,\phi_2 \colon (\R_+)^2 \to \R$ be measurable locally bounded functions satisfying \eqref{eq:integral_equation_for_phi}.
We fix some $x_0 \geq 0$ and prove that $\phi_1 = \phi_2$ on the triangle $T \coloneqq \{ (t,x) \in (\R_+)^2: x+t \leq x_0 \}$.
It follows from \eqref{eq:integral_equation_for_phi} that
\begin{align} \label{rl}
\abs{\phi_1(t,x)- \phi_2(t,x)}
& \leq C \int_0^t \int_0^{x+t-s} \abs{\phi_1(s,y)-\phi_2(s,y)} \diff y \diff s,
\end{align}
with $C \coloneqq (\sup_{\R_+} \rho) (\sup_T \abs{\phi_1} + \sup_T \abs{\phi_2})$, using that $\rho$ is bounded and $\phi_1,\phi_2$ are locally bounded.
Now, we introduce $T_t \coloneqq T \cap ([0,t]\times \R_+)$: note that, for any $(x,t) \in T$, the variable $(y,s)$ in the right-hand side of \eqref{rl} remains in $T_t$.
Let $t_0 \coloneqq 1/(2Cx_0)$, we prove by induction on $k \in \N$ that $\phi_1-\phi_2 = 0$ on $T_{k t_0}$.
For $k = 0$, its clearly true. 
Now, assume that it holds for some $k \in \N$.
Then, for any $(t,x) \in T_{(k+1)t_0}$, it follows from \eqref{rl} that
\begin{align*} 
\abs{\phi_1(t,x)- \phi_2(t,x)}
& \leq C (t-k t_0) x_0 \sup_{T_{(k+1)t_0}} \abs{\phi_1-\phi_2},
\end{align*}
and, taking the supremum over $(t,x) \in T_{(k+1)t_0}$, we get
\begin{align*}
\sup_{T_{(k+1)t_0}} \abs{\phi_1-\phi_2}
& \leq C t_0 x_0 \sup_{T_{(k+1)t_0}} \abs{\phi_1-\phi_2}
\leq \frac{1}{2} \sup_{T_{(k+1)t_0}} \abs{\phi_1-\phi_2},
\end{align*}
which proves that $\phi_1-\phi_2 = 0$ on $T_{(k+1)t_0}$.
Therefore, we proved that $\phi_1=\phi_2$ on~$T$ and, since it holds for any $x_0 \geq 0$, we get the uniqueness.
The function $\phi$ defined by \eqref{eq:def_phi} takes value in $[0,1]$ and the function $(t,x) \mapsto \e^{-(\Theta(t)+x \Theta'(t))}$ is continuous on $(\R_+)^2$ and therefore locally bounded.
Moreover, they both satisfy Equation \eqref{eq:integral_equation_for_phi} so they coincide.
\end{proof}

\subsection{An auxiliary differential equation}

Recall that our aim is to prove the convergence in law of $(N_t/t^2,M_t/t^2)$ and, for this, the convergence of the Laplace transform $\phi_{\theta_1/t^2,\theta_2/t^2} (t,x_t)$.
Therefore, we will need to study the function $\Theta$ solution of \eqref{eq:equa_diff_q_mu} for small initial conditions $\varepsilon_1,\varepsilon_2$ and, when $\Theta$ is small, it behaves like the function $a$, solution to the following linearized equation:
\begin{equation}
a'' = \rho\,  a. \label{eq:equa_diff_a}
\end{equation}
In this subsection, we prove a preliminary result concerning the asymptotic behavior of the function $a$ solution of \eqref{eq:equa_diff_a}.
\begin{lemma} \label{lem:function_a}
There exists positive constants $\gamma_1, \gamma_2 > 0$ such that, for any $\theta_1, \theta_2 \in \R$, if $a$ is the unique solution on $\R_+$ to the differential equation $a'' = \rho a$ with initial conditions $a(0)=\theta_1$ and $a'(0)=\theta_2$,
then we have $a(t) \sim (\gamma_1 \theta_1 + \gamma_2 \theta_2) t^2$ and $a'(t) \sim 2 (\gamma_1 \theta_1 + \gamma_2 \theta_2) t$ as $t \to \infty$.
\end{lemma}
\begin{proof}
First note that, by linearity, it is sufficient to deal with the cases $(\theta_1,\theta_2) = (1,0)$ and $(\theta_1,\theta_2) = (0,1)$. 
From now on, assume that we are in one of this cases.
In particular, this implies that $a(t) > 0$ and $a'(t) > 0$ for any $t > 0$.
The main idea of the proof is to show that, for large $t$, we can replace $\rho(t)$ by $2/t^2$ in the equation: 
indeed, it follows from Proposition \ref{prop:asymptotics_of_p_and_lambda} that
\[
\rho(t) = \frac{2}{t^2} - \frac{16 \log t}{3 t^3} 
+ o\left( \frac{\log t}{t^3} \right),
\]
so we can choose some $t_0$ large enough such that
\[
\forall t \geq t_0, \quad 
0 < \frac{2}{t^2} - \frac{1}{t^{5/2}}
\leq
\rho(t) 
\leq 
\frac{2}{t^2}.
\]
We first prove a upper bound on $a(t)$.
Let $a_0$ be the solution on $[t_0,\infty)$ of the following ODE:
\begin{align*}
a_0''(t) = \frac{2}{t^2} a_0(t),
\quad
a_0(t_0) = a(t_0), 
\quad 
a_0'(t_0) = a'(t_0).
\end{align*}
Then, $a_0(t) = c_1 t^2 + c_2 t^{-1}$ for any $t \geq t_0$, for some $c_1,c_2 \in \R$ and, using the initial conditions, we get
\begin{equation} \label{a0t}
a_0(t) = \frac{a(t_0)+t_0 a'(t_0)}{3 t_0^2} t^2
+ \frac{2t_0a(t_0) - t_0^2a'(t_0)}{3 t}.
\end{equation}
Then, setting
\begin{align*}
C = C(t_0) \coloneqq \frac{a(t_0)+t_0 a'(t_0)}{t_0^2} > 0,
\end{align*}
we get that $a_0(t) \leq C t^2$ for any $t \geq t_0$.
On the other hand, for any $t \geq t_0$, we have $a(t) \leq a_0(t)$ and, therefore, we proved that
\begin{align}
\forall t \geq t_0, \quad a(t) \leq C t^2, \label{ac}
\end{align}
which is a upper bound with a rough constant.

Now, we prove that $a(t)$ is of order $t^2$ and $a'(t)$ of order $t$ as $t \to \infty$.
For this, we consider some $t_1 > t_0$ and let $a_1$ be the solution on $[t_1,\infty)$ of the following ODE:
\begin{align*}
a_1''(t) = \frac{2}{t^2} a_1(t),
\quad
a_1(t_1) = a(t_1), 
\quad 
a_1'(t_1) = a'(t_1).
\end{align*}
Similarly to $a_0$, $a_1$ has an exact expression of the form \eqref{a0t} with $t_0$ replaced by $t_1$. 
As before,  for any $t \geq t_1$, we have $a(t) \leq a_1(t)$ and
\begin{align}
(a_1-a)''(t) 
& = \frac{2}{t^2} a_1(t) - \rho(t) a(t)
= \frac{2}{t^2} \left( a_1(t) - a(t) \right)
+ \left( \frac{2}{t^2} - \rho(t) \right) a(t) \nonumber \\
& \leq \frac{2}{t^2} \left( a_1(t) - a(t) \right) 
+ \frac{1}{t^{5/2}} C t^2, \label{ab}
\end{align}
using that $a(t) \leq C t^2$ for any $t \geq t_0$.
Then, we introduce the auxiliary function $y(t) \coloneqq a_1(t) -a(t) + \frac{4 C}{5} t^{3/2}$ and it follows from \eqref{ab} that, for any $t \geq t_1$, 
\begin{align*}
y''(t) \leq \frac{2}{t^2} y(t).
\end{align*}
Therefore, we get $y \leq \overline{y}$, where $\overline{y}$ is the solution on $[t_1,\infty)$ of the following ODE:
\begin{align*}
\overline{y}''(t) \leq \frac{2}{t^2} \overline{y}(t),
\quad \overline{y}(t_1) = y(t_1) =\frac{4 C}{5} t_1^{3/2},
\quad \overline{y}'(t_1) = y'(t_1) = \frac{6 C}{5} t_1^{1/2}.
\end{align*}
One compute as before that, for any $t \geq t_1$, 
\begin{align*}
\overline{y}(t) = \frac{2 C}{\sqrt{t_1}} t^2 + \frac{2C t_1^{5/2}}{15 t}.
\end{align*}
Therefore, it follows that
\begin{align} \label{ah}
a_1(t) - a(t) 
= y(t) - \frac{4 C}{5} t^{3/2}
\leq \frac{2 C}{\sqrt{t_1}} t^2 + \frac{2C t_1^{5/2}}{15 t}.
\end{align}
and, using the explicit value of $a_1$, we finally proved that, for any $t_1 \geq t_0$, as $t \to \infty$,
\begin{align} \label{af}
\left( \frac{a(t_1)+t_1 a'(t_1)}{3 t_1^2} 
	- \frac{2 C}{\sqrt{t_1}} \right) t^2
+ o \left(t^2 \right)
\leq a(t) 
\leq \frac{a(t_1)+t_1 a'(t_1)}{3 t_1^2} t^2
+ o \left(t^2 \right).
\end{align}
We apply this result with $t_1 = t_0$ and it gives, as $t \to \infty$,
\begin{align} \label{ad}
\left( \frac{C}{3} - \frac{2 C}{\sqrt{t_0}} \right) t^2
+ o \left(t^2 \right)
\leq a(t)
\leq \frac{C}{3} + o \left(t^2 \right),
\end{align}
where $\frac{C}{3} - \frac{2 C}{\sqrt{t_0}} > 0$ by choosing $t_0 > 36$ at the beginning. 
Thus, we proved that $a(t)$ is of order $t^2$. 
Combining this with $a'(t) = \int_0^t \rho(s) a(s) \diff s$ and $\rho(t) \sim 2/t^2$, we get
\begin{align} \label{ae}
2 \left( \frac{C}{3} - \frac{2 C}{\sqrt{t_0}} \right) t + o(t)
\leq a'(t)
\leq \frac{2C}{3} t + o(t),
\end{align}
so $a'(t)$ is of order $t$.
Finally, we prove the result. 
Using \eqref{ad} and \eqref{ae}, there exist an increasing sequence $(s_n)_{n\in\N}$ with $s_n \to \infty$ and two constants $c,c' > 0$ such that
\begin{align} \label{ag}
\frac{a(s_n)}{s_n^2} \xrightarrow[n\to\infty]{} c
\quad \text{and} \quad 
\frac{a'(s_n)}{s_n} \xrightarrow[n\to\infty]{} c'.
\end{align}
Then, for any $n \in \N$, applying \eqref{af} with $t_1 = s_n$, we get
\begin{align*} 
\left( \frac{a(s_n)+s_n a'(s_n)}{3 s_n^2} 
	- \frac{2 C}{\sqrt{s_n}} \right) t^2
+ o \left(t^2 \right)
\leq a(t) 
\leq \frac{a(s_n)+s_n a'(s_n)}{3 s_n^2} t^2
+ o \left(t^2 \right)
\end{align*}
and, using \eqref{ag}, it follows that $a(t) / t^2 \to (c+c')/3$.
In particular, we have $c = (c+c')/3$.
Using again that $a'(t) = \int_0^t \rho(s) a(s) \diff s$ and $\rho(t) \sim 2/t^2$, it follows that $a'(t)/t \to 2c$.
Therefore, we proved the result for $(\theta_1,\theta_2) = (1,0)$ and $(\theta_1,\theta_2) = (0,1)$ and the result follows.
\end{proof}
\begin{remark}
Let $\psi(t,x) \coloneqq \E_x[N_t]$ (or similarly $\E_x[M_t]$), then one can check in the same way as in Lemma \ref{lem:function_phi} that $\psi (t,x) = a(t) + a'(t) x$ with $a$ solution of Equation \eqref{eq:equa_diff_a}.
Therefore, Lemma \ref{lem:function_a} implies that $\E_{xt}[N_t] \sim \gamma_1 (1+2x) t^2$ and $\E_{xt}[M_t] \sim \gamma_2 (1+2x) t^2$ as $t \to \infty$.
\end{remark}

\subsection{Asymptotics of the Laplace transform}\label{ssub:laplace}

In this subsection, we prove Theorem \ref{th:total_number_mass_of_red_leaves}, using results of the two previous subsections.
Recall that $\phi_{\varepsilon_1,\varepsilon_2}(t,x) 
= \E_x[\e^{- \varepsilon_1 N_t - \varepsilon_2 M_t }] = \e^{-(\Theta(t) + x \Theta'(t))}$, where $\Theta$ is the solution of \eqref{eq:equa_diff_q_mu}.
Our aim is to prove the convergence of $\E_{x_t}[\e^{- \theta_1 t^{-2} N_t - \theta_2 t^{-2} M_t }]$ as $t \to \infty$.
Hence, we consider $\Theta_\varepsilon$ solution of 
\begin{equation}
\addtolength\arraycolsep{-0.125cm}
\left\{\begin{array}{rl} 
\Theta_\varepsilon'' & = \rho \, (1-\e^{-\Theta_\varepsilon}), \\
\Theta_\varepsilon(0) & = \theta_1 \varepsilon, \\
\Theta_\varepsilon'(0) & = \theta_2 \varepsilon,
\end{array}\right. \label{eq:equa_diff_q_mu_varepsilon}
\end{equation}
for any fixed $\theta_1,\theta_2 > 0$, and study the asymptotics of $\Theta_\varepsilon(1/\sqrt{\varepsilon})$ and $\Theta_\varepsilon'(1/\sqrt{\varepsilon})$ as $\varepsilon \to 0$.
The first step is to control the behavior of $\Theta_\varepsilon$ and $\Theta_\varepsilon'$ at large but fixed time, as $\varepsilon \to 0$. This is done in the following result.
\begin{lemma} \label{lem:mu_at_fixed_time}
Let $\theta_1, \theta_2 \in \R$ and $\delta > 0$.
For large enough $t$, there exists $\varepsilon_0=\varepsilon_0(t, \delta)>0$ such that for any $0<\varepsilon<\varepsilon_0$,  
\[
\abs{\Theta_\varepsilon(t) - \varepsilon ct^2} \leq \delta \varepsilon t^2 
\quad \text{and} \quad 
\abs{\Theta_\varepsilon'(t) - 2 \varepsilon ct^2} \leq \delta \varepsilon t^2,
\]
where $c = \gamma_1 \theta_1 + \gamma_2 \theta_2$ and $\gamma_1,\gamma_2$ are defined in Lemma \ref{lem:function_a}.
\end{lemma}
\begin{proof}
Let $a$ be solution of $a'' = \rho a$ with initial conditions $a(0) = \theta_1$ and $a'(0)= \theta_2$.
The strategy of the proof is to show that $\Theta_\varepsilon(t)$ and $\Theta_\varepsilon'(t)$ are close to $\varepsilon a(t)$ and $\varepsilon a'(t)$ for small $\varepsilon$ and then to apply Lemma \ref{lem:function_a}.
From now on, we fix some $t \geq 0$ and the constants $C_i$ can depend on $t$, as well as on $\theta_1,\theta_2,\delta$.
First note that $\Theta_\varepsilon '' = \rho(1-\e^{-\Theta_\varepsilon}) \leq \rho \Theta_\varepsilon$ and, therefore, $\Theta_\varepsilon \leq \varepsilon a$.
Then, it follows from Lemma \ref{lem:function_a} that there exists $C > 0$ such that, for any $s \geq 0$,
\begin{align} \label{am}
\Theta_\varepsilon(s) \leq \varepsilon a(s) \leq C \varepsilon (s+1)^2.
\end{align}
In order to prove that $\Theta_\varepsilon(t)$ is close to $\varepsilon a(t)$, we introduce $y \coloneqq \varepsilon a - \Theta_\varepsilon$, which satisfies $y(0)=0$,  $y'(0)=0$ and, for any $s \geq 0$,
\begin{align*}
y''(s) 
&  = \rho(s) \left(\varepsilon a(s) - 1+\e^{-\Theta_\varepsilon(s)} \right) 
= \rho(s) \left(y(s) + \Theta_\varepsilon(s) - 1+\e^{-\Theta_\varepsilon(s)} \right) \\
& \leq \rho(s) \left( y(s) + \frac{\Theta_\varepsilon(s)^2}{2} \right)
\leq \rho(s) \left( y(s) + \frac{C}{2} (s+1)^4 \varepsilon^2 \right),
\end{align*}
using \eqref{am}.
Since the function $\rho$ is bounded, we deduce that, for any $s \in [0,t]$, 
$y''(s) \leq C_1 y(s) + C_2 \varepsilon^2$, where $C_2$ depends on $t$.
Thus, we get, for any $s \in [0,t]$, 
\begin{align*}
y(s) 
& \leq C_2 \varepsilon^2 \left( \cosh \left( \sqrt{C_1} s \right) - 1 \right)
\leq C_3 \varepsilon^2.
\end{align*}
It follows that, for any $s \in [0,t]$, 
\begin{align} \label{an}
\abs{\Theta_\varepsilon(s) - \varepsilon a(s)} \leq C_3 \varepsilon^2.
\end{align}
Moreover, since $\Theta_\varepsilon '(t) = \int_0^t \rho(s) (1-\e^{-\Theta_\varepsilon(s)})  \diff s$ and $a'(t) = \int_0^t \rho(s) a(s) \diff s$, we get
\begin{align*}
\abs{\Theta_\varepsilon'(t) - \varepsilon a'(t)}
& \leq \int_0^t \rho(s) \abs{1-\e^{-\Theta_\varepsilon(s)} - \varepsilon a(s)} \diff s \\
& \leq \int_0^t \rho(s)
\left( \frac{\Theta_\varepsilon(s)^2}{2} + \abs{\Theta_\varepsilon(t)-\varepsilon a(t)} \right) \diff s \\
& \leq \varepsilon^2 \int_0^t \rho(s) \left( \frac{C}{2} (s+1)^4 + C_3 \right) \diff s
= C_4 \varepsilon^2,
\end{align*}
using \eqref{am} and \eqref{an}.
Therefore, we proved that, for any $\varepsilon < \delta t^2/(2(C_3 \vee C_4))$, we have 
\begin{align*}
\abs{\Theta_\varepsilon(t) - \varepsilon a(t)} \leq \frac{\delta \varepsilon t^2}{2}
\quad \text{and} \quad 
\abs{\Theta_\varepsilon'(t) - \varepsilon a'(t)} \leq \frac{\delta \varepsilon t^2}{2}
\end{align*}
But, on the other hand, it follows from Lemma \ref{lem:function_a} that there exists $t_0 > 0$ such that, for all $t \geq t_0$, 
\begin{align*}
\abs{a(t)-ct^2} \leq \frac{\delta}{2} t^2
\quad \text{and} \quad 
\abs{a'(t)-2ct^2} \leq \frac{\delta}{2} t^2,
\end{align*}
so the result holds for $t \geq t_0$ and $\varepsilon < \delta t^2/(2(C_3 \vee C_4))$.
\end{proof}
Using the previous lemma, we now can control $\Theta_\varepsilon$ at some large fixed time $t$. Then, in order to get the asymptotics of $\Theta_\varepsilon(1/\sqrt{\varepsilon})$ as $\varepsilon\to0$, we can replace $\rho(s)$ by $2/s^2$ for $s \in [t,1/\sqrt{\varepsilon}]$ in Equation \eqref{eq:equa_diff_q_mu_varepsilon} and this will prove the following result.
\begin{lemma} \label{lem:mu_asymptotic_behavior}
Let $\theta_1, \theta_2 \in \R$.
Then, we have 
\begin{align*}
\Theta_\varepsilon(1/\sqrt{\varepsilon}) 
\xrightarrow[\varepsilon\to0]{} 
\log \frac{\sinh^2(\sqrt{3c})}{3c}
\quad \text{and} \quad
\frac{\Theta_\varepsilon'(1/\sqrt{\varepsilon})}{\sqrt{\varepsilon}} 
\xrightarrow[\varepsilon\to0]{} 
2 \left( \sqrt{3c} \coth(\sqrt{3c}) - 1 \right),
\end{align*}
where $c = \gamma_1 \theta_1 + \gamma_2 \theta_2$ and $\gamma_1,\gamma_2$ are defined in Lemma \ref{lem:function_a}.
\end{lemma}
\begin{proof}
By Proposition \ref{prop:asymptotics_of_p_and_lambda}, we can choose $t_1 \geq 1$ large enough such that 
\begin{align}
\forall s \geq t_1, \quad 
0 < \frac{2}{s^2} - \frac{1}{s^{5/2}}
\leq
\rho(s) 
\leq 
\frac{2}{s^2}.
\label{ap}
\end{align}
We fix some $t \geq t_1$, and consider $z$ the solution (depending on $t$ and $\varepsilon$) of the following equation:
\[ 
\forall s \geq t, \quad 
z''(s)=\frac{2}{s^2} (1- \e^{-z(s)}),
\] 
with initial conditions $z(t)= \Theta_\varepsilon(t)$ and $z'(t)= \Theta_\varepsilon'(t)$.
Then, $z$ is explicitly given by 
\[ 
\forall s \geq t, \quad 
z(s)= \log \frac{\sinh^2(\sqrt{\xi}(s-t) + \zeta)}{\xi s^2}, 
\]
where 
\[ 
\xi \coloneqq \frac{1}{t^2} \left( \left( 1+ \frac{t}{2} \Theta_\varepsilon'(t)\right)^2 - \e^{- \Theta_\varepsilon(t)} \right)
\quad \text{and} \quad
\zeta \coloneqq \mathrm{arccosh} \left( \e^{\Theta_\varepsilon(t)/2} 
\left( 1+ \frac{t}{2} \Theta_\varepsilon'(t)\right) \right).
\] 
Using Lemma \ref{lem:mu_at_fixed_time}, we get that
\[
\lim_{t \to \infty} \limsup_{\varepsilon\to0} \abs{\frac{\xi}{\varepsilon} - 3c} = 0
\quad \text{and} \quad 
\lim_{t \to \infty} \limsup_{\varepsilon\to0} \abs{\frac{\zeta}{\sqrt{\varepsilon}} - \sqrt{3c} t} = 0,
\]
and it follows that
\begin{align}
\lim_{t \to \infty} \limsup_{\varepsilon\to0} 
\abs{z(1/\sqrt{\varepsilon}) - \log \frac{\sinh^2(\sqrt{3c})}{3c} } & = 0, \label{aq} \\
\lim_{t \to \infty} \limsup_{\varepsilon\to0} 
\abs{\frac{z'(1/\sqrt{\varepsilon})}{\sqrt{\varepsilon}} 
- 2 \left( \sqrt{3c} \coth(\sqrt{3c}) - 1 \right) } & = 0. \label{ar}
\end{align}
Now, we have to prove that $\Theta_\varepsilon(1/\sqrt{\varepsilon})$ and $\Theta_\varepsilon'(1/\sqrt{\varepsilon})$ are close to $z(1/\sqrt{\varepsilon})$ and $z'(1/\sqrt{\varepsilon})$ when $\varepsilon \to 0$ and then $t \to \infty$.
Firstly, it follows from \eqref{ap} that $\Theta_\varepsilon(s) \leq z(s)$ for any $s \geq t$. 
Then, recalling from \eqref{am} that there exists $C > 0$ such that $\Theta_\varepsilon(s) \leq C \varepsilon s^2$ for any $s \geq 1$, we get that, for any $s \geq t$, 
\begin{align}
(z-\Theta_\varepsilon)''(s)
& = \frac2{s^2} \left( \e^{-\Theta_\varepsilon(s)}- \e^{-z(s)} \right)
+ \left( \frac{2}{s^2}- \rho(s) \right) \left( 1- \e^{- \Theta_\varepsilon(s)} \right) \nonumber \\
& \leq \frac{2}{s^2} (z-\Theta_\varepsilon)(s) + \frac{C \varepsilon}{\sqrt{s}}, \label{as}
\end{align}
using that $\Theta_\varepsilon(s), z(s) \geq 0$ and \eqref{ap}.
This is the same differential inequality as in \eqref{ab}, so we get the same result as in \eqref{ah}: for any $s \geq t$,
\begin{align} \label{at}
z(s)-\Theta_\varepsilon(s)
\leq \frac{2 C\varepsilon}{\sqrt{t}} s^2 + \frac{2C\varepsilon t^{5/2}}{15 s}.
\end{align}
Recalling that $z(s)-\Theta_\varepsilon(s) \geq 0$, we get 
\begin{align*}
\lim_{t \to \infty} \limsup_{\varepsilon\to0} 
\abs{z(1/\sqrt{\varepsilon}) - \Theta_\varepsilon(1/\sqrt{\varepsilon}) } = 0.
\end{align*}
Coming back to \eqref{aq}, the first convergence of the lemma follows.
For the second one, note that we have $\Theta_\varepsilon'(1/\varepsilon) \leq z'(1/\varepsilon)$ and, integrating \eqref{as} and using \eqref{at},
\begin{align*}
(z-\Theta_\varepsilon)'(1/\varepsilon)
& \leq \int_t^{1/\sqrt{\varepsilon}} \left(
\frac{2}{s^2} \left( \frac{2 C\varepsilon}{\sqrt{t}} s^2 + \frac{2C\varepsilon t^{5/2}}{15 s} \right) + \frac{C \varepsilon}{\sqrt{s}}
\right) \diff s \\
& \leq \frac{4C \sqrt{\varepsilon}}{\sqrt{t}}
+ \frac{2C\varepsilon \sqrt{t}}{15}
+ 2C \varepsilon^{3/4}.
\end{align*}
Combining this with \eqref{ar}, it proves the second convergence.
\end{proof}
We can now conclude this subsection by proving Theorem \ref{th:total_number_mass_of_red_leaves}.
\begin{proof}[Proof of Theorem \ref{th:total_number_mass_of_red_leaves}]
Let $\theta_1,\theta_2 > 0$.
Then, with $\varepsilon = 1/t^2$ and $\Theta_\varepsilon$ the solution of \eqref{eq:equa_diff_q_mu_varepsilon}, we have by Lemma \ref{lem:mu_asymptotic_behavior}
\begin{align}
\E_{x_t} \left[ \e^{- \theta_1 t^{-2} N_t - \theta_2 t^{-2} M_t } \right]
& = \phi_{\theta_1/t^2,\theta_2/t^2} (t,x_t)
= \exp \left(
- \Theta_\varepsilon(1/\sqrt{\varepsilon}) 
- \frac{x_t}{t} \frac{\Theta_\varepsilon'(1/\sqrt{\varepsilon})}{\sqrt{\varepsilon}}
\right) \nonumber \\
& \xrightarrow[t\to\infty]{}
\frac{3c}{\sinh^2(\sqrt{3c})}
\exp \left(
- 2x \left( \sqrt{3c} \coth(\sqrt{3c}) - 1 \right)
\right), \label{av}
\end{align}
where $c = \gamma_1 \theta_1 + \gamma_2 \theta_2$.
The right-hand side of \eqref{av}, seen as a function of $(\theta_1,\theta_2)$, is the Laplace transform of $(\gamma_1 \eta,\gamma_2 \eta)$ by \cite[Equation (2.5)]{Yor1992}, 
hence it proves the result.
\end{proof}

\begin{proof}[Proof of Theorem \ref{t:critical}]
It follows immediately from Theorems \ref{th:scaling_limit_of_the_red_tree} and \ref{th:total_number_mass_of_red_leaves}.
\end{proof} 

\section{Some properties of Derrida--Retaux models}\label{s:general}

In this section, we prove some elementary results on general continuous--time  DR models, in other words, we do not assume the initial law to be a mixture of exponential laws.  We end the paper by some open questions. 

In the continuous DR model, the analogous to the free energy in the discrete case is the quantity
\begin{equation}
  \label{eqn:definitionEnergieLibreTempsContinu}
  F_\infty(\mu_0) \coloneqq \lim_{t \to \infty} \e^{-t} \E\left( X_t \right)  \in [0, \infty).
\end{equation}
In this section, we prove that the above quantity always exists, and that $\e^{-t}X_t$ converges in law, as $t \to \infty$, toward an exponential random variable with parameter $F_\infty(\mu_0)^{-1}$. We also give some necessary conditions for $F_\infty(\mu_0) =0$, giving a partial characterization of the supercritical and subcritical laws. The proofs are based on simple coupling arguments, as well as the computation of $\E(f(X_t))$ for various $\mathcal{C}^1$ function $f$, using the SDE representation of the process and the It\^o formula.

\begin{lemma}
\label{lem:freeEnergyExists}
The limit in \eqref{eqn:definitionEnergieLibreTempsContinu} exists for any initial distribution $\mu_0$ on $\R_+$. More precisely, if $\E(X_0) = \infty$, then $\lim_{t \to \infty} \e^{-t} X_t = \infty$ in law and, if $\E(X_0) < \infty$, then
\[
  \forall t \geq 0, \quad
  F_\infty(\mu_0) = \e^{-t}\E(X_t) - \int_t^\infty \e^{-s} \P(X_s > 0) \diff s.
\]
\end{lemma}

\begin{proof}
We first observe that for all $t \geq 0$, the random variable $X_t$ stochastically dominates
\[
  S_t = \sum_{j=1}^{\# \mathcal{N}_t} X^j_0 - L_t,
\]
where $(X^j_0, j \geq 1)$ are i.i.d.\@ copies of $X_0$, and $L_t$ is the total length at time $t$ of the Yule tree. By classical results on continuous time branching processes (see Athreya and Ney \cite[Section III.7]{athreyaney72}), there exists a random variable $N_\infty$ such that $\e^{-t} \#\mathcal{N}_t \to N_\infty > 0$ a.s.
and it follows easily that
\begin{equation} \label{eq:length_of_Yule_tree}
  \e^{-t} L_t = \e^{-t} \int_0^t \# \mathcal{N}_s \diff s 
  \xrightarrow[t\to\infty]{\text{a.s.}} N_\infty.
\end{equation}
If $\E(X_0) = \infty$, then, by the law of large numbers, we get $\lim_{t \to \infty} \e^{-t} S_t = \infty$ a.s.\@ and therefore $F_\infty(\mu_0)= \infty$.

Reciprocally, we now assume that $\E(X_0) < \infty$, and we prove that $F_\infty(\mu_0)$ exists and is finite. Using the SDE representation of $X_t$, we have immediately that
\[
  \E\left( X_t \right) = \E(X_0) - \int_0^t \P(X_s > 0) \diff s + \int_0^t \E(X_s) \diff s.
\]
Solving this differential equation, we obtain that
\[
  \E\left( X_t \right) = \E(X_0)\e^t - \e^t \int_0^t \e^{-s} \P(X_s > 0) \diff s.
\]
As $\P(X_s > 0) \in [0,1]$, we deduce that $\E(\e^{-t}X_t)$ converges as $t \to \infty$, and that
\[
  F_\infty(\mu_0) = \E(X_0) - \int_0^\infty \e^{-s}\P(X_s > 0) \diff s.
\]
Applying this formula with the starting law $X_t$, this concludes the proof of the lemma.
\end{proof}

This lemma proves the existence of the free energy of the continuous-time DR model, and that as soon as $\E(X_t) < \infty$, the free energy can be seen as a preserved quantity of the evolution.
In the next lemma, we observe that if $X_0$ has moments of order $n$, then the $n$th moment of $X_t$ converges toward the $n$th moment of an exponential random variable of parameter $F_\infty(\mu_0)^{-1}$.
 
\begin{lemma}
\label{lem:moments}
If $\E(X_0^n) < \infty$ for some $n \in \N$, then $\E(X_t^n)\e^{-nt}$ converges to $n! F_\infty(\mu_0)^n$ as $t \to \infty$. 
\end{lemma}

\begin{proof}
We prove this result by recurrence on $n$. By Lemma \ref{lem:freeEnergyExists}, we know that $\E(X_t)\e^{-t} \to F_\infty(\mu_0)$ as long as $\E(X_0) < \infty$, hence the property is true for $n =1$.

Let $n \geq 2$ and assume the lemma to be true for initial random variables in $L^k$ for $k < n$. Let $X_0$ be a random variable in $L^n$, we note that $X_0 \in L^k$ for all $k < n$, hence by recurrence assumption, we have
\begin{equation}
  \label{eqn:recchyp}
  \lim_{t \to \infty} \E((\e^{-t}X_t)^k) = k! F_\infty(\mu_0)^k.
\end{equation}
Using the SDE representation of $X$, we now observe that we can write
\[
  \E(X_t^n) = \E(X_0^n) - n  \int_0^t  \E(X_s^{n-1}) \diff s + \int_0^t \E( (X_s + \tilde{X}_s)^n - X_s^n) \diff s,
\]
with $\tilde{X}_s$ an independent copy of $X_s$. We rewrite this equality as follows:
\[
  \E(X_t^n) = \E(X_0^n) + \int_0^t \E(X_s^n) \diff s + \int_0^t  \left( -n\E(X_s^{n-1}) + \sum_{k=1}^{n-1} \binom{n}{k}\E(X_s^k)\E(X_s^{n-k}) \right) \diff s.
\]
This first order differential equation can be solved as
\begin{equation} \label{eq:moments_X_t}
  \E(X_t^n) = \E(X_0^n)\e^{t} + \e^t \int_0^t \e^{-s} \left( -n \E(X_s)^{n-1} + \sum_{k=1}^{n-1} \binom{n}{k} \E(X_s^k) \E(X_s^{n-k})\right)\diff s,
\end{equation}
which can be rewritten
\begin{multline*}
  \E\left( (X_t\e^{-t})^n \right) = \E(X_0^n) \e^{-(n-1)t} - \e^{-(n-1)t} \int_0^t \e^{(n-2) s} n \E((X_s\e^{-s})^{n-1}) \diff s \\+ \e^{-(n-1)t} \int_0^t \e^{(n-1) s}\sum_{k=1}^{n-1} \binom{n}{k} \E((X_s \e^{-s})^k) \E((X_s\e^{-s})^{n-k})) \diff s.
\end{multline*}
Therefore, using \eqref{eqn:recchyp} and the L'Hospital rule, we obtain that $\E((X_t\e^{-t})^n)$ converges toward
\begin{align*}
  \frac{1}{n-1} \sum_{k=1}^{n-1} \binom{n}{k} \lim_{t \to \infty}\E((X_s \e^{-s})^k) \E((X_s\e^{-s})^{n-k})) &= \frac{1}{n-1} \sum_{k=1}^{n-1} \binom{n}{k} F_\infty(\mu_0)^n k!(n-k)! \\&= n!F_\infty(\mu_0)^n 
\end{align*}
completing the proof.
\end{proof}

A direct consequence of this proof is the convergence, as $t \to \infty$, of the law of~$\e^{-t} X_t$, as soon as $X_0$ has some finite exponential moments.
In fact, this convergence holds as soon as $F_\infty(\mu_0)$ is well-defined, which is equivalent to $\E[X_0] < \infty$ by Lemma \ref{lem:freeEnergyExists}.

\begin{proposition} \label{p:convergenceexponentielle}
Let $(X_t, t \geq 0)$ be a DR model with $\E(X_0)< \infty$. Then $\e^{-t} X_t$ converges in law toward an exponential random variable with parameter $F_\infty(\mu_0)^{-1}$.
\end{proposition}

\begin{proof}
We first observe that by Lemma \ref{lem:freeEnergyExists}, $\sup_{t \geq 0} \E(\e^{-t} X_t) < \infty$. Therefore the family $(\e^{-t} X_t, t \geq 0)$ is tight. The only thing left to do is thus to identify the limit in distribution.

Remark that for all $t,T \geq 0$, the random variable $X_{T+t}$ can be written as
\[
  \e^{-(t+T)} X_{t+T} = \e^{-T} \sum_{u \in \mathcal{N}_T} \e^{-t} X^{(u)}_t - \e^{-(t+T)} R_T,
\]
where, conditionally on $\mathcal{N}_T$, $(X^{(u)}_t, u \in \mathcal{N}_T)$ are i.i.d.\@ random variables with the same law as $X_t$, and $R_T$ is a random variable which takes values between $0$ and $L_T$ the length of the Yule tree of height $T$. 
Recall from \eqref{eq:length_of_Yule_tree} that $\lim_{T \to \infty} \e^{-T}L_T$ exists a.s.\@, hence if $t$ and $T$ both tend to $\infty$, then $\e^{-(t+T)} R_T$ converges to $0$ in probability.

We now write $(t_n, n \geq 1)$ a sequence along which $(X_{t_n})$ converges in law toward a random variable $Y$. Up to extraction, we can assume that $r_n \coloneqq t_{n+1}-t_n \to \infty$. Therefore, we have that
\[
  Y = \lim_{n \to \infty} \e^{-t_{n+1}} X_{t_{n+1}} = \lim_{n \to \infty} \e^{-r_n} \sum_{u \in \mathcal{N}_{r_n}} \e^{-t_n} X^{(u)}_{t_n}.
\]
Thus, computing the characteristic function of $Y$ for some $\xi \in \R$ and using that $\#\mathcal{N}_t$ is a geometric random variable with parameter $\e^{-t}$, we have
\begin{align*}
  \E\left[\e^{i \xi Y}\right] 
  &= \lim_{n \to \infty} \E\left[ \prod_{u \in \mathcal{N}_{r_n}} \e^{i\xi \e^{-r_n-t_{n}} X_{t_n}^{(u)}} \right]\\
  &= \lim_{n \to \infty} \frac{ \e^{-r_n} \E\left[ \e^{i\xi \e^{-r_n-t_{n}} X_{t_n}}  \right]}{1 - (1 - \e^{-r_n}) \E\left[ \e^{i\xi \e^{-r_n-t_{n}} X_{t_n}} \right]}\\
  &= \lim_{n \to \infty} \frac{\E\left[ \e^{i\xi \e^{-r_n-t_{n}} X_{t_n}} \right]}{\E\left( \e^{i\xi \e^{-r_n-t_{n}} X_{t_n}}  \right) + \e^{r_n}\E\left[ 1-\e^{i \xi \e^{-r_n-t_{n}} X_{t_n}} \right]}.
\end{align*}
Using the fact that $\E(X_{t_n} \e^{-t_n}) \to F_\infty(\mu_0)$ by Lemma \ref{lem:freeEnergyExists}, we have that
\[
  \E\left[ \e^{i\xi \e^{-r_n-t_{n}} X_{t_n}}  \right] = 1 + i \xi \e^{-r_n} F_\infty(\mu_0) + o(\e^{-r_n}) \quad \text{ as } n \to \infty.
\]
Therefore, this yields
\[
  \E\left[\e^{i \xi Y}\right] = \frac{1}{1 - i \xi F_\infty(\mu_0)},
\]
which is the characteristic function of the exponential distribution of parameter $F_\infty(\mu_0)^{-1}$, concluding the proof.
\end{proof}

In the rest of the section, we now focus on the (sub)critical case $F_\infty(\mu_0) = 0$. Previous lemma shows that $\lim_{t \to \infty} X_t \e^{-t} = 0$ in probability. However, an interesting problem is to obtain the precise asymptotic behavior of $X_t$ in this (sub)critical case. It is known that for discrete DR models with integer-valued measures, we have $\lim_{n \to \infty} X_n = 0$ in probability when the free energy $F_\infty=0$. Therefore, this yields the following conjecture.

\begin{conjecture}\label{c:xt->0}
If $F_\infty(\mu_0)=0$, then $X_t \to 0$ in probability as $t \to \infty$.
\end{conjecture}

Assuming the above conjecture, one might be interested in the limiting law of $X_t$ conditioned to be positive. In light of \eqref{edppetphi}, if such a limit exists, its density should satisfy $\phi'(x)= - \phi(0) \phi(x), x>0.$ This implies the following conjecture:
 
 \begin{conjecture}\label{c:xt->exp}
If $F_\infty(\mu_0)=0$, then  ${\mathcal L}( X_t \, |\, X_t>0)$ converges to an exponential distribution of parameter larger  than or equal to $1$.
\end{conjecture}

We were not able to prove these conjectures, and more generally to obtain characterizations, in terms of the law of $X_0$, of measures satisfying $F_\infty(\mu_0) = 0$ (as \eqref{eqn:cegm} in the discrete case with integer-valued measures). However, we obtained the following partial answer, which is notably similar to what is known in the general discrete DR model.

\begin{proposition}\label{p:conditionnessaire}
Let $(X_t, t \geq 0)$ be a DR process, with $\mu_0$ the law of $X_0$ satisfying $F_\infty(\mu_0)=0$. The following results hold:
\begin{enumerate}
  \item For any $t,x \geq 0$, $\P(X_t \geq x) \leq \e^{1-x}$.
  \item For any $t\geq 0$, $\E[X_t^2] < \frac{1}{2}$ and, in particular, $\E[X_t] \leq \frac{1}{\sqrt{2}}$.
  \item For any $\theta \in (0,1)$ and $t \geq 0$,
  \[
    \E\left[ \left(1 - \theta X_t - 2(1-\theta)X_t^2\right)\e^{\theta X_t} \right] \geq 0. 
  \]
  \item We have $\sup_{t\geq 0} \E[X_t\e^{X_t}] < \infty$ and, for any $t\geq0$, $\E[X_t\e^{X_t}] \leq \E[\e^{X_t}]$. 
\end{enumerate}
\end{proposition}
Note that  Point 4  gives a characterization of the subcritical case similar to \eqref{eqn:cegm} for the discrete model, but here it does not depend only on moments of $X_0$.
The main reason for this is that the sign of $\E[X_n 2^{X_n}] - \E[2^{X_n}]$ is an invariant of the dynamics for the discrete model with integer valued measures, whereas it is not the case for $\E[X_t\e^{X_t}] - \E[\e^{X_t}]$ here.
\begin{proof}
We first observe that, if $F_\infty(\mu_0)=0$, then $\E(X_t) \leq 1$ for any $t \geq 0$. 
Indeed, by Lemma \ref{lem:freeEnergyExists}, we have
\begin{equation}
  \label{eq:esperanceX_t_souscritique}
\E[X_t] = \int_t^\infty \e^{-(s-t)} \P(X_s>0) \diff s
\end{equation}
and it follows that $\E[X_t] \leq 1$.

Let $s,t \geq 0$, we consider the Yule tree representation of $X_{t+s}$. One has immediately that
\[
  X_{t+s} \geq \sum_{u \in \mathcal{N}_s} (X^{(u)}_t -s)_+ \geq \sum_{u \in \mathcal{N}_s} \1_{\{X^{(u)}_t > s+1\}},
\]
where, conditionally on $\mathcal{N}_s$, $(X^{(u)}_t, u \in \mathcal{N}_s)$ are i.i.d.\@ random variables with the same law as $X_t$.
Hence, we have that $\E(X_{t+s}) \geq \e^s \P(X_t > s+1)$ and, since $\E(X_{t+s}) \leq 1$, it proves Point 1.

It follows from Point 1.\@ that $\sup_{t\geq 0} \E[X_t^2] < \infty$.
Then, using \eqref{eq:moments_X_t} with $n=2$, we have
\[
  \E[X_t^2] 
  = \e^{t} \E[X_0^2] + \e^t \int_0^t 2 \e^{-s} \E[X_s] \left( \E[X_s] -1 \right)\diff s,
\]
and, letting $t \to \infty$,
\[
  \E[X_0^2] = 2 \int_0^\infty  \e^{-s} \E[X_s] \left( 1- \E[X_s] \right)\diff s
  \leq \frac{1}{2},
\]
using that $\E[X_s] \in [0,1]$ for any $s \geq 0$.
Finally, $\E[X_t^2] \leq \frac{1}{2}$ holds for any $t\geq 0$, because $F_\infty(\mu_t) = \e^t F_\infty(\mu_0) = 0$.
This proves Point 2.

Let $b \in \R$ be a constant whose value will be fixed later on. Consider the function
\[
  \phi(t) \coloneqq \E\left[ (1 - \theta X_t - b X_t^2)\e^{\theta X_t} \right].
\]
By It\^o's formula, we have that
\begin{align*}
  \phi'(t) 
  & = \E\left[ \left((\theta^2 + 2 b) X_t + \theta b X_t^2\right) \e^{\theta X_t} \right]\\ 
  & \relphantom{=} {} + \E\left[ \left(1 - \theta (X_t+Y_t) - b (X_t + Y_t)^2 \right)  \e^{\theta (X_t + Y_t)}\right] - \E\left[\left(1 - \theta X_t - bX_t^2\right)\e^{\theta X_t}\right],
\end{align*}
where $Y_t$ is an independent copy of $X_t$. We introduce the functions
\[
  g(t) \coloneqq \E[\e^{\theta X_t}] 
  \quad \text{and} \quad 
  h(t) \coloneqq \E[X_t \e^{\theta X_t}],
\]
we can rewrite
\[
  \phi' = (2 g - \theta - 1) \phi + 2b(h-h^2) - (g^2 - \theta g).
\]
Let $\eta \coloneqq 2 b (h - h^2) - (g^2 - \theta g)$, then, for any $0 \leq r \leq t$, the differential equation can be solved as
\[
  \phi(t) = \e^{\int_r^t (2g(s) - \theta - 1) \diff s} \left( \phi(r) + \int_r^t \eta(s) \e^{\int_r^s(2 g(u) - \theta - 1) \diff u} \diff s \right).
\]
On the one hand, we have $g \geq 1$ and therefore $g^2 -\theta g \geq 1 - \theta$ and, on the other hand, $h-h^2 \leq 1/4$.
Thus, choosing $b \coloneqq 2(1-\theta)$, we have $\eta \leq b/2 - 1 +\theta = 0$.
If $\phi(r) < 0$ for some $r \geq 0$, then 
\[
\phi(t) \leq \e^{\int_r^t (2g(s) - \theta - 1) \diff s} \phi(r) 
\xrightarrow[t \to \infty]{} -\infty,
\]
which contradicts the fact that, by Point 1., $\sup_{t \geq 0} \E[\e^{\lambda X_t}] < \infty$ for any $\lambda \in (\theta,1)$. Therefore, $\phi(r) \geq 0$ for any $r \geq 0$ and this proves Point 3.

Now we prove that necessarily $\E[X_0 \e^{X_0}] < \infty$. Indeed, if $\E[X_0 \e^{X_0}] = \infty$, then $\lim_{n\to\infty} \E [X_0 \e^{X_0} \1_{\{2 \le X_0< n\}}] \to \infty$. Choose and fix an integer $n$ such that 
\[\E [X_0 \e^{X_0} \1_{\{2 \le X_0< n\}}] > 2.\]
Write $Y_0\coloneqq X_0 \1_{\{2 \le X_0< n\}}$ and $(Y_t)_{t\ge 0}$ the solution of the SDE \eqref{eqn:DRsde} starting from $Y_0$. By coupling, a.s.\@ $X_t\ge Y_t$ for all $t\ge0$. 
Observe that 
\[
\E [(1-Y_0) \e^{Y_0}] 
= \P(Y_0=0)+\E [(1-Y_0) \e^{Y_0} \1_{\{Y_0>0\}}] 
\le 1- \frac{1}{2} \E [ Y_0 \e^{Y_0} \1_{\{Y_0>0\}}] < 0. 
\]
By continuity, for $\theta<1$ sufficiently close to $1$, we have $\E[(1- \theta Y_0) \e^{\theta Y_0}] < 0$, which, by Point 3., proves that $\lim_{t\to\infty} \e^{-t} \E[Y_t] >0$. But we have $F_\infty(\mu_0) \ge \lim_{t\to\infty} \e^{-t} \E[Y_t] >0$, so it contradicts the assumption $F_\infty(\mu_0)=0$. This proves that $\E[X_0 \e^{X_0}] < \infty.$

The same argument shows that, for any $t>0$, $\E[X_t \e^{X_t}] < \infty$. 
Therefore, noting that $(1-\theta)X_t^2 \e^{\theta X_t} \leq X_t \e^{X_t} \sup_{x \geq 0} x \e^{-x}$, we can apply dominated convergence theorem as $\theta\to 1$ in Point 3.\@ and get $\E[(1-X_t) \e^{X_t}] \geq 0$ for any $t \geq 0$.
This proves the second part of Point 4.\@ and can also be re-written as 
\[
\E[(1-X_t) \e^{X_t} \1_{\{X_t< 2\}}] \ge \E[(X_t-1) \e^{X_t} \1_{\{X_t\ge 2\}}].
\]
Note that, on the one hand, $\E[(1-X_t) \e^{X_t} \1_{\{X_t< 2\}}] \le \e^2$ and, on the other hand, $\E[(X_t-1) \e^{X_t} \1_{\{X_t\ge 2\}}]\ge \frac12 \E[X_t \e^{X_t} \1_{\{X_t\ge 2\}}]$. Hence $ \E[X_t \e^{X_t} \1_{\{X_t\ge 2\}}]\le 2 \e^2$, which yields that $ \E[X_t \e^{X_t} ] \le 4 \e^2$ and proves the first part of Point 4.
\end{proof}

In addition to Conjectures \ref{c:nombreetmasse} and  \ref{c:xt->0}, we list  some open questions below. At first, Proposition \ref{p:conditionnessaire} gives some necessary conditions on $\mu_0$ so that the free energy vanishes. Naturally we may ask 

\begin{question}
What are necessary and sufficient condition on $\mu_0$ so that $F_\infty(\mu_0)=0$?
\end{question}

It is  predicted by physicists that the infinite order phase transition  holds for a large classe of recursive model.  This motivates  our next question:
\begin{question}
Replace  the  indicator function $\ind{x>0}$ in  \eqref{eqn:mcKeanRepresentation} by a monotone increasing bounded function $f$ with $f(0)=0$, such as  $\tanh(x), 1 -\e^{-x}$ or $x \wedge 1$. For regular $f$,  there exists a unique strong solution of the  McKean--Vlasov type SDE, see \cite{Gra92}. Are the phase transitions for these models of same nature as for the DR process? 
\end{question}

Concerning our exactly solvable model, a crucial point is the preservation of the mixture of exponential distributions and Dirac masses at $0$. This was proved by using the explicit solution of \eqref{eqn:pdeDef}. However, there is no direct probabilistic proof of this result, yielding the following question.

\begin{question}
Is there any probabilistic proof that the family of mixtures of exponentials and Dirac masses at $0$  is stable under  the dynamics of  the continuous--time DR model? 
\end{question}

\paragraph{Acknowledgements} We are grateful to Bernard Derrida and Zhan Shi for stimulating discussions at the beginning of the project.


\begin{thebibliography}{10}

\bibitem{AB05}
David~J. Aldous and Antar Bandyopadhyay.
\newblock A survey of max-type recursive distributional equations.
\newblock {\em Ann. Appl. Probab.}, 15(2):1047--1110, 2005.

\bibitem{athreyaney72}
Krishna~B. Athreya and Peter~E. Ney.
\newblock {\em Branching processes}.
\newblock Springer-Verlag, New York-Heidelberg, 1972.
\newblock Die Grundlehren der mathematischen Wissenschaften, Band 196.

\bibitem{BeT13}
Quentin Berger and Fabio~Lucio Toninelli.
\newblock Hierarchical pinning model in correlated random environment.
\newblock {\em Ann. Inst. Henri Poincar\'{e} Probab. Stat.}, 49(3):781--816,
  2013.

\bibitem{CDDHLS}
Xinxing Chen, Victor Dagard, Bernard Derrida, Yueyun Hu, Mikhail Lifshits, and
  Zhan Shi.
\newblock The {D}errida–{R}etaux conjecture for recursive models.
\newblock In preparation, 2018+.

\bibitem{CDHLS17}
Xinxing Chen, Bernard Derrida, Yueyun Hu, Mikhail Lifshits, and Zhan Shi.
\newblock A hierarchical renormalization model: some properties and open
  questions.
\newblock Submitted, may 2017.

\bibitem{CEGM84}
P.~Collet, J.-P. Eckmann, V.~Glaser, and A.~Martin.
\newblock Study of the iterations of a mapping associated to a spin glass
  model.
\newblock {\em Comm. Math. Phys.}, 94(3):353--370, 1984.

\bibitem{CurienHenard}
Nicolas Curien and Olivier Hénard.
\newblock Critical parking on a random tree.
\newblock In preparation, 2018+.

\bibitem{dhv92}
B.~Derrida, V.~Hakim, and J.~Vannimenus.
\newblock Effect of disorder on two-dimensional wetting.
\newblock {\em J. Statist. Phys.}, 66(5-6):1189--1213, 1992.

\bibitem{DR14}
B.~Derrida and M.~Retaux.
\newblock The depinning transition in presence of disorder: a toy model.
\newblock {\em J. Statist. Phys}, 156:268--290, 2014.

\bibitem{DGLT09}
Bernard Derrida, Giambattista Giacomin, Hubert Lacoin, and Fabio~Lucio
  Toninelli.
\newblock Fractional moment bounds and disorder relevance for pinning models.
\newblock {\em Comm. Math. Phys.}, 287(3):867--887, 2009.

\bibitem{G07}
Giambattista Giacomin.
\newblock {\em Random polymer models}.
\newblock Imperial College Press, London, 2007.

\bibitem{G11}
Giambattista Giacomin.
\newblock {\em Disorder and critical phenomena through basic probability
  models}, volume 2025 of {\em Lecture Notes in Mathematics}.
\newblock Springer, Heidelberg, 2011.
\newblock Lecture notes from the 40th Probability Summer School held in
  Saint-Flour, 2010, \'{E}cole d'\'{E}t\'{e} de Probabilit\'{e}s de
  Saint-Flour. [Saint-Flour Probability Summer School].

\bibitem{glt2010}
Giambattista Giacomin, Hubert Lacoin, and Fabio~Lucio Toninelli.
\newblock Hierarchical pinning models, quadratic maps and quenched disorder.
\newblock {\em Probab. Theory Related Fields}, 147(1-2):185--216, 2010.

\bibitem{GP16}
C.~Goldschmidt and M.~Przyckucki.
\newblock Parking on a random tree.
\newblock To appear, 2016.

\bibitem{Gra92}
Carl Graham.
\newblock Mc{K}ean-{V}lasov {I}t\^{o}-{S}korohod equations, and nonlinear
  diffusions with discrete jump sets.
\newblock {\em Stochastic Process. Appl.}, 40(1):69--82, 1992.

\bibitem{harrisroberts2014}
Simon~C. Harris and Matthew~I. Roberts.
\newblock A strong law of large numbers for branching processes: almost sure
  spine events.
\newblock {\em Electron. Commun. Probab.}, 19:no. 28, 6, 2014.

\bibitem{HuS17}
Yueyun Hu and Zhan Shi.
\newblock The free energy in the {D}errida-{R}etaux recursive model.
\newblock {\em J. Stat. Phys.}, 172(3):718--741, 2018.

\bibitem{Lac10}
Hubert Lacoin.
\newblock Hierarchical pinning model with site disorder: disorder is marginally
  relevant.
\newblock {\em Probab. Theory Related Fields}, 148(1-2):159--175, 2010.

\bibitem{Lam18}
A.~Lambert and E.~Schertzer.
\newblock Coagulation-transport equations and the nested coalescents.
\newblock To appear, 2018.

\bibitem{McK66}
H.~P. McKean.
\newblock A class of markov processes associated with nonlinear parabolic
  equations.
\newblock {\em Proc Natl Acad Sci U S A.}, 56(6):1907--1911, 1966.

\bibitem{MoG08}
C\'{e}cile Monthus and Thomas Garel.
\newblock Critical points of quadratic renormalizations of random variables and
  phase transitions of disordered polymer models on diamond lattices.
\newblock {\em Phys. Rev. E (3)}, 77(2):021132, 16, 2008.

\bibitem{Smoluchowski1916}
M.~V. Smoluchowski.
\newblock Drei {V}ortr\"age \"uber {D}iffusion, {B}rownsche {B}ewegung und
  {K}oagulation von {K}olloidteilchen.
\newblock {\em Z. Phys.}, 17:557--585, 1916.

\bibitem{Yor1992}
Marc Yor.
\newblock {\em Some aspects of {B}rownian motion. {P}art {I}}.
\newblock Lectures in Mathematics ETH Z\"{u}rich. Birkh\"{a}user Verlag, Basel,
  1992.
\newblock Some special functionals.

\end{thebibliography}

\end{document}